\documentclass[a4paper,12pt]{amsart}

\usepackage[left=21mm,right=21mm,bottom=21mm,top=23mm]{geometry} %margins
\usepackage{newtxtext,newtxmath,mathtools}
\usepackage[utf8]{inputenc} 
\usepackage[T1]{fontenc}
\usepackage[colorlinks=true]{hyperref}
\usepackage{mathrsfs} %curly U
\usepackage{stackengine}

\usepackage[all]{xy}
\usepackage{graphicx} %includegraphics

\usepackage{tikz,tikz-cd}
\usetikzlibrary{knots}

\usepackage{adjustbox}

\usepackage{enumerate} %to enumerate with i,ii,iii etc.
\usepackage[normalem]{ulem}

\numberwithin{equation}{section}
\graphicspath{{figures/}}

\newcommand{\C}[1]{\mathcal{#1}}
\newcommand{\B}[1]{\mathbb{#1}}

\newcommand{\Tor}{\text{\rm Tor}}
\newcommand{\Ext}{\text{\rm Ext}}
\newcommand{\Span}{\text{\rm Span}}
\newcommand{\OP}{\text{\bf OP}}

\newcommand{\PRO}{\text{\bf PRO}}

\newcommand{\Alg}{\text{{\bf Alg}}}
\newcommand{\alg}{\text{{\bf alg}}}

\newcommand{\Assoc}{\text{\rm Assoc}}
\newcommand{\Comm}{\text{\rm Comm}}
\newcommand{\Simp}{\text{\rm Simp}}
\newcommand{\Leib}{\text{\rm Leib}}
\newcommand{\Lie}{\text{\rm Lie}}
\newcommand{\Mag}{\text{\rm Mag}}
\newcommand{\Sym}{\text{\rm Sym}}
\newcommand{\Braid}{\text{\rm Braid}}

\newcommand{\brd}[2]{\chi^{#1}_{#2}}

\newtheorem{theorem}{Theorem}[section]
\newtheorem{proposition}[theorem]{Proposition}

\newtheorem{lemma}[theorem]{Lemma}
\theoremstyle{definition}
\newtheorem{defn}[theorem]{Definition}

\newtheorem{remark}[theorem]{Remark}

\title{The Leibniz PROP is a crossed presimplicial algebra}

\author{Murat Can Aşkaroğullari}
\email{askarogullari19@itu.edu.tr,mcaskar@gtu.edu.tr}
\address{Istanbul Technical University, Turkey}

\author{Atabey Kaygun}
\email{kaygun@itu.edu.tr}
\address{Istanbul Technical University, Turkey}

\begin{document}
\begin{abstract}
  We prove that the Leibniz PROP is isomorphic as $\Bbbk$-linear categories (not as monoidal categories) to the symmetric crossed presimplicial algebra $\Bbbk[(\Delta^+)^{op} \B{S}]$ where $\Delta^+$ is the skeletal category of finite well-ordered sets with surjections, but the distributive law between $(\Delta^+)^{op}$ and the symmetric groups $\B{S} = \bigsqcup_{n\geq 1} S_n$ is not the standard one.
\end{abstract}

\maketitle
% \tableofcontents

\section*{Introduction}

Leibniz algebras are non-commutative analogues of Lie algebras. Like any other algebraic structure, there is an operad, or better yet a PRO or a PROP, that organizes all algebraic operations on that algebraic structure in a combinatorial and categorical framework. There are PRO(P)s for algebras of different types (magmatic, associative, commutative, Lie, Leibniz, Poisson, Jacobian, etc.). They come in different flavors (plain set, linear, piecewise linear, topological, homotopical, simplicial, differential graded, etc.), but are always considered as symmetric or braided monoidal categories.  Our key observation is that forgetting the monoidal structure of a PRO(P) yields a $\Bbbk$-linear category, and this reduction allows us to embed $\Bbbk[(\Delta^+)^{op}]$ into $\text{cat}\textbf{Leib}$. 

%This reduction would allow us to  study operadic structures via categorical algebras. 

The specific parametrizing $\Bbbk$-linear categories we are going to consider in this paper are as follows:
\begin{enumerate}[(a)]
\item $\Mag$ for magmatic algebras, i.e., algebras with a binary operation with no other condition.
\item $\Simp$ for (not necessarily unital) associative algebras.
\item $\Sym$ for graded vector spaces with compatible actions of symmetric groups $\B{S} = \bigsqcup_{n\geq 1} S_{n}$.
\item $\Braid$ for graded vector spaces with compatible actions of Artin braid groups $\B{B} = \bigsqcup_{n\geq 1} B_{n}$.
\item $\Leib$ and $\Leib^{op}$ for left and right Leibniz algebras.
\end{enumerate}
We define each of these categories by specifying explicit generators and relations, and we provide an explicit description of the distributive laws between the relevant categories on these generators. The proofs we give in this paper are highly algebro-combinatorial in nature, and in theory, should yield themselves for machine verification as in~\cite{BremnerMarkl19, Bremner_2020}, but we meticulously check them by hand aided by string diagrams.

Our main result is Theorem~\ref{thm:MainResult} in which we show that the parametrizing $\Bbbk$-linear category $\Leib$ for Leibniz algebras is isomorphic to a twisted product of $\Simp$ (which is isomorphic to the categorical algebra of the opposite of $\Delta^+$, the subcategory of epimorphisms in the simplicial category $\Delta$) and $\B{S}$. The isomorphism, however, is just an isomorphism of $\Bbbk$-linear categories, not of monoidal $\Bbbk$-linear categories. Moreover, the twisted product is determined by a distributive law between $(\Delta^+)^{op}$ and $\B{S}$, but not the standard one. Such objects defined by the distributive laws between the (pre)simplicial category and groups are usually called $\emph{crossed simplicial groups}$~\cite{crossedsimplicial}. Since we are defining everything in $\Bbbk$-linear categories we adopt the terminology crossed presimplicial algebra instead of crossed presimplicial group.

\subsection*{Plan of the article} In Section~\ref{sect:CombinatorialOperations}, we define all of the $\Bbbk$-linear categories we use by explicit generators and relations. Section~\ref{sect:DistributiveLaws} deals with the distributive laws between these $\Bbbk$-linear categories in all the different variations that we are going to need. In Section~\ref{sect:LeibnizAlgebras}, we define the $\Bbbk$-linear category for Leibniz algebras, and finally give our main result Theorem~\ref{thm:MainResult}. In Section~\ref{subsect:homologicalramifications} we discuss homological ramifications of our main result.

\subsection*{Notation and conventions} We use $\cup$ for ordinary union of sets while $\sqcup$ will denote the disjoint union. We use a base field $\Bbbk$ with no assumption on the characteristic. All unadorned tensor products are over $\Bbbk$. For each natural number $n\geq 1$, we use $S_n$ to denote the symmetric group on $n$-letters, $B_n$ to denote the Artin braid group on $n$-strands.  We use $\Delta$ for the skeletal category of finite well-ordered sets with nondecreasing set maps. The subcategories of injective and surjective maps are respectively denoted by $\Delta^-$ and $\Delta^+$. For a small $\Bbbk$-linear category $\C{C}$, the categorical algebra of $\C{C}$ is the direct sum of morphisms of $\C{C}$ where the product is given by the composition of the morphisms whenever they are composable and $0$ otherwise.

\section{Free $\B{K}$-algebras} \label{sect:CombinatorialOperations}

\subsection{$\B{K}$-bimodules and $\B{K}$-algebras}

We denote the smallest $\Bbbk$-linear PRO  by $\B{K}$ whose non-zero morphisms consist of the constant multiples of identities on each object ~\cite{adams1978infinite} ~\cite{boardman1973homotopy}. With this definition, one can see that every $\Bbbk$-linear PRO contains $\B{K}$ as a subcategory. There exists an inclusion functor $1_{\C{C}} \colon  {\B{K}} \to \C{C}$ for every $\Bbbk$-linear PRO $\C{C}$. Note that we are only interested in the underlying $\Bbbk$-linear categories of PROs omitting the monoidal structure.

One can also describe $\B{K}$ as the locally unital algebra ${\B{K}} = \Bbbk^{\oplus \B{N}}$ spanned algebraically by countably many vectors $1_n$ for every $n\in\B{N}$ subject to the condition that $1_n\cdot 1_m = \delta_{nm}$ for every $n,m\in\B{N}$.  With this definition at hand, one can now define a $\B{K}$-module as a countable collection $(V_n)_{n\in\B{N}}$ of vector spaces indexed by $\B{N}$. Then a $\B{K}$-bimodule $(V_{n,m})_{n,m\in\B{N}}$ is a collection of vector spaces doubly indexed by $\B{N}$.  A $\B{K}$-(bi)module $V$ is called \emph{locally finite} if for every $x\in V$ there are only finitely many $n,m\in\B{N}$ such that $1_n\cdot x$ and $x\cdot 1_m$ are non-zero.  A locally finite $x\in V$ is called \emph{faithful} if $\sum_{n\in\B{N}} 1_n\cdot x$ is well-defined and is equal to $x$.  A similar condition holds for bimodules.  This condition is equivalent to the fact that $V = \bigoplus_{n\in\B{N}} 1_n\cdot V$ for a module and $V = \bigoplus_{n,m\in\B{N}} 1_n\cdot V\cdot 1_m$ for a bimodule.  Note that the fact that $\B{K}$-bimodule $V = (V_{n,m})_{n,m\in\B{N}}$ is faithful is equivalent to the fact that 
\[ {\B{K}}\otimes_{\B{K}}V \cong V \cong V\otimes_{\B{K}}{\B{K}} \]

\begin{remark}
Technically, we should call $\B{K}$-(bi)modules as $\B{K}$-algebras since $\B{K}$ is a PRO. However, later on, we will need to work with monoid objects in the category of $\B{K}$-bimodules, and we would have had to refer to what we call $\B{K}$-algebras  ``algebras in the category of $\B{K}$-algebras'' which is needlessly confusing.
\end{remark}

\begin{proposition}
  The category of faithful $\B{K}$-bimodules is strictly monoidal with $\B{K}$ being the unit object and a product defined on the objects as
  \[V\otimes_{\B{K}} W = \left(\bigoplus_{m\in\B{N}} V_{n,m}\otimes W_{m,\ell}\right)_{n,\ell\in\B{N}} \]
\end{proposition}

\begin{proof}
  Assume the faithful $\B{K}$-bimodules $V$ and $W$ have bases $B_{n,m}$ and $C_{n,m}$, respectively. Then the bigraded vector space $(V\otimes_{\B{K}}W)_{n,\ell}$ has the basis $\bigsqcup_{m\in\B{N}} B_{n,m}\times C_{m,\ell}$, and thus, it is a faithful $\B{K}$-bimodule.
\end{proof}

A faithful $\B{K}$-bimodule together with a unital associative operation $\mu\colon V\otimes_{\B{K}} V\to V$, $1_V \colon \B{K} \to V$ making the following diagrams commutative is called a $\B{K}$-algebra.
\[\xymatrix{
	V\otimes_{\B{K}}V\otimes_{\B{K}}V
	\ar[rr]^{\mu\otimes V}\ar[d]_{V\otimes\mu}
	& & V\otimes_{\B{K}}V \ar[d]^{\mu}\\
	V\otimes_{\B{K}}V \ar[rr]_{\mu} 
	& & V
}\qquad
\xymatrix{ V \ar[rr]^{1_V\otimes V} \ar@{=}[rrd] 
	& & V\otimes_{\B{K}}V \ar[d]_{\mu}  
	& & \ar[ll]_{V\otimes 1_V} \ar@{=}[lld] V\\
	& & V & & 
}\]
If $V$ is a $\B{K}$-algebra such that each $V_{n,m}$ is finite dimensional, we will call $V$ a locally finite dimensional $\B{K}$-algebra. We will use $\Alg(\B{K})$ and $\alg(\B{K})$ to denote the category of $\B{K}$-algebras and locally finite dimensional $\B{K}$-algebras, respectively.

\subsection{Free $\B{K}$-algebras}

The free associative $\B{K}$-algebra $T(V)$ from a faithful $\B{K}$-bimodule $V$ is defined as
\[ T(V):={\B{K}} \oplus \bigoplus_{n \geq 1} \overbrace{V \otimes_{\B{K}} \cdots \otimes_{\B{K}} V}^{n \text {-times }}\]
where the multiplication on $T(V)$ is given by concatenation.  Note that any $\B{K}$-algebra can be written as a quotient of a free associative $\B{K}$-algebra of a faithful $\B{K}$-bimodule.

Consider the faithful $\B{K}$-bimodule  $\partial=\bigoplus_{n \geq 0}{}^{n+1}\partial^{n}$ where ${}^{n+1}\partial^{n} = \Span_{\Bbbk}(\partial_j^n \mid 0\leq j \leq n)$ and the rest of the bigraded parts are assumed to be zero. This definition implies that $1_{n+1} \cdot \partial_j^n=\partial_j^n \cdot 1_n= \partial_j^n $ and the rest of the left or right actions of $\B{K}$ yield zero. 

Next, consider the faithful $\B{K}$-bimodule  $\chi=\bigoplus_{n \geq 1}{}^n\chi^{n}$ where ${}^n\chi^{n}= \Span_{\Bbbk}(\chi_i^n \mid 0\leq i \leq n-1)$  and rest of the bigraded parts are assumed to be zero. This definition implies that $1_{n} \cdot \chi_j^n=\chi_j^n \cdot 1_n= \chi_j^n $ and the rest of the left or right actions of $\B{K}$ yield zero.

In the sections below, we are going to consider the free algebras $T(\partial)$ and $T(\chi)$ and various quotients of them.

\subsection{$\Mag$}
 We will use $\langle a_i\mid i\in I \rangle$ to denote the bilateral ideal generated by an indexed family of elements $a_i$ for $i\in I$.

Now, after \cite[Section 1.2.5]{loday:triples} and \cite{inasaridze:pseudosimplicial, Frabetti_2001}, we define \emph{the magmatic} $\B{K}$-algebra $\Mag $ as the quotient of $T(\partial)$ by the ideal $I_{\Mag }$ where
\begin{equation}\label{eq:magma}
  I_{\Mag }=\langle\partial_i^{n+1}\partial_j^n-\partial_{j+1}^{n+1}\partial_i^n \mid 0\leq i <j \leq n\text{ and } n\geq 0\rangle
\end{equation}
	
This is called a \emph{pseudo-simplicial structure} in \cite{inasaridze:pseudosimplicial}, and an \emph{almost-simplicial structure} in~\cite{Frabetti_2001}. This $\B{K}$-algebra models a $\Bbbk$-linear vector space equipped with a binary operation and no other condition.

The generator $\partial_j^n$ can be depicted as follows:
\[ \begin{tikzpicture}[xscale=0.5, yscale=0.5]
    
    \begin{knot}[
      clip width=5,
      clip radius=8pt,
      ]
      
      \node at (-1, -0.5)  {\scalebox{0.5}{$0$}} ; 
      \node at (0 , -0.5) {\scalebox{0.5}{$1$}} ; 
      \node at (3 , -0.5) {\scalebox{0.5}{$j$}} ;
      \node at (6 , -0.5) {\scalebox{0.5}{$n-1$}} ;
      \node at (7 , -0.5) {\scalebox{0.5}{$n$}} ; 
      \node at (1 , 0.6) {\scalebox{1.2}{$\cdots$}};
      \node at (-1, 1.7 )  {\scalebox{0.5}{$0$}} ; 
      \node at (0 , 1.7) {\scalebox{0.5}{$1$}} ; 
      \node at (2 , 1.7) {\scalebox{0.5}{$j$}} ; 
      \node at (4 , 1.7) {\scalebox{0.5}{$j+1$}} ; 
      \node at (6 , 1.7) {\scalebox{0.5}{$n$}} ;
      \node at (7 , 1.7) {\scalebox{0.5}{$n+1$}} ;
      \node at (5 , 0.6) {\scalebox{1.2}{$\cdots$}};
      
      % Strand at (-1, 1.2)
      \strand[thick] (-1, 0.0) to (-1, 0.27) to (-1, 0.93) to (-1, 1.2) ;
      
      % Strand at (0, 1.2)
      \strand[thick] (0, 0.0) to (0, 0.27) to (0, 0.93) to (0, 1.2) ;
      
      % Strand at (1, 1.2)
      \strand[thick] (3, 0.27) to (2, 0.93) to (2, 1.2) ;
      
      % Strand at (3, 1.2)
      \strand[thick] (3, 0.0) to (3, 0.27) to (4, 0.93) to (4, 1.2) ;
      
      % Strand at (-1, 1.2)
      \strand[thick] (6, 0.0) to (6, 0.27) to (6, 0.93) to (6, 1.2) ;
      
      % Strand at (0, 1.2)
      \strand[thick] (7, 0.0) to (7, 0.27) to (7, 0.93) to (7, 1.2) ;
      
    \end{knot} 
    
  \end{tikzpicture}
\]

The following string diagram represents the element with the smallest indices $\partial_{0}^{2}\partial_1^1-\partial_2^{2}\partial_0^1$ in the ideal $ I_{\Mag}$.
\begin{equation}
  \begin{tabular}{c c}
    \begin{tikzpicture}[xscale=0.5, yscale=0.5]		
      \begin{knot}[
        clip width=5,
        clip radius=8pt,
        ]
        
        \node at (-1, -0.5) {\scalebox{0.5}{$0$}} ; 
        \node at (2 , -0.5) {\scalebox{0.5}{$1$}} ; 
        \node at (-2, 2.9 ) {\scalebox{0.5}{$0$}} ; 
        \node at (0 , 2.9) {\scalebox{0.5}{$1$}} ; 
        \node at (1 , 2.9) {\scalebox{0.5}{$2$}} ; 
        \node at (3 , 2.9) {\scalebox{0.5}{$3$}} ; 
        \node at (5 , 1.3) {\scalebox{1}{$-$}} ; 
        
        % Strand at (-2, 2.4)
        \strand[thick] (-1, 1.47) to (-2, 2.13) to (-2, 2.4) ;         
        % Strand at (0, 2.4)
        \strand[thick] (-1, 0.0) to (-1, 0.27) to (-1, 0.93) to (-1, 1.2) to (-1, 1.47) to (0, 2.13) to (0, 2.4) ;         
        % Strand at (1, 2.4)
        \strand[thick] (2, 0.27) to (1, 0.93) to (1, 1.2) to (1, 1.47) to (1, 2.13) to (1, 2.4) ;	
        % Strand at (3, 2.4)
        \strand[thick] (2, 0.0) to (2, 0.27) to (3, 0.93) to (3, 1.2) to (3, 1.47) to (3, 2.13) to (3, 2.4) ;			
      \end{knot} 		
    \end{tikzpicture}
    & \begin{tikzpicture}[xscale=0.5, yscale=0.5]
      \begin{knot}[
        clip width=5,
        clip radius=8pt,
        ]
        
        \node at (-1, -0.5)  {\scalebox{0.5}{$0$}} ; 
        \node at (2 , -0.5) {\scalebox{0.5}{$1$}} ; 
        \node at (-2, 2.9 )  {\scalebox{0.5}{$0$}} ; 
        \node at (0 , 2.9) {\scalebox{0.5}{$1$}} ; 
        \node at (1 , 2.9) {\scalebox{0.5}{$2$}} ; 
        \node at (3 , 2.9) {\scalebox{0.5}{$3$}} ; 
        % Strand at (-2, 2.4)
        \strand[thick] (-1, 0.27) to (-2, 0.93) to (-2, 1.2) to (-2, 1.47) to (-2, 2.13) to (-2, 2.4) ;
        
        % Strand at (0, 2.4)
        \strand[thick] (-1, 0.0) to (-1, 0.27) to (0, 0.93) to (0, 1.2) to (0, 1.47) to (0, 2.13) to (0, 2.4) ;
        
        % Strand at (1, 2.4)
        \strand[thick] (2, 1.47) to (1, 2.13) to (1, 2.4) ;
        
        % Strand at (3, 2.4)
        \strand[thick] (2, 0.0) to (2, 0.27) to (2, 0.93) to (2, 1.2) to (2, 1.47) to (3, 2.13) to (3, 2.4) ;
      \end{knot} 
    \end{tikzpicture} 
  \end{tabular}
\end{equation}

\begin{proposition}\label{prop:MagStraighten}
  The magmatic $\B{K}$-algebra $\Mag$ has a $\Bbbk$-basis of trivial monomials $1_n$, and non-trivial monomials of the form $\partial^{m}_{i_{m}}\cdots\partial^n_{i_n}$ with $m\geq n$ and $i_{m}\geq \cdots\geq i_n$ with $0\leq i_j\leq j$ for all $j=n,\ldots,m$. 
\end{proposition}

\begin{proof}
  It is clear that $\Mag$ has a basis consisting of monomials in $\partial^\ell_i$'s with $0\leq i\leq \ell$. Since the defining identities in $\Mag$ are difference of two monomials of the same length, we can write a preferred basis by replacing certain submonomials (of length 2) with other submonomials (of length 2). We will refer to the operation as \emph{to straighten} from now on.  Based on the relations in $I_{\Mag }$, if $i_{j+1}<i_{j}$ occurs in a monomial $\partial^{n+\ell}_{i_{n+\ell}}\cdots\partial^n_{i_n}$ we can swap these indices using the identity $\partial_{i_{j+1}}^{j+1}\partial_{i_{j}}^{j} = \partial_{i_{j}+1}^{j+1}\partial_{i_{j+1}}^{j}$ in $\Mag$. In other words, we can write a basis of monomials in which subscripts are non-increasing $i_{n+\ell}\geq \cdots \geq i_n$.
\end{proof}

\begin{remark}\label{rk:PreferredBasis}
  Consider the normalized basis of $\Mag$ in terms of the monomials of the form 
  \begin{equation}\label{eq:MagBasis}
    \partial^m_{j_m}\cdots \partial^n_{j_n} \text{ with $j_m\geq \cdots\geq j_n$ and $0\leq j_u\leq u$ for all $u=n,\ldots,m$.}
  \end{equation}
  If we let $n=0$ then the indexing sequences of integers we used above are called \emph{parking functions}~\cite[Chapter 13]{Bona:Handbook}. We consider the cases where $m\geq n-1$. The case $m=n-1$ corresponds to the case where we have the trivial monomial $1_n$ from $\Mag$, and the case $m=n$ corresponds to the case where we have only one term $\partial^n_j$. In other words $m-n+1$ is the length of words in $\partial^n_j$'s.
\end{remark}

\subsection{$\Simp$}

We define the \emph{presimplicial} $\B{K}$-algebra $\Simp$ as the quotient $T(\partial) / I_{\Simp } $ where
\begin{equation}\label{eq:assoc}
  I_{\Simp }   
  = \langle \partial_i^{n+1} \partial_j^n-\partial_{j+1}^{n+1} \partial_i^n \mid 0 \leq i\leq j\leq n, \text{ and } n\geq 0\rangle  
\end{equation}
Note that $\Simp$ can also be defined as a quotient of the magmatic $\B{K}$-algebra $\Mag$ by the bilateral ideal $\langle \partial_i^{n+1} \partial_i^n-\partial_{i+1}^{n+1} \partial_i^n \mid 0 \leq i \leq n \rangle$. The additional constraint  models associativity on magmatic algebras. The following string diagram represents the element of the added constraints with the smallest indices $\partial_{0}^{1}\partial_0^0-\partial_1^{1}\partial_0^0$.

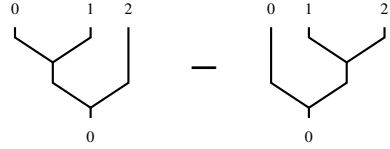
\begin{center}
	\begin{tabular}{c c}

	\begin{tikzpicture}[xscale=0.5, yscale=0.5]
		
		\begin{knot}[
			clip width=5,
			clip radius=8pt,
			]
			
			\node at (0, -0.5)  {\scalebox{0.5}{$0$}} ; 
			\node at (-2, 2.9 )  {\scalebox{0.5}{$0$}} ; 
			\node at (0 , 2.9) {\scalebox{0.5}{$1$}} ; 
			\node at (1 , 2.9) {\scalebox{0.5}{$2$}} ; 
			\node at (3 , 1.2) {\scalebox{1.5}{$-$}} ; 
			% Strand at (-2, 2.4)
			\strand[thick] (-1, 1.47) to (-2, 2.13) to (-2, 2.4) ;
			
			% Strand at (0, 2.4)
			\strand[thick] (0, 0.27) to (-1, 0.93) to (-1, 1.2) to (-1, 1.47) to (0, 2.13) to (0, 2.4) ;
			
			% Strand at (1, 2.4)
			\strand[thick] (0, 0.0) to (0, 0.27) to (1, 0.93) to (1, 1.2) to (1, 1.47) to (1, 2.13) to (1, 2.4) ;
			
		\end{knot} 
		
	\end{tikzpicture} 
&
	
	\begin{tikzpicture}[xscale=0.5, yscale=0.5]
		
		\begin{knot}[
			clip width=5,
			clip radius=8pt,
			]
			
			\node at (0, -0.5)  {\scalebox{0.5}{$0$}} ; 
			\node at (-1, 2.9 )  {\scalebox{0.5}{$0$}} ; 
			\node at (0 , 2.9) {\scalebox{0.5}{$1$}} ; 
			\node at (2 , 2.9) {\scalebox{0.5}{$2$}} ; 
			% Strand at (-1, 2.4)
			\strand[thick] (0, 0.27) to (-1, 0.93) to (-1, 1.2) to (-1, 1.47) to (-1, 2.13) to (-1, 2.4) ;
			
			% Strand at (0, 2.4)
			\strand[thick] (1, 1.47) to (0, 2.13) to (0, 2.4) ;
			
			% Strand at (2, 2.4)
			\strand[thick] (0, 0.0) to (0, 0.27) to (1, 0.93) to (1, 1.2) to (1, 1.47) to (2, 2.13) to (2, 2.4) ;
			
		\end{knot} 
		
	\end{tikzpicture} 
		\end{tabular}
\end{center}

\begin{proposition}\label{prop:SimpStraighten}
  The presimplicial $\B{K}$-algebra $\Simp$ has a $\Bbbk$-basis consisting of monomials of the form 
  \begin{equation}\label{eq:BasisSimp}
    \partial^{m}_{i_m}\cdots\partial^n_{i_n} \text{ with $i_m>\cdots>i_n$ and $0\leq i_j\leq j$}
  \end{equation}
  for every $j=n,\ldots,m$ where $m\geq n-1$.
\end{proposition}
\begin{proof}
  Recall from Remark~\ref{rk:PreferredBasis} that we have a preferred basis for $\Mag$ of the form~\eqref{eq:MagBasis}.
  Now, in $\Simp$ we replace monomials of the form $\partial^{j+1}_{\ell}\partial^j_\ell$ with $\partial^{j+1}_{\ell+1}\partial^j_\ell$. Then the fact that in the basis monomials we must have $i_m>\cdots>i_n$ easily follows. 
\end{proof}

\begin{proposition}\label{prop:simpdelta}
  $\Simp$ is isomorphic to the categorical algebra of $(\Delta^+)^{op}$.
\end{proposition}

\begin{proof}
  The maps $\sigma_j^n: \{0,\ldots ,n+1\} \to \{0,\ldots , n\}$ are defined as the order-preserving surjections that send both $j$ and $j+1$ to $j$, for $0\leq j \leq n$ and $n\geq 0$. These surjections are subject to relations $\sigma^{n}_{j} \circ \sigma^{n+1}_{i}=\sigma^{n}_{i} \circ \sigma^{n+1}_{j+1}$ for $0\leq i \leq j \leq n$ and for all $n\geq 0$.  So any morphism $\phi\colon\{0,\ldots,m\} \to \{0,\ldots,n\}$ in $\Delta^+$ can be uniquely decomposed as 
    \[ \phi= \sigma^{n}_{i_n}\circ \cdots \circ \sigma^{m-1}_{i_{m-1}} \]
    where $i_n<\cdots<i_{m-1}$. Notice that the monomials of the form~\eqref{eq:BasisSimp} are in bijection with these unique compositions in the reverse order. This finishes the proof.
\end{proof}

\subsection{$\Braid$}\label{sec:braid}
The relations below that define the $\B{K}$-algebra $\Braid $ come from the braid groups as defined by Artin~\cite{Artin}. We rewrite their presentations in our context.

We define the \emph{braid} $\B{K}$-algebra $\Braid $ as the quotient of $T(\chi)$ by the ideal $I_{\Braid }$ where $$I_{\Braid }=\langle\chi_i^n \chi_j^n-\chi_j^n \chi_i^n,\ \chi_i^n \chi_{i+1}^n \chi_i^n-\chi_{i+1}^n \chi_i^n \chi_{i+1}^n \mid 2\leq |i-j|,\ 2\leq n,\ 0\leq i \leq n-2 \rangle$$

The relations above describe the braid groups $B_{n+1}$ on $n+1$ strands for $n\geq 1$. The generator $\brd{n}{j}$ can be depicted as follows:
\begin{equation}
	\begin{tikzpicture}[xscale=0.5, yscale=0.5]
		
		\begin{knot}[
			clip width=5,
			clip radius=8pt,
			]
			
			\node at (-1, -0.5)  {\scalebox{0.5}{$0$}} ; 
			\node at (0 , -0.5) {\scalebox{0.5}{$1$}} ; 
			\node at (2 , -0.5) {\scalebox{0.5}{$j$}} ;
			\node at (3 , -0.5) {\scalebox{0.5}{$j+1$}} ;
			\node at (5 , -0.5) {\scalebox{0.5}{$n-1$}} ;
			\node at (6 , -0.5) {\scalebox{0.5}{$n$}} ; 
			\node at (1 , 0.6) {\scalebox{1.2}{$\cdots$}};
			\node at (-1, 1.7 )  {\scalebox{0.5}{$0$}} ; 
			\node at (0 , 1.7) {\scalebox{0.5}{$1$}} ; 
			\node at (2 , 1.7) {\scalebox{0.5}{$j$}} ; 
			\node at (3 , 1.7) {\scalebox{0.5}{$j+1$}} ; 
			\node at (5 , 1.7) {\scalebox{0.5}{$n-1$}} ;
			\node at (6 , 1.7) {\scalebox{0.5}{$n$}} ;
			\node at (4 , 0.6) {\scalebox{1.2}{$\cdots$}};
			
			% Strand at (-1, 1.2)
			\strand[thick] (-1, 0.0) to (-1, 0.27) to (-1, 0.93) to (-1, 1.2) ;
			
			% Strand at (0, 1.2)
			\strand[thick] (0, 0.0) to (0, 0.27) to (0, 0.93) to (0, 1.2) ;
			
			% Strand at (1, 1.2)
			\strand[thick] (2, 0.0) to (2, 0.27) to (3, 0.93) to (3, 1.2) ;
			
			% Strand at (3, 1.2)
			\strand[thick] (3, 0.0) to (3, 0.27) to (2, 0.93) to (2, 1.2) ;
			
			% Strand at (-1, 1.2)
			\strand[thick] (5, 0.0) to (5, 1.2) ;
			
			% Strand at (0, 1.2)
			\strand[thick] (6, 0.0) to (6, 1.2) ;
			
		\end{knot} 
		
	\end{tikzpicture}
\end{equation}	
Then the two defining relations for $i<j$ in $\Braid$ can be depicted  as follows:
\[\begin{tabular}{c c }
    \begin{tikzpicture}[xscale=0.5, yscale=0.5]
      
      \begin{knot}[
        clip width=5,
        clip radius=8pt,
        ]
        
        \node at (0, -0.5)  {\scalebox{0.5}{$i$}} ; 
        \node at (1 , -0.5) {\scalebox{0.5}{$i+1$}} ; 
        \node at (3 , -0.5) {\scalebox{0.5}{$j$}} ; 
        \node at (4 , -0.5) {\scalebox{0.5}{$j+1$}} ; 
        \node at (2.1 , 1.2) {\scalebox{1}{$\cdots$}} ;
        \node at (0, 2.9 )  {\scalebox{0.5}{$i$}} ; 
        \node at (1 , 2.9) {\scalebox{0.5}{$i+1$}} ; 
        \node at (3 , 2.9) {\scalebox{0.5}{$j$}} ; 
        \node at (4 , 2.9) {\scalebox{0.5}{$j+1$}} ; 
        % Strand at (1, 2.4)
        \strand[thick] (0, 0.0) to (0, 0.27) to (1, 0.93) to (1, 1.2) to (1, 1.47) to (1, 2.13) to (1, 2.4) ;
        
        % Strand at (0, 2.4)
        \strand[thick] (1, 0.0) to (1, 0.27) to (0, 0.93) to (0, 1.2) to (0, 1.47) to (0, 2.13) to (0, 2.4) ;

        % Strand at (3, 0.0)
        \strand[thick] (3, 0.0) to (3, 0.27) to (3, 0.93) to (3, 1.2) to (3, 1.47) to (4, 2.13) to (4, 2.4);
        
        % Strand at (3, 2.4)
        \strand[thick] (4, 0.0) to (4, 0.27) to (4, 0.93) to (4, 1.2) to (4, 1.47) to (3, 2.13) to (3, 2.4);

      \end{knot}       
    \end{tikzpicture}
    & \begin{tikzpicture}[xscale=0.5, yscale=0.5]
      
      \begin{knot}[
        clip width=5,
        clip radius=8pt,
        ]
        
        \node at (0, -0.5)  {\scalebox{0.5}{$i$}} ; 
        \node at (1 , -0.5) {\scalebox{0.5}{$i+1$}} ; 
        \node at (3 , -0.5) {\scalebox{0.5}{$j$}} ; 
        \node at (4 , -0.5) {\scalebox{0.5}{$j+1$}} ; 
        \node at (2.1 , 1.2) {\scalebox{1}{$\cdots$}} ;
        \node at (0, 2.9 )  {\scalebox{0.5}{$i$}} ; 
        \node at (1 , 2.9) {\scalebox{0.5}{$i+1$}} ; 
        \node at (3 , 2.9) {\scalebox{0.5}{$j$}} ; 
        \node at (4 , 2.9) {\scalebox{0.5}{$j+1$}} ; 
        \node at (-2 , 1.2) {\scalebox{1}{$=$}} ; 
        % Strand at (1, 2.4)
        \strand[thick] (0, 0.0) to (0, 0.27) to (0, 0.93) to (0, 1.2) to (0, 1.47) to (1, 2.13) to (1, 2.4) ;
        
        % Strand at (0, 2.4)
        \strand[thick] (1, 0.0) to (1, 0.27) to (1, 0.93) to (1, 1.2) to (1, 1.47) to (0, 2.13) to (0, 2.4) ;
        
        % Strand at (3, 2.4)
        \strand[thick] (3, 0.0) to (3, 0.27) to (4, 0.93) to (4, 1.2) to (4, 1.47) to (4, 2.13) to (4, 2.4) ;
        
        % Strand at (3, 2.4)
        \strand[thick] (4, 0.0) to (4, 0.27) to (3, 0.93) to (3, 1.2) to (3, 1.47) to (3, 2.13) to (3, 2.4) ;
        
      \end{knot}       
    \end{tikzpicture}        
  \end{tabular}
  \qquad\qquad
  \begin{tabular}{c c }
    
    \begin{tikzpicture}[xscale=0.5, yscale=0.5]
      
      \begin{knot}[
        clip width=5,
        clip radius=8pt,
        ]
        
        \node at (0, -0.5)  {\scalebox{0.5}{$i$}} ; 
        \node at (1 , -0.5) {\scalebox{0.5}{$i+1$}} ; 
        \node at (2 , -0.5) {\scalebox{0.5}{$i+2$}} ; 
        \node at (0, 4.1 )  {\scalebox{0.5}{$i$}} ; 
        \node at (1 , 4.1) {\scalebox{0.5}{$i+1$}} ; 
        \node at (2 , 4.1) {\scalebox{0.5}{$i+2$}};
        \node at (4 , 2.3) {\scalebox{1}{$=$}}
        ; 
        % Strand at (2, 3.6)
        \strand[thick] (0, 0.0) to (0, 0.27) to (1, 0.93) to (1, 1.2) to (1, 1.47) to (2, 2.13) to (2, 2.4) to (2, 2.67) to (2, 3.33) to (2, 3.6) ;
        
        % Strand at (1, 3.6)
        \strand[thick] (1, 0.0) to (1, 0.27) to (0, 0.93) to (0, 1.2) to (0, 1.47) to (0, 2.13) to (0, 2.4) to (0, 2.67) to (1, 3.33) to (1, 3.6) ;
        
        % Strand at (0, 3.6)
        \strand[thick] (2, 0.0) to (2, 0.27) to (2, 0.93) to (2, 1.2) to (2, 1.47) to (1, 2.13) to (1, 2.4) to (1, 2.67) to (0, 3.33) to (0, 3.6) ;
        
      \end{knot} 
      
    \end{tikzpicture} 
    & \begin{tikzpicture}[xscale=0.5, yscale=0.5]
        
        \begin{knot}[
          clip width=5,
          clip radius=8pt,
          ]
          
          \node at (0, -0.5)  {\scalebox{0.5}{$i$}} ; 
          \node at (1 , -0.5) {\scalebox{0.5}{$i+1$}} ; 
          \node at (2 , -0.5) {\scalebox{0.5}{$i+2$}} ; 
          \node at (0, 4.1 )  {\scalebox{0.5}{$i$}} ; 
          \node at (1 , 4.1) {\scalebox{0.5}{$i+1$}} ; 
          \node at (2 , 4.1) {\scalebox{0.5}{$i+2$}} ; 
          % Strand at (2, 3.6)
          \strand[thick] (0, 0.0) to (0, 0.27) to (0, 0.93) to (0, 1.2) to (0, 1.47) to (1, 2.13) to (1, 2.4) to (1, 2.67) to (2, 3.33) to (2, 3.6) ;
          
          % Strand at (1, 3.6)
          \strand[thick] (1, 0.0) to (1, 0.27) to (2, 0.93) to (2, 1.2) to (2, 1.47) to (2, 2.13) to (2, 2.4) to (2, 2.67) to (1, 3.33) to (1, 3.6) ;
          
          % Strand at (0, 3.6)
          \strand[thick] (2, 0.0) to (2, 0.27) to (1, 0.93) to (1, 1.2) to (1, 1.47) to (0, 2.13) to (0, 2.4) to (0, 2.67) to (0, 3.33) to (0, 3.6) ;
          
        \end{knot}       
      \end{tikzpicture}
  \end{tabular}
\]

\subsection{$\Sym$}
Now we consider the Coxeter groups corresponding to the Artin braid groups defined in Section~\ref{sec:braid} and write their presentations.

We define the \emph{symmetric} $\B{K}$-algebra $\Sym $ as the quotient $T(\chi) / I_{\Sym }$ where
\[ I_{\Sym }=  I_{\Braid } + \langle \chi_i^{n} \chi_i^n-1_n \mid 0 \leq i \leq n-1, \text{ and } 1\leq n \rangle \]
One can equivalently define $\Sym$ as the quotient $\Braid  / \langle \chi_i^{n} \chi_i^n-1_n \mid 0 \leq i \leq n-1, \text{ and } 1\leq n \rangle$. 

The string diagram for the extra relation defining $\Sym$ is depicted as follows:
\[\begin{tabular}{c c}
    \begin{tikzpicture}[xscale=0.5, yscale=0.5]
      
      \begin{knot}[
        clip width=5,
        clip radius=8pt,
        ]
        
        \node at (0, -0.5)  {\scalebox{0.5}{$i$}} ; 
        \node at (1 , -0.5) {\scalebox{0.5}{$i+1$}} ; 
        \node at (0, 2.9 )  {\scalebox{0.5}{$i$}} ; 
        \node at (1 , 2.9) {\scalebox{0.5}{$i+1$}} ; 
        \node at (3 , 1.2) {\scalebox{1}{$=$}} ; 
        % Strand at (0, 2.4)
        \strand[thick] (0, 0.0) to (0, 0.27) to (1, 0.93) to (1, 1.2) to (1, 1.47) to (0, 2.13) to (0, 2.4) ;
        
        % Strand at (1, 2.4)
        \strand[thick] (1, 0.0) to (1, 0.27) to (0, 0.93) to (0, 1.2) to (0, 1.47) to (1, 2.13) to (1, 2.4) ;
         \flipcrossings{2}
      \end{knot} 
      
    \end{tikzpicture}
    & \begin{tikzpicture}[xscale=0.5, yscale=0.5]
      
      \begin{knot}[
        clip width=5,
        clip radius=8pt,
        ]
        
        \node at (0, -0.5)  {\scalebox{0.5}{$i$}} ; 
        \node at (1 , -0.5) {\scalebox{0.5}{$i+1$}} ; 
        \node at (0, 2.9 )  {\scalebox{0.5}{$i$}} ; 
        \node at (1 , 2.9) {\scalebox{0.5}{$i+1$}} ; 
        % Strand at (0, 2.4)
        \strand[thick] (0, 0.0) to (0, 2.4) ;
        
        % Strand at (1, 2.4)
        \strand[thick] (1, 0.0) to (1, 2.4) ;
        
      \end{knot} 
      
    \end{tikzpicture} 
  \end{tabular}
\]

%Note that if we fix the superscript $n$, this is equivalent to Coxeter group representation of $\Bbbk[S_{n+1}]$ where $\chi_i^{n}$ corresponds to the transposition $(i+1 \, i+2)$.

\section{Distributive laws between $\B{K}$-algebras}\label{sect:DistributiveLaws}

In this Section, we are going to define a distributive law on a free $\B{K}$-algebra, and then extend it to the other cases we are interested in. The distributive laws we are going to define can be considered as distributive laws of underlying $\Bbbk$-linear categories of PRO(P)s associated to (non)symmetric operads. There are examples of such distributive laws in the literature~\cite{Lack_04,BremnerMarkl19,Bremner_2020}, but our cases are different and we build them from the ground up using explicit generators and relations.

\subsection{Transpositions}

For a $\B{K}$-algebra $\C{B}$ and $\B{K}$-bimodule $\C{A}$, a morphism of $\B{K}$-bimodules $\omega\colon  \C{B} \otimes_{\B{K}} \C{A} \rightarrow \C{A} \otimes_{\B{K}} \C{B}$ is called a right transposition for $\C{B}$ if the following diagram commutes:
\begin{equation}\label{eq:RightTransposition}
  \begin{tikzcd}
    \C{B} \otimes_{\B{K}} \C{B} \otimes_{\B{K}} \C{A}
    \arrow[rr, "{{\C{B}} \otimes \omega}"]
    \arrow[d, "{\mu_{\C{B}}\otimes{\C{A}}}"']
    & & \C{B} \otimes_{\B{K}} \C{A} \otimes_{\B{K}} \C{B}
    \arrow[rr, "{\omega \otimes  {\C{B}} }"]
    & & \C{A} \otimes_{\B{K}} \C{B} \otimes_{\B{K}} \C{B}
    \arrow[d, "{{\C{A}}\otimes\mu_{\C{B}}}"] \\
    \C{B} \otimes_{\B{K}} \C{A} \arrow[rrrr, "\omega"']
    & & & & \C{A} \otimes_{\B{K}} \C{B}
  \end{tikzcd}
\end{equation}

A right transposition is called \emph{unital} if it satisfies the unitality condition:
\begin{equation}\label{eq:UnitalTransposition}
  \begin{tikzcd}
    \C{B} \otimes_{\B{K}} \C{A} \arrow[d, "{\omega}"']  
    & &  
    \C{B} 
    \arrow[ll, "{{\C{B}} \otimes 1_{\C{A}}}"']  
    \arrow[lld, "{1_{\C{A}} \otimes {\C{B}}}"]  
    \\
    \C{A} \otimes_{\B{K}} \C{B}
  \end{tikzcd}
\end{equation}

One can also define (unital) left transpositions similarly.

\subsection{Distributive laws}

Let $\C{B}$ and $\C{A}$ be two $\B{K}$-algebras. A morphism of $\B{K}$-bimodules $\omega\colon\C{B} \otimes_{\B{K}} \C{A} \rightarrow \C{A} \otimes_{\B{K}} \C{B}$ is called a distributive law if it is a unital left transposition for $\C{A}$ and it is a unital right transposition for $\C{B}$.  A distributive law $\omega$ is called \emph{balanced} if $\omega$ is invertible and $\omega^{-1}$ is also a distributive law.

\begin{proposition}\label{prop:algebradistributive}
  A distributive law for two $\B{K}$-algebras  $\omega\colon  \C{B} \otimes_{\B{K}} \C{A} \rightarrow \C{A} \otimes_{\B{K}} \C{B}$ induces a unital associative algebra structure on the $\B{K}$-bimodule $ \C{A} \otimes_{\B{K}} \C{B}$ with the multiplication $(\mu_{\C{A}}\otimes_{\B{K}}\mu_{\C{B}})(\C{A}\otimes_{\B{K}}\omega\otimes_{\B{K}} \C{B})$. We denote the resulting $\B{K}$-algebra by $ \C{A} \otimes_{\omega} \C{B}$.
\end{proposition}

\begin{proof}
  The functors $(\ \cdot\ )\otimes_{\B{K}}\C{A}$ and $(\ \cdot\ )\otimes_{\B{K}}\C{B}$ are monads on the category of $\B{K}$-modules. Then $(\ \cdot\ )\otimes_{\B{K}} \C{A}\otimes_{\B{K}}\C{B}$ is the composite endofunctor which is a monad by \cite[Sect.1.]{Beck:Distributive}.
\end{proof}

The string diagram of this twisted product algebra structure is as follows:
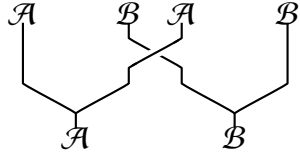
\[\begin{tikzpicture} [xscale=0.35, yscale=0.2]	
    \begin{knot}[
      clip width=5,
      clip radius=8pt,
      ]
      \node (A) at (0,7.6)  {$\mathcal{B}$};
      \node (A) at (-10,7.6)  {$\mathcal{A}$};
      \node (B) at (-6,7.6)  {$\mathcal{B}$};
      \node (C) at (-4,7.6)  {$\mathcal{A}$};
      \node (D) at (-2,-0.6)  {$\mathcal{B}$};
      \node (E) at (-8,-0.6)  {$\mathcal{A}$};
      
      \strand[thick] (-2,0)
      to [out=up, in=down, looseness=0.1] (-2,1)
      to [out=up, in=down, looseness=0.1] (0,3)
      to [out=up, in=down, looseness=0.1] (0,7);
      
      \strand[thick] (-2,1)
      to [out=up, in=down, looseness=0.1] (-4,3)
      to [out=up, in=down, looseness=0.1] (-4,4)
      to [out=up, in=down, looseness=0.1] (-6,6)
      to [out=up, in=down, looseness=0.1] (-6,7);
      
      \strand[thick] (-8,0)
      to [out=up, in=down, looseness=0.1] (-8,1)
      to [out=up, in=down, looseness=0.1] (-10,3)
      to [out=up, in=down, looseness=0.1] (-10,7);
      
      \strand[thick] (-8,1)
      to [out=up, in=down, looseness=0.1] (-6,3)
      to [out=up, in=down, looseness=0.1] (-6,4)
      to [out=up, in=down, looseness=0.1] (-4,6)
      to [out=up, in=down, looseness=0.1] (-4,7);

      \flipcrossings{1}
    \end{knot}
  \end{tikzpicture}
\]
	
\begin{proposition}\label{prop:unique}
  Let $\C{C}$ be  a $\B{K}$-algebra and $\C{A}$, and $\C{B}$ be two $\B{K}$-subalgebras. If  $\C{C}$ is isomorphic to $\C{A}\otimes_{\B{K}}\C{B}$ as a $\B{K}$-bimodule then there is a unique distributive law $\omega\colon \C{B}\otimes_{\B{K}}\C{A} \rightarrow \C{A}\otimes_{\B{K}}\C{B} $ which makes $\C{A}\otimes_{\omega}\C{B}$ isomorphic to $\C{C}$ as a $\B{K}$-algebra.
\end{proposition}

\begin{proof}
  The distributive law $\omega\colon\C{B}\otimes_{\B{K}}\C{A}\to \C{A}\otimes_{\B{K}}\C{B}$ comes from the multiplication in $\C{C}$ and the fact that $\C{C}$ is isomorphic to $\C{A}\otimes_{\B{K}}\C{B}$ as $\B{K}$-bimodules. The fact that $\omega$ is a unital distributive law (commutativity of the Diagram~\eqref{eq:UnitalTransposition}) comes from the fact that $1_\C{A}=1_\C{B}$, and the fact that $\omega$ is a right transposition (i.e. that the Diagram~\eqref{eq:RightTransposition} commutes) can be described as $(bb')a = b(b'a)$, while the fact that $\omega$ is a left transposition (the dual diagram of
  \eqref{eq:RightTransposition}) can be described as $b(aa')=(ba)a'$ for every $b,b'\in\C{B}$ and $a,a'\in \C{A}$.
\end{proof}

\subsection{Distributive laws on free algebras}

\begin{proposition}\label{420}
  Any morphism of faithful $\B{K}$-bimodules $\omega\colon  W \otimes_{\B{K}} V \rightarrow V \otimes_{\B{K}} W$ can be extended into a unique morphism of the type
  \[ \omega_n^m\colon  W^{\otimes_{\B{K}}m} \otimes_{\B{K}} V^{\otimes_{\B{K}}n} \rightarrow V^{\otimes_{\B{K}}n} \otimes_{\B{K}} W^{\otimes_{\B{K}}m} \] 
  for any $n,m \geq 1$ by applying $\omega$ successively to the  $W^{\otimes_{\B{K}}m} \otimes_{\B{K}} V^{\otimes_{\B{K}}n}$ while keeping other components the same. In other words, the order of application of $\omega$'s yield the same unique morphism.
\end{proposition}

\begin{proof}
  Extension exists and the difference in the order of application can be boiled down to the commutativity of the following diagram for the case $n,m \geq 2$ which is obvious.
  \[ \xymatrix{
      W\otimes_{\B{K}} V\otimes_{\B{K}} W\otimes_{\B{K}} V
      \ar[rr]^{\omega\otimes W\otimes V}
      \ar[d]_{W\otimes V\otimes\omega}
      & & V\otimes_{\B{K}} W\otimes_{\B{K}} W\otimes_{\B{K}} V
      \ar[d]^{V\otimes W\otimes\omega}\\
      W\otimes_{\B{K}} V\otimes_{\B{K}} V\otimes_{\B{K}} W
      \ar[rr]_{\omega\otimes V\otimes W}
      & & V\otimes_{\B{K}} W\otimes_{\B{K}} V\otimes_{\B{K}} W
    }\]
\end{proof}

We can take the direct sum of such extensions and write the ultimate version $\omega_*^* \colon T(W) \otimes_{\B{K}} T(V) \rightarrow T(V) \otimes_{\B{K}} T(W) $.
\begin{proposition}
  Extension of any morphism of faithful $\B{K}$-bimodules $\omega\colon  W \otimes_{\B{K}} V \rightarrow V \otimes_{\B{K}} W$ to their free $\B{K}$-algebras $\omega_*^* \colon T(W) \otimes_{\B{K}} T(V) \rightarrow T(V) \otimes_{\B{K}} T(W) $ is a unital distributive law.
\end{proposition}

\subsection{The fundamental (or first) distributive law}

\begin{lemma}\label{lem:CanonicalDistributiveLaw}
  The $\B{K}$-bimodule morphism $\zeta\colon  T(\partial) \otimes_{\B{K}} T(\chi)\to T(\chi) \otimes_{\B{K}} T(\partial)$ defined on the generators of $T(\partial) \otimes_{\B{K}} T(\chi)$ as
  \begin{equation}\label{eq:CanonicalDistributiveLaw}
    \zeta\left(\partial_i^n \otimes \chi_j^n\right)
    = \begin{cases}
      \chi_{j+1}^{n+1} \otimes \partial_i^n
      & \text { if } i<j \\
      \chi_{i+1}^{n+1} \chi_i^{n+1} \otimes \partial_{i+1}^n
      & \text { if } i=j \\
      \chi_{i-1}^{n+1} \chi_{i}^{n+1} \otimes \partial_{i-1}^n
      & \text { if } i=j+1 \\ \chi_j^{n+1} \otimes \partial_i^n & \text { if } i>j+1
    \end{cases}
  \end{equation}
  defines a distributive law of free $\B{K}$-algebras of the form $\zeta\colon T(\partial)\otimes_{\B{K}} T(\chi)\to T(\chi)\otimes_{\B{K}}T(\partial)$.
\end{lemma}

The representations of this distributive law in terms of string diagrams  for the second and third cases are shown below:
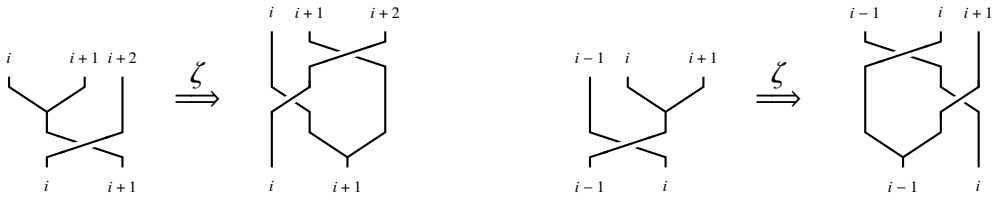
\[\begin{tabular}{c c}
    \begin{tikzpicture}[xscale=0.5, yscale=0.5]
      
      \begin{knot}[
        clip width=5,
        clip radius=8pt,
        ]
        
        \node at (0, -0.5)  {\scalebox{0.5}{$i$}} ; 
        \node at (2 , -0.5) {\scalebox{0.5}{$i+1$}} ; 
        \node at (-1, 2.9 )  {\scalebox{0.5}{$i$}} ; 
        \node at (1 , 2.9) {\scalebox{0.5}{$i+1$}} ; 
        \node at (2 , 2.9) {\scalebox{0.5}{$i+2$}} ; 
        \node at (4 , 1.8) {\scalebox{1}{$\Longrightarrow$}} ;
        \node at (4 , 2.45) {\scalebox{1}{$\zeta$}} ; 
        % Strand at (2, 2.4)
        \strand[thick] (0, 0.0) to (0, 0.27) to (2, 0.93) to (2, 1.2) to (2, 1.47) to (2, 2.13) to (2, 2.4) ;
        
        % Strand at (-1, 2.4)
        \strand[thick] (0, 1.47) to (-1, 2.13) to (-1, 2.4) ;
        
        % Strand at (1, 2.4)
        \strand[thick] (2, 0.0) to (2, 0.27) to (0, 0.93) to (0, 1.2) to (0, 1.47) to (1, 2.13) to (1, 2.4) ;
        
      \end{knot} 
      
    \end{tikzpicture}
    & \begin{tikzpicture}[xscale=0.5, yscale=0.5]
      
      \begin{knot}[
        clip width=5,
        clip radius=8pt,
        ]
        
        \node at (-1, -0.5)  {\scalebox{0.5}{$i$}} ; 
        \node at (1 , -0.5) {\scalebox{0.5}{$i+1$}} ; 
        \node at (-1, 4.1 )  {\scalebox{0.5}{$i$}} ; 
        \node at (0 , 4.1) {\scalebox{0.5}{$i+1$}} ; 
        \node at (2 , 4.1) {\scalebox{0.5}{$i+2$}} ; 
        
        % Strand at (2, 3.6)
        \strand[thick] (-1, 0.0) to (-1, 0.27) to (-1, 0.93) to (-1, 1.2) to (-1, 1.47) to (0, 2.13) to (0, 2.4) to (0, 2.67) to (2, 3.33) to (2, 3.6) ;
        
        % Strand at (-1, 3.6)
        \strand[thick] (1, 0.27) to (0, 0.93) to (0, 1.2) to (0, 1.47) to (-1, 2.13) to (-1, 2.4) to (-1, 2.67) to (-1, 3.33) to (-1, 3.6) ;
        
        % Strand at (0, 3.6)
        \strand[thick] (1, 0.0) to (1, 0.27) to (2, 0.93) to (2, 1.2) to (2, 1.47) to (2, 2.13) to (2, 2.4) to (2, 2.67) to (0, 3.33) to (0, 3.6) ;
        
      \end{knot} 
      
    \end{tikzpicture} 
  \end{tabular}
  \qquad\qquad
  \begin{tabular}{c c}
    \begin{tikzpicture}[xscale=0.5, yscale=0.5]
      
      \begin{knot}[
        clip width=5,
        clip radius=8pt,
        ]
        
        \node at (-1, -0.5)  {\scalebox{0.5}{$i-1$}} ; 
        \node at (1 , -0.5) {\scalebox{0.5}{$i$}} ; 
        \node at (-1, 2.9 )  {\scalebox{0.5}{$i-1$}} ; 
        \node at (0 , 2.9) {\scalebox{0.5}{$i$}} ; 
        \node at (2 , 2.9) {\scalebox{0.5}{$i+1$}} ;
        \node at (4 , 1.8) {\scalebox{1}{$\Longrightarrow$}} ;
        \node at (4 , 2.45) {\scalebox{1}{$\zeta$}} ;  
        % Strand at (0, 2.4)
        \strand[thick] (1, 1.47) to (0, 2.13) to (0, 2.4) ;
        
        % Strand at (2, 2.4)
        \strand[thick] (-1, 0.0) to (-1, 0.27) to (1, 0.93) to (1, 1.2) to (1, 1.47) to (2, 2.13) to (2, 2.4) ;
        
        % Strand at (-1, 2.4)
        \strand[thick] (1, 0.0) to (1, 0.27) to (-1, 0.93) to (-1, 1.2) to (-1, 1.47) to (-1, 2.13) to (-1, 2.4) ;
        
      \end{knot} 
      
    \end{tikzpicture}
    & \begin{tikzpicture}[xscale=0.5, yscale=0.5]
      
      \begin{knot}[
        clip width=5,
        clip radius=8pt,
        ]
        
        \node at (0, -0.5)  {\scalebox{0.5}{$i-1$}} ; 
        \node at (2 , -0.5) {\scalebox{0.5}{$i$}} ; 
        \node at (-1, 4.1 )  {\scalebox{0.5}{$i-1$}} ; 
        \node at (1 , 4.1) {\scalebox{0.5}{$i$}} ; 
        \node at (2 , 4.1) {\scalebox{0.5}{$i+1$}} ; 
        % Strand at (1, 3.6)
        \strand[thick] (0, 0.27) to (-1, 0.93) to (-1, 1.2) to (-1, 1.47) to (-1, 2.13) to (-1, 2.4) to (-1, 2.67) to (1, 3.33) to (1, 3.6) ;
        
        % Strand at (2, 3.6)
        \strand[thick] (0, 0.0) to (0, 0.27) to (1, 0.93) to (1, 1.2) to (1, 1.47) to (2, 2.13) to (2, 2.4) to (2, 2.67) to (2, 3.33) to (2, 3.6) ;
        
        % Strand at (-1, 3.6)
        \strand[thick] (2, 0.0) to (2, 0.27) to (2, 0.93) to (2, 1.2) to (2, 1.47) to (1, 2.13) to (1, 2.4) to (1, 2.67) to (-1, 3.33) to (-1, 3.6) ;
        
      \end{knot} 
      
    \end{tikzpicture} 
  \end{tabular}
\]

\begin{remark}
  In the remaining sections below, we are going to show that the distributive law $\zeta\colon T(\partial)\otimes_{\B{K}} T(\chi)\to T(\chi)\otimes_{\B{K}}T(\partial)$ induces  distributive laws of $\B{K}$-algebras 
  \begin{align}
    \Mag  \otimes_{\B{K}} \Braid\ & \longrightarrow \Braid  \otimes_{\B{K}} \Mag \\
    \Simp  \otimes_{\B{K}} \Braid & \longrightarrow \Braid  \otimes_{\B{K}} \Simp \\
    \Mag  \otimes_{\B{K}} \Sym & \longrightarrow \Sym  \otimes_{\B{K}} \Mag \\
    \Simp  \otimes_{\B{K}} \Sym & \longrightarrow \Sym  \otimes_{\B{K}} \Simp 
  \end{align}
  Below, we will also display the string diagrams of \emph{some} of the identities. The inclusion of a string diagram simply serves to guide the reader in cases where the identity involves lengthy algebraic manipulations.
\end{remark}

But before we proceed, we are going to need the following Lemma:
\begin{lemma}\label{lem:distributive}
  Let $\C{A}=T(V) / I_{\C{A}}$ and $\C{B}=T(W) / I_{\C{B}}$ be two $\B{K}$-algebras. Let $\omega\colon W \otimes_{\B{K}} V \rightarrow V \otimes_{\B{K}} W$ be a morphism of faithful $\B{K}$-bimodules that induces a left transposition for $\C{A}$ of the form $\omega\colon  T(W) \otimes_{\B{K}}\C{A} \rightarrow \C{A} \otimes_{\B{K}} T(W)$ and a right transposition for $\C{B}$ as  $\omega\colon\C{B} \otimes_{\B{K}} T(V) \rightarrow T(V) \otimes_{\B{K}} \C{B}$. Then $\omega$ induces a unital distributive law of the form $\omega\colon \C{B} \otimes_{\B{K}} \C{A} \rightarrow \C{A} \otimes_{\B{K}} \C{B}$.
\end{lemma}

\begin{proof}
  The only thing we need to prove is that we have a well-defined morphism of $\B{K}$-bimodules of the form $\omega\colon \C{B} \otimes_{\B{K}} \C{A} \rightarrow \C{A} \otimes_{\B{K}} \C{B}$. The assumption that $\omega$ induces a right transposition for $\C{B}$ and a left transposition for $\C{A}$  implies that $$\omega(I_{\C{B}}\otimes_{\B{K}} T(V)) \subset T(V) \otimes_{\B{K}}I_{\C{B}} \quad \text{ and }\quad \omega(T(W) \otimes_{\B{K}}I_{\C{A}}) \subset I_{\C{A}}\otimes_{\B{K}} T(W)$$ Since we have
  \[ \C{B}\otimes_{\B{K}}\C{A} := \frac{T(W)\otimes_{\B{K}} T(V)}{I_{\C{B}}\otimes_{\B{K}} T(V) + T(W)\otimes_{\B{K}} I_{\C{A}}} \quad\text{ and }\quad \C{A}\otimes_{\B{K}}\C{B} := \frac{T(V)\otimes_{\B{K}} T(W)}{I_{\C{A}}\otimes_{\B{K}} T(W) + T(V)\otimes_{\B{K}} I_{\C{B}}}
  \]
  we see that the induced $\B{K}$-bimodule morphism is a well-defined unital distributive law.
\end{proof}
	
\subsection{$\Mag  \otimes_{\B{K}} \Braid \to \Braid  \otimes_{\B{K}} \Mag $}\label{subsect:MagToBraid}

We need to show that the conditions stated in Lemma~\ref{lem:distributive} are satisfied for $\zeta$. Firstly, we are going to show that $\zeta: T(\partial) \otimes_{\B{K}} T(\chi) \to T(\chi) \otimes_{\B{K}} T(\partial)$ extends to a left transposition of the form $\zeta: T(\partial) \otimes_{\B{K}} \Braid \to \Braid \otimes_{\B{K}} T(\partial)$. We need to show that the following diagram commutes: 
\[\begin{tikzcd}
    T(\partial)\otimes_{\B{K}}\Braid\otimes_{\B{K}}\Braid   
    \arrow[r, "{\zeta\otimes\Braid}"]   
    \arrow[d, "{T(\partial)\otimes\mu_{\Braid}}"'] 
    &   
    \Braid \otimes_{\B{K}} T(\partial) \otimes_{\B{K}} \Braid 
    \arrow[r, "{\Braid\otimes\zeta}"]
    & 	
    \Braid \otimes_{\B{K}}\Braid\otimes_{\B{K}}T(\partial)  
    \arrow[d, "{\mu_{\Braid}\otimes T(\partial)}"]
    \\
    T(\partial) \otimes_{\B{K}} \Braid  
    \arrow[rr, "\zeta"'] 
    & & 	
    \Braid\otimes_{\B{K}} T(\partial)
  \end{tikzcd}\]
We check if this is well-defined, i.e., if $\zeta (T(\partial) \otimes_{\B{K}} \mathcal{I}_{\Braid }) \subset \mathcal{I}_{\Braid } \otimes_{\B{K}} T(\partial) $. For this, we have to check that the defining relations of $\mathcal{I}_{\Braid}$ are preserved.

We check the first relation in $\mathcal{I}_{\Braid}$, namely,
\[ \zeta (\partial_i^n \otimes \chi_j^n \chi_k^n) = \zeta (\partial_i^n \otimes \chi_k^n \chi_j^n) \]
in $\Braid \otimes_{\B{K}} T(\partial) $. We need to verify this equality for all $0\leq i \leq n$ and $0\leq j,k \leq n-1$ with $|j-k|\geq 2$. We examine this in $7$ cases as follows:

\begin{enumerate}[{Case} (1):]
\item  $i<j$ and $i < k$: 
  \begin{align*}
    \zeta (\partial_i^n \otimes \chi_j^n\chi_k^n)
    = & \chi_{j+1}^{n+1} \chi_{k+1}^{n+1} \otimes \partial_i^n \\
    = & \chi_{k+1}^{n+1} \chi_{j+1}^{n+1} \otimes \partial_i^n 
        = \zeta (\partial_i^n \otimes \chi_k^n\chi_j^n)
  \end{align*}
  since $|(j+1)-(k+1)|=|j-k|\geq 2$.
  
\item $i<j$ and $i=k$. Since $|(j+1)-(i+1)|=|(j+1)-(k+1)|=|j-k|=|j-i| \geq 2$ and clearly $ |(j+1)-i| \geq |j-i|$, we get
  \begin{align*}
    \zeta (\partial_i^n \otimes \chi_j^n\chi_i^n)
    = & \chi_{j+1}^{n+1} \chi_{i+1}^{n+1} \chi_{i}^{n+1} \otimes \partial_{i+1}^n\\
    = & \chi_{i+1}^{n+1} \chi_{i}^{n+1} \chi_{j+1}^{n+1} \otimes \partial_{i+1}^n
        = \zeta (\partial_i^n \otimes \chi_i^n\chi_j^n)
  \end{align*}
  We can depict the equality above with the string diagrams below. 
  \[ \begin{tabular}{c c}
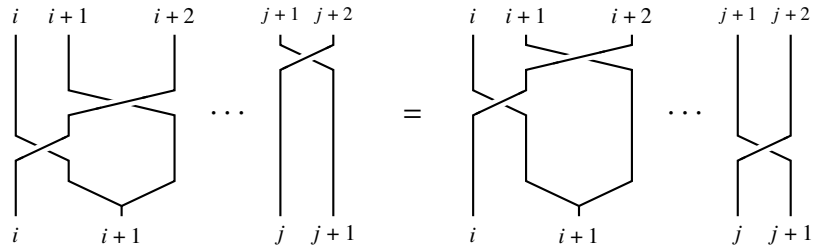
    \begin{tikzpicture}[xscale=0.7, yscale=0.5]
      \begin{knot}[
        clip width=5,
        clip radius=8pt,
        ]
        
        \node at (-1, -0.5)  {\scalebox{0.7}{$i$}} ; 
        \node at (1 , -0.5) {\scalebox{0.7}{$i+1$}} ; 
        \node at (4 , -0.5) {\scalebox{0.7}{$j$}} ; 
        \node at (5 , -0.5) {\scalebox{0.7}{$j+1$}} ; 
        \node at (-1, 5.3 )  {\scalebox{0.7}{$i$}} ; 
        \node at (0 , 5.3) {\scalebox{0.7}{$i+1$}} ; 
        \node at (2 , 5.3) {\scalebox{0.7}{$i+2$}} ; 
        \node at (4 , 5.3) {\scalebox{0.6}{$j+1$}} ; 
        \node at (5 , 5.3) {\scalebox{0.6}{$j+2$}} ; 
        \node at (3 , 2.65) {\scalebox{1}{$\,\cdots$}} ;
        \node at (6.5, 2.65)  {\scalebox{1}{$=$}} ;
        Strand at (2, 4.8)
        \strand[thick] (-1, 0.0) to (-1, 0.27) to (-1, 0.93) to (-1, 1.2) to (-1, 1.47) to (0, 2.13) to (0, 2.4) to (0, 2.67) to (2, 3.33) to (2, 3.6) to (2, 3.87) to (2, 4.53) to (2, 4.8) ;
        
        Strand at (-1, 4.8)
        \strand[thick] (1, 0.27) to (0, 0.93) to (0, 1.2) to (0, 1.47) to (-1, 2.13) to (-1, 2.4) to (-1, 2.67) to (-1, 3.33) to (-1, 3.6) to (-1, 3.87) to (-1, 4.53) to (-1, 4.8) ;
        
        Strand at (0, 4.8)
        \strand[thick] (1, 0.0) to (1, 0.27) to (2, 0.93) to (2, 1.2) to (2, 1.47) to (2, 2.13) to (2, 2.4) to (2, 2.67) to (0, 3.33) to (0, 3.6) to (0, 3.87) to (0, 4.53) to (0, 4.8) ;
        
        Strand at (5, 4.8)
        \strand[thick] (4, 0.0) to (4, 0.27) to (4, 0.93) to (4, 1.2) to (4, 1.47) to (4, 2.13) to (4, 2.4) to (4, 2.67) to (4, 3.33) to (4, 3.6) to (4, 3.87) to (5, 4.53) to (5, 4.8);
        
        Strand at (4, 4.8)
        \strand[thick] (5, 0.0) to (5, 0.27) to (5, 0.93) to (5, 1.2) to (5, 1.47) to (5, 2.13) to (5, 2.4) to (5, 2.67) to (5, 3.33) to (5, 3.6) to (5, 3.87) to (4, 4.53) to (4, 4.8);
        
      \end{knot} 
    \end{tikzpicture}
    & \begin{tikzpicture}[xscale=0.7, yscale=0.5]	
      \begin{knot}[
        clip width=5,
        clip radius=8pt,
        ]
        
        \node at (-1, -0.5)  {\scalebox{0.7}{$i$}} ; 
        \node at (1 , -0.5) {\scalebox{0.7}{$i+1$}} ; 
        \node at (4 , -0.5) {\scalebox{0.7}{$j$}} ; 
        \node at (5 , -0.5) {\scalebox{0.7}{$j+1$}} ; 
        \node at (-1, 5.3 )  {\scalebox{0.7}{$i$}} ; 
        \node at (0 , 5.3) {\scalebox{0.7}{$i+1$}} ; 
        \node at (2 , 5.3) {\scalebox{0.7}{$i+2$}} ; 
        \node at (4 , 5.3) {\scalebox{0.6}{$j+1$}} ; 
        \node at (3 , 2.65) {\scalebox{1}{$\,\cdots$}} ;
        \node at (5 , 5.3) {\scalebox{0.6}{$j+2$}} ; 
        Strand at (2, 4.8)
        \strand[thick] (-1, 0.0) to (-1, 0.27) to (-1, 0.93) to (-1, 1.2) to (-1, 1.47) to (-1, 2.13) to (-1, 2.4) to (-1, 2.67) to (0, 3.33) to (0, 3.6) to (0, 3.87) to (2, 4.53) to (2, 4.8) ;
        
        Strand at (-1, 4.8)
        \strand[thick] (1, 0.27) to (0, 0.93) to (0, 1.2) to (0, 1.47) to (0, 2.13) to (0, 2.4) to (0, 2.67) to (-1, 3.33) to (-1, 3.6) to (-1, 3.87) to (-1, 4.53) to (-1, 4.8) ;
        
        Strand at (0, 4.8)
        \strand[thick] (1, 0.0) to (1, 0.27) to (2, 0.93) to (2, 1.2) to (2, 1.47) to (2, 2.13) to (2, 2.4) to (2, 2.67) to (2, 3.33) to (2, 3.6) to (2, 3.87) to (0, 4.53) to (0, 4.8) ;
        
        Strand at (5, 4.8)
        \strand[thick] (4, 0.0) to (4, 0.27) to (4, 0.93) to (4, 1.2) to (4, 1.47) to (5, 2.13) to (5, 2.4) to (5, 2.67) to (5, 3.33) to (5, 3.6) to (5, 3.87) to (5, 4.53) to (5, 4.8);
        
        Strand at (4, 4.8)
        \strand[thick] (5, 0.0) to (5, 0.27) to (5, 0.93) to (5, 1.2) to (5, 1.47) to (4, 2.13) to (4, 2.4) to (4, 2.67) to (4, 3.33) to (4, 3.6) to (4, 3.87) to (4, 4.53) to (4, 4.8);
      \end{knot} 
    \end{tikzpicture}
  \end{tabular}
\]
Since all derivations are reversible, this case is equivalent to the case where $i<k$ and $i=j$.

\item  $i<j$ and $i=k+1$ (or $k=i-1$.) Conditions imply that $k < i < j$, therefore $(j+1)-i= (j+1)-(k+1)=|j-k| \geq 2$ and clearly $ (j+1)-(i-1)= ((j+1)-i) + 1 \geq 3  $. So, we have
  \begin{align*}
    \zeta (\partial_i^n \otimes \chi_j^n\chi_{i-1}^n)
    = & \chi_{j+1}^{n+1} \chi_{i-1}^{n+1} \chi_{i}^{n+1} \otimes \partial_{i-1}^n  \\
    = & \chi_{i-1}^{n+1} \chi_{i}^{n+1} \chi_{j+1}^{n+1} \otimes \partial_{i-1}^n 
        = \zeta (\partial_i^n \otimes \chi_{i-1}^n\chi_j^n)
  \end{align*}
  One can depict the equality above by the following string diagrams:
  \[\begin{tabular}{c c}
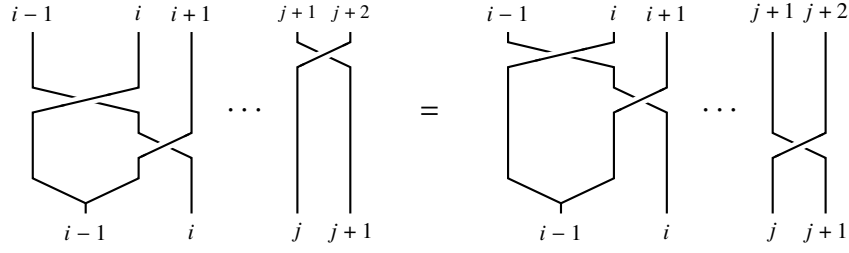
      \begin{tikzpicture}[xscale=0.7, yscale=0.5]
        
        \begin{knot}[
          clip width=5,
          clip radius=8pt,
          ]
          
          \node at (0, -0.5)  {\scalebox{0.7}{$i-1$}} ; 
          \node at (2 , -0.5) {\scalebox{0.7}{$i$}} ;  
          \node at (4 , -0.5) {\scalebox{0.7}{$j$}} ; 
          \node at (5 , -0.5) {\scalebox{0.7}{$j+1$}} ; 
          \node at (3 , 2.65) {\scalebox{1}{$\,\cdots$}} ;
          \node at (-1, 5.3 )  {\scalebox{0.7}{$i-1$}} ; 
          \node at (2 , 5.3) {\scalebox{0.7}{$i+1$}} ; 
          \node at (1 , 5.3) {\scalebox{0.7}{$i$}} ; 
          \node at (4 , 5.3) {\scalebox{0.6}{$j+1$}} ; 
          \node at (5 , 5.3) {\scalebox{0.6}{$j+2$}} ;
          \node at (6.5, 2.65)  {\scalebox{1}{$=$}} ;
          Strand at (1, 4.8)
          \strand[thick] (0, 0.27) to (-1, 0.93) to (-1, 1.2) to (-1, 1.47) to (-1, 2.13) to (-1, 2.4) to (-1, 2.67) to (1, 3.33) to (1, 3.6) to (1, 3.87) to (1, 4.53) to (1, 4.8) ;
          
          Strand at (2, 4.8)
          \strand[thick] (0, 0.0) to (0, 0.27) to (1, 0.93) to (1, 1.2) to (1, 1.47) to (2, 2.13) to (2, 2.4) to (2, 2.67) to (2, 3.33) to (2, 3.6) to (2, 3.87) to (2, 4.53) to (2, 4.8) ;
          
          Strand at (-1, 4.8)
          \strand[thick] (2, 0.0) to (2, 0.27) to (2, 0.93) to (2, 1.2) to (2, 1.47) to (1, 2.13) to (1, 2.4) to (1, 2.67) to (-1, 3.33) to (-1, 3.6) to (-1, 3.87) to (-1, 4.53) to (-1, 4.8) ;

          Strand at (6, 4.8)
          \strand[thick] (4, 0.0) to (4, 0.27) to (4, 0.93) to (4, 1.2) to (4, 1.47) to (4, 2.13) to (4, 2.4) to (4, 2.67) to (4, 3.33) to (4, 3.6) to (4, 3.87) to (5, 4.53) to (5, 4.8);
          
          Strand at (5, 4.8)
          \strand[thick] (5, 0.0) to (5, 0.27) to (5, 0.93) to (5, 1.2) to (5, 1.47) to (5, 2.13) to (5, 2.4) to (5, 2.67) to (5, 3.33) to (5, 3.6) to (5, 3.87) to (4, 4.53) to (4, 4.8);
          
        \end{knot} 
      \end{tikzpicture} 
      & \begin{tikzpicture}[xscale=0.7, yscale=0.5]
        
        \begin{knot}[
          clip width=5,
          clip radius=8pt,
          ]
          
          \node at (0, -0.5)  {\scalebox{0.7}{$i-1$}} ; 
          \node at (2 , -0.5) {\scalebox{0.7}{$i$}} ;  
          \node at (4 , -0.5) {\scalebox{0.7}{$j$}} ; 
          \node at (5 , -0.5) {\scalebox{0.7}{$j+1$}} ; 
          \node at (3 , 2.65) {\scalebox{1}{$\,\cdots$}} ;
          \node at (-1, 5.3 )  {\scalebox{0.7}{$i-1$}} ; 
          \node at (2 , 5.3) {\scalebox{0.7}{$i+1$}} ; 
          \node at (1 , 5.3) {\scalebox{0.7}{$i$}} ; 
          \node at (4 , 5.3) {\scalebox{0.7}{$j+1$}} ; 
          \node at (5 , 5.3) {\scalebox{0.7}{$j+2$}} ; 
          
          Strand at (1, 4.8)
          \strand[thick] (0, 0.27) to (-1, 0.93) to (-1, 1.2) to (-1, 1.47) to (-1, 2.13) to (-1, 2.4) to (-1, 2.67) to (-1, 3.33) to (-1, 3.6) to (-1, 3.87) to (1, 4.53) to (1, 4.8) ;
          
          Strand at (2, 4.8)
          \strand[thick] (0, 0.0) to (0, 0.27) to (1, 0.93) to (1, 1.2) to (1, 1.47) to (1, 2.13) to (1, 2.4) to (1, 2.67) to (2, 3.33) to (2, 3.6) to (2, 3.87) to (2, 4.53) to (2, 4.8) ;
          
          Strand at (-1, 4.8)
          \strand[thick] (2, 0.0) to (2, 0.27) to (2, 0.93) to (2, 1.2) to (2, 1.47) to (2, 2.13) to (2, 2.4) to (2, 2.67) to (1, 3.33) to (1, 3.6) to (1, 3.87) to (-1, 4.53) to (-1, 4.8) ;

          Strand at (5, 4.8)
          \strand[thick] (4, 0.0) to (4, 0.27) to (4, 0.93) to (4, 1.2) to (4, 1.47) to (5, 2.13) to (5, 2.4) to (5, 2.67) to (5, 3.33) to (5, 3.6) to (5, 3.87) to (5, 4.53) to (5, 4.8);
          
          Strand at (4, 4.8)
          \strand[thick] (5, 0.0) to (5, 0.27) to (5, 0.93) to (5, 1.2) to (5, 1.47) to (4, 2.13) to (4, 2.4) to (4, 2.67) to (4, 3.33) to (4, 3.6) to (4, 3.87) to (4, 4.53) to (4, 4.8);
         
        \end{knot}     
      \end{tikzpicture} 
    \end{tabular}
  \]
  As before, since all derivations are reversible, this case is equivalent to the case where $i<k$ and $i=j+1$.
  
\item $i<j$ and $i>k+1$. Since $k < i < j$ we have $j-k \geq 2 $, and therefore,  $(j+1) - k \geq 2$. Then $\chi_{j+1}^{n+1} \chi_{k}^{n+1}=\chi_{k}^{n+1} \chi_{j+1}^{n+1}$ which implies
  \begin{align*}
    \zeta (\partial_i^n \otimes \chi_j^n\chi_k^n)
    = & \chi_{j+1}^{n+1} \chi_{k}^{n+1} \otimes \partial_i^n \\
    = & \chi_{k}^{n+1} \chi_{j+1}^{n+1} \otimes \partial_i^n 
        = \zeta (\partial_i^n \otimes \chi_k^n\chi_j^n)
  \end{align*}
  
\item $i=j$ and $i > k+1$.  Since $k < k+1 < i =j $ we have $|(i+1)-k|=|(j+1)-k|=j-k+1 \geq 2$ and clearly $ |i-k|=|j-k| \geq 2$. Then
  \begin{align*}
    \zeta (\partial_i^n \otimes \chi_i^n\chi_k^n)
    = & \chi_{i+1}^{n+1} \chi_{i}^{n+1} \chi_{k}^{n+1} \otimes \partial_{i+1}^n \\
    = & \chi_{k}^{n+1}\chi_{i+1}^{n+1} \chi_{i}^{n+1} \otimes \partial_{i+1}^n  
        = \zeta (\partial_i^n \otimes \chi_k^n\chi_i^n)
  \end{align*}
  We can represent the equality above with the string diagram below:
  \[\begin{tabular}{c c}		
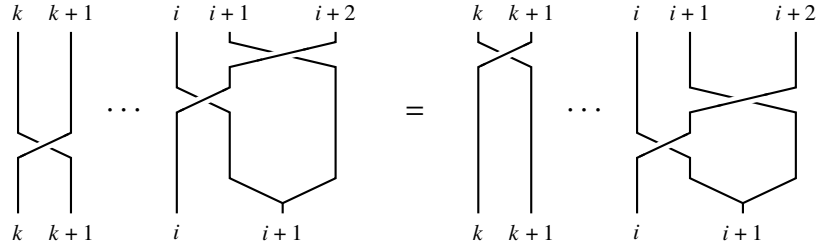
    \begin{tikzpicture}[xscale=0.7, yscale=0.5]
      
      \begin{knot}[
        clip width=5,
        clip radius=8pt,
        ]
        
        \node at (-1, -0.5)  {\scalebox{0.7}{$k$}} ; 
        \node at (0 , -0.5) {\scalebox{0.7}{$k+1$}} ; 
        \node at (2 , -0.5) {\scalebox{0.7}{$i$}} ; 
        \node at (4 , -0.5) {\scalebox{0.7}{$i+1$}} ; 
        \node at (-1, 5.3 )  {\scalebox{0.7}{$k$}} ; 
        \node at (0 , 5.3) {\scalebox{0.7}{$k+1$}} ; 
        \node at (2 , 5.3) {\scalebox{0.7}{$i$}} ; 
        \node at (3 , 5.3) {\scalebox{0.7}{$i+1$}} ; 
        \node at (5 , 5.3) {\scalebox{0.7}{$i+2$}} ; 
        \node at (1 , 2.65) {\scalebox{1}{$\,\cdots$}} ;
        \node at (6.5, 2.65)  {\scalebox{1}{$=$}} ;
        Strand at (2, 4.8)
        \strand[thick] (2, 0.0) to (2, 0.27) to (2, 0.93) to (2, 1.2) to (2, 1.47) to (2, 2.13) to (2, 2.4) to (2, 2.67) to (3, 3.33) to (3, 3.6) to (3, 3.87) to (5, 4.53) to (5, 4.8);
        
        Strand at (2, 4.8)
        \strand[thick] (4, 0.27) to (3, 0.93) to (3, 1.2) to (3, 1.47) to (3, 2.13) to (3, 2.4) to (3, 2.67) to (2, 3.33) to (2, 3.6) to (2, 3.87) to (2, 4.53) to (2, 4.8);
        
        Strand at (3, 4.8)
        \strand[thick] (4, 0.0) to (4, 0.27) to (5, 0.93) to (5, 1.2) to (5, 1.47) to (5, 2.13) to (5, 2.4) to (5, 2.67) to (5, 3.33) to (5, 3.6) to (5, 3.87) to (3, 4.53) to (3, 4.8);

        Strand at (0, 4.8)
        \strand[thick] (-1, 0.0) to (-1, 0.27) to (-1, 0.93) to (-1, 1.2) to (-1, 1.47) to (0, 2.13) to (0, 2.4) to (0, 2.67) to (0, 3.33) to (0, 3.6) to (0, 3.87) to (0, 4.53) to (0, 4.8);
        
        Strand at (-1, 4.8)
        \strand[thick] (0, 0.0) to (0, 0.27) to (0, 0.93) to (0, 1.2) to (0, 1.47) to (-1, 2.13) to (-1, 2.4) to (-1, 2.67) to (-1, 3.33) to (-1, 3.6) to (-1, 3.87) to (-1, 4.53) to (-1, 4.8);
      \end{knot} 
    \end{tikzpicture}
    & \begin{tikzpicture}[xscale=0.7, yscale=0.5]		
      \begin{knot}[
        clip width=5,
        clip radius=8pt,
        ]
        
        \node at (-1, -0.5)  {\scalebox{0.7}{$k$}} ; 
        \node at (0 , -0.5) {\scalebox{0.7}{$k+1$}} ; 
        \node at (2 , -0.5) {\scalebox{0.7}{$i$}} ; 
        \node at (4 , -0.5) {\scalebox{0.7}{$i+1$}} ; 
        \node at (-1, 5.3 )  {\scalebox{0.7}{$k$}} ; 
        \node at (0 , 5.3) {\scalebox{0.7}{$k+1$}} ; 
        \node at (2 , 5.3) {\scalebox{0.7}{$i$}} ; 
        \node at (3 , 5.3) {\scalebox{0.7}{$i+1$}} ; 
        \node at (5 , 5.3) {\scalebox{0.7}{$i+2$}} ; 
        \node at (1 , 2.65) {\scalebox{1}{$\,\cdots$}} ;
        Strand at (2, 4.8)
        \strand[thick] (2, 0.0) to (2, 0.27) to (2, 0.93) to (2, 1.2) to (2, 1.47) to (3, 2.13) to (3, 2.4) to (3, 2.67) to (5, 3.33) to (5, 3.6) to (5, 3.87) to (5, 4.53) to (5, 4.8);
        
        Strand at (2, 4.8)
        \strand[thick] (4, 0.27) to (3, 0.93) to (3, 1.2) to (3, 1.47) to (2, 2.13) to (2, 2.4) to (2, 2.67) to (2, 3.33) to (2, 3.6) to (2, 3.87) to (2, 4.53) to (2, 4.8);
        
        Strand at (3, 4.8)
        \strand[thick] (4, 0.0) to (4, 0.27) to (5, 0.93) to (5, 1.2) to (5, 1.47) to (5, 2.13) to (5, 2.4) to (5, 2.67) to (3, 3.33) to (3, 3.6) to (3, 3.87) to (3, 4.53) to (3, 4.8);
        
        Strand at (0, 4.8)
        \strand[thick] (-1, 0.0) to (-1, 0.27) to (-1, 0.93) to (-1, 1.2) to (-1, 1.47) to (-1, 2.13) to (-1, 2.4) to (-1, 2.67) to (-1, 3.33) to (-1, 3.6) to (-1, 3.87) to (0, 4.53) to (0, 4.8);
        
        Strand at (-1, 4.8)
        \strand[thick] (0, 0.0) to (0, 0.27) to (0, 0.93) to (0, 1.2) to (0, 1.47) to (0, 2.13) to (0, 2.4) to (0, 2.67) to (0, 3.33) to (0, 3.6) to (0, 3.87) to (-1, 4.53) to (-1, 4.8);	
      \end{knot} 
    \end{tikzpicture}
  \end{tabular}
\]
These derivations are reversible as before. Thus, this case is equivalent to the case where $i<k$ and $i>j+1$.	
		
\item $i=j+1$ (or $j=i-1$) and $i>k+1$. Since $i > j > k$ with $j-k \geq 2$ in this case, we have $i-k > (i-1)-k = j-k \geq 2$. Then
  \begin{align*}
    \zeta (\partial_i^n \otimes \chi_{i-1}^n\chi_{k}^n)
    = & \chi_{i-1}^{n+1} \chi_{i}^{n+1} \chi_{k}^{n+1} \otimes \partial_{i-1}^n  \\
    = & \chi_{k}^{n+1} \chi_{i-1}^{n+1} \chi_{i}^{n+1}  \otimes \partial_{i-1}^n
        = \zeta (\partial_i^n \otimes \chi_k^n\chi_{i-1}^n)
  \end{align*}
  The string diagram depicting the equality above is as follows:	
  \[\begin{tabular}{c c}
    
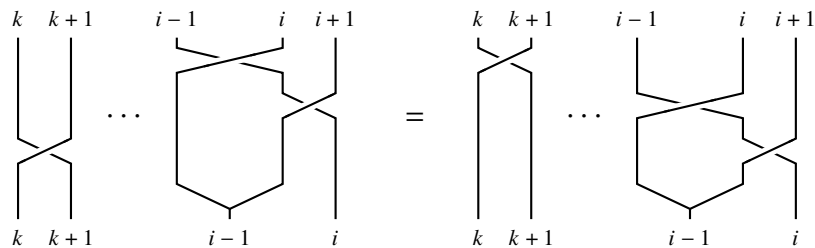
    \begin{tikzpicture}[xscale=0.7, yscale=0.5]
      \begin{knot}[
        clip width=5,
        clip radius=8pt,
        ]
        
        \node at (-1, -0.5)  {\scalebox{0.7}{$k$}} ; 
        \node at (0 , -0.5) {\scalebox{0.7}{$k+1$}} ; 
        \node at (3, -0.5) {\scalebox{0.7}{$i-1$}} ; 
        \node at (5 , -0.5) {\scalebox{0.7}{$i$}} ; 
        \node at (-1, 5.3 )  {\scalebox{0.7}{$k$}} ; 
        \node at (0 , 5.3) {\scalebox{0.7}{$k+1$}} ; 
        \node at (2 , 5.3) {\scalebox{0.7}{$i-1$}} ; 
        \node at (4 , 5.3) {\scalebox{0.7}{$i$}} ; 
        \node at (5 , 5.3) {\scalebox{0.7}{$i+1$}} ; 
        \node at (1 , 2.65) {\scalebox{1}{$\,\cdots$}} ; 
        \node at (6.5, 2.65)  {\scalebox{1}{$=$}} ;
        
        Strand at (4, 4.8)
        \strand[thick] (3, 0.27) to (2, 0.93) to (2, 1.2) to (2, 1.47) to (2, 2.13) to (2, 2.4) to (2, 2.67) to (2, 3.33) to (2, 3.6) to (2, 3.87) to (4, 4.53) to (4, 4.8) ;
        
        Strand at (5, 4.8)
        \strand[thick] (3, 0.0) to (3, 0.27) to (4, 0.93) to (4, 1.2) to (4, 1.47) to (4, 2.13) to (4, 2.4) to (4, 2.67) to (5, 3.33) to (5, 3.6) to (5, 3.87) to (5, 4.53) to (5, 4.8) ;
        
        Strand at (1, 4.8)
        \strand[thick] (5, 0.0) to (5, 0.27) to (5, 0.93) to (5, 1.2) to (5, 1.47) to (5, 2.13) to (5, 2.4) to (5, 2.67) to (4, 3.33) to (4, 3.6) to (4, 3.87) to (2, 4.53) to (2, 4.8) ;
        
        Strand at (0, 4.8)
        \strand[thick] (-1, 0.0) to (-1, 0.27) to (-1, 0.93) to (-1, 1.2) to (-1, 1.47) to (0, 2.13) to (0, 2.4) to (0, 2.67) to (0, 3.33) to (0, 3.6) to (0, 3.87) to (0, 4.53) to (0, 4.8);
        
        Strand at (-1, 4.8)
        \strand[thick] (0, 0.0) to (0, 0.27) to (0, 0.93) to (0, 1.2) to (0, 1.47) to (-1, 2.13) to (-1, 2.4) to (-1, 2.67) to (-1, 3.33) to (-1, 3.6) to (-1, 3.87) to (-1, 4.53) to (-1, 4.8);				
      \end{knot} 				
    \end{tikzpicture}
    & \begin{tikzpicture}[xscale=0.7, yscale=0.5]	
      \begin{knot}[
        clip width=5,
        clip radius=8pt,
        ]
        
        \node at (-1, -0.5)  {\scalebox{0.7}{$k$}} ; 
        \node at (0 , -0.5) {\scalebox{0.7}{$k+1$}} ; 
        \node at (3, -0.5) {\scalebox{0.7}{$i-1$}} ; 
        \node at (5 , -0.5) {\scalebox{0.7}{$i$}} ; 
        \node at (-1, 5.3 )  {\scalebox{0.7}{$k$}} ; 
        \node at (0 , 5.3) {\scalebox{0.7}{$k+1$}} ; 
        \node at (2 , 5.3) {\scalebox{0.7}{$i-1$}} ; 
        \node at (4 , 5.3) {\scalebox{0.7}{$i$}} ; 
        \node at (5 , 5.3) {\scalebox{0.7}{$i+1$}} ; 
        \node at (1 , 2.65) {\scalebox{1}{$\,\cdots$}} ; 
        
        Strand at (4, 4.8)
        \strand[thick] (3, 0.27) to (2, 0.93) to (2, 1.2) to (2, 1.47) to (2, 2.13) to (2, 2.4) to (2, 2.67) to (4, 3.33) to (4, 3.6) to (4, 3.87) to (4, 4.53) to (4, 4.8) ;
        
        Strand at (5, 4.8)
        \strand[thick] (3, 0.0) to (3, 0.27) to (4, 0.93) to (4, 1.2) to (4, 1.47) to (5, 2.13) to (5, 2.4) to (5, 2.67) to (5, 3.33) to (5, 3.6) to (5, 3.87) to (5, 4.53) to (5, 4.8) ;
        
        Strand at (1, 4.8)
        \strand[thick] (5, 0.0) to (5, 0.27) to (5, 0.93) to (5, 1.2) to (5, 1.47) to (4, 2.13) to (4, 2.4) to (4, 2.67) to (2, 3.33) to (2, 3.6) to (2, 3.87) to (2, 4.53) to (2, 4.8) ;

        Strand at (1, 4.8)
        \strand[thick] (-1, 0.0) to (-1, 0.27) to (-1, 0.93) to (-1, 1.2) to (-1, 1.47) to (-1, 2.13) to (-1, 2.4) to (-1, 2.67) to (-1, 3.33) to (-1, 3.6) to (-1, 3.87) to (0, 4.53) to (0, 4.8);
        
        Strand at (0, 4.8)
        \strand[thick] (0, 0.0) to (0, 0.27) to (0, 0.93) to (0, 1.2) to (0, 1.47) to (0, 2.13) to (0, 2.4) to (0, 2.67) to (0, 3.33) to (0, 3.6) to (0, 3.87) to (-1, 4.53) to (-1, 4.8);
      \end{knot} 
    \end{tikzpicture} 
      
  \end{tabular}
\]
This case is equivalent to the case where  $i>j+1$ and $i=k+1$ since the derivations are reversible.
		
\item $i>j+1$ and $i>k+1$
  \begin{align*}
    \zeta (\partial_i^n \otimes \chi_j^n\chi_k^n)
    = & \chi_{j}^{n+1} \chi_{k}^{n+1} \otimes \partial_i^n \\
    = & \chi_{k}^{n+1} \chi_{j}^{n+1} \otimes \partial_i^n 
        = \zeta (\partial_i^n \otimes \chi_k^n\chi_j^n)
  \end{align*}
\end{enumerate}

We have to check the remaining relation in $\mathcal{I}_{\Braid}$, namely,
\[ \zeta (\partial_i^n \otimes \chi_j^n \chi_{j+1}^n \chi_j^n) = \zeta (\partial_i^n \otimes \chi_{j+1}^n \chi_j^n \chi_{j+1}^n) \]
in $\Braid \otimes_{\B{K}} T(\partial) $. We need to verify this equality for all $0\leq i \leq n$ and $0\leq j \leq n-2$ where $n \geq 2$. We examine this in $5$ cases as follows:

\begin{enumerate}[{Case} (1):]
\item $i<j$
  \begin{align*}
    \zeta (\partial_i^n \otimes \chi_j^n \chi_{j+1}^n \chi_j^n)
    = & \chi_{j+1}^{n+1} \chi_{j+2}^{n+1} \chi_{j+1}^{n+1}\otimes \partial_i^n\\
    = & \chi_{j+2}^{n+1} \chi_{j+1}^{n+1} \chi_{j+2}^{n+1} \otimes \partial_i^n
        = \zeta (\partial_i^n \otimes \chi_{j+1}^n \chi_j^n \chi_{j+1}^n)
  \end{align*}
  
\item $i>j+2$
  \begin{align*}
    \zeta (\partial_i^n \otimes \chi_j^n \chi_{j+1}^n \chi_j^n)
    = & \chi_j^{n+1} \chi_{j+1}^{n+1} \chi_j^{n+1} \otimes \partial_i^n \\
    = & \chi_{j+1}^{n+1} \chi_j^{n+1} \chi_{j+1}^{n+1} \otimes \partial_i^n
        = \zeta (\partial_i^n \otimes \chi_{j+1}^n \chi_j^n \chi_{j+1}^n)
  \end{align*}
  
\item $i=j+2$. We have 
  \begin{align*}
    \zeta (\partial_{j+2}^n \otimes \chi_{j+1}^n \chi_j^n \chi_{j+1}^n)
    % = & \zeta (\partial_{j+2}^n \otimes \chi_{j+1}^n) \otimes \chi_{j}^n \chi_{j+1}^n \\ 
    % = & \chi_{j+1}^{n+1} \chi_{j+2}^{n+1} \otimes \zeta (\partial_{j+1}^n \otimes \chi_{j}^n) \otimes \chi_{j+1}^n \\ 
    % = & \chi_{j+1}^{n+1} \chi_{j+2}^{n+1} \otimes \chi_j^{n+1}  \chi_{j+1}^{n+1} \otimes \zeta (\partial_{j}^n \otimes \chi_{j+1}^n) \\ 
    = & \chi_{j+1}^{n+1} \chi_{j+2}^{n+1}  \chi_j^{n+1}  \chi_{j+1}^{n+1}  \chi_{j+2}^{n+1} \otimes \partial_{j}^n
  \end{align*}
  On the other hand, we have
  \begin{align*}
    \chi_{j+1}^{n+1} \chi_{j+2}^{n+1}\chi_j^{n+1}  \chi_{j+1}^{n+1}\chi_{j+2}^{n+1}  
    % = & \chi_{j+1}^{n+1} (\chi_{j}^{n+1}  \chi_{j+2}^{n+1})  \chi_{j+1}^{n+1}  \chi_{j+2}^{n+1} \\
    = & \chi_{j+1}^{n+1} \chi_{j}^{n+1}  \chi_{j+2}^{n+1} \chi_{j+1}^{n+1}  \chi_{j+2}^{n+1} \\
    % = & \chi_{j+1}^{n+1} \chi_{j}^{n+1}  (\chi_{j+1}^{n+1} \chi_{j+2}^{n+1}  \chi_{j+1}^{n+1})\\
    = & \chi_{j+1}^{n+1} \chi_{j}^{n+1}  \chi_{j+1}^{n+1} \chi_{j+2}^{n+1}  \chi_{j+1}^{n+1} 
    % = & \chi_{j}^{n+1} \chi_{j+1}^{n+1}  \chi_{j}^{n+1} \chi_{j+2}^{n+1}  \chi_{j+1}^{n+1}\\
    % = & \chi_{j}^{n+1} \chi_{j+1}^{n+1}  (\chi_{j}^{n+1} \chi_{j+2}^{n+1} ) \chi_{j+1}^{n+1}\\
    % = & \chi_{j}^{n+1} \chi_{j+1}^{n+1}  (\chi_{j+2}^{n+1} \chi_{j}^{n+1} ) \chi_{j+1}^{n+1}\\
			= \chi_{j}^{n+1} \chi_{j+1}^{n+1}  \chi_{j+2}^{n+1} \chi_{j}^{n+1} \chi_{j+1}^{n+1}
  \end{align*}
  Thus
  \begin{align*}
    \chi_{j+1}^{n+1} \chi_{j+2}^{n+1}  \chi_j^{n+1}  \chi_{j+1}^{n+1}  \chi_{j+2}^{n+1} \otimes \partial_{j}^n
    % = & \zeta (\partial_{j+2}^n \otimes \chi_j^n) \otimes \chi_{j+1}^n \chi_j^n  \\ 
    % = & \chi_j^{n+1} \otimes \zeta (\partial_{j+2}^n \otimes \chi_{j+1}^n) \otimes \chi_j^n \\ 
    % = & \chi_j^{n+1} \otimes  \chi_{j+1}^{n+1} \chi_{j+2}^{n+1} \otimes \zeta (\partial_{j+1}^n \otimes \chi_{j}^n) \\ 
    = & \chi_j^{n+1}  \chi_{j+1}^{n+1} \chi_{j+2}^{n+1} \chi_{j}^{n+1} \chi_{j+1}^{n+1}  \otimes \partial_{j}^n\\
    = & \zeta (\partial_{j+2}^n \otimes \chi_j^n \chi_{j+1}^n \chi_j^n) 
  \end{align*}
  These identities can be seen more easily from the following diagrams:
  \[ \begin{tabular}{cccc}
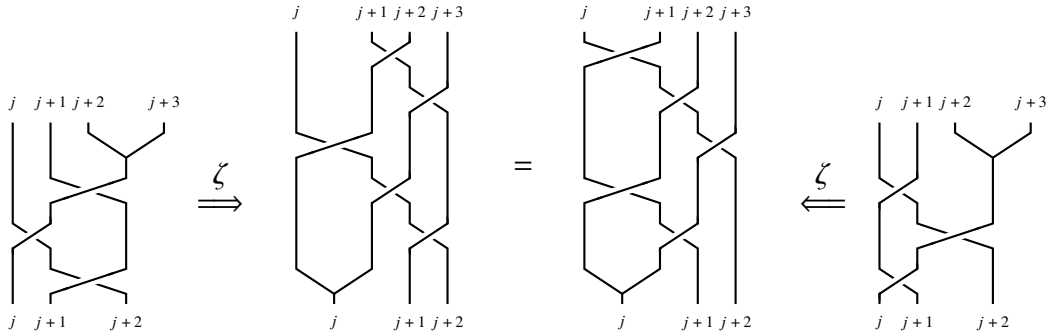
    \begin{tikzpicture}[xscale=0.5, yscale=0.5]
      
      \begin{knot}[
        clip width=5,
        clip radius=8pt,
        ]
        
        \node at (-1, -0.5)  {\scalebox{0.5}{$j$}} ; 
        \node at (0 , -0.5) {\scalebox{0.5}{$j+1$}} ; 
        \node at (2 , -0.5) {\scalebox{0.5}{$j+2$}} ; 
        \node at (-1, 5.3 )  {\scalebox{0.5}{$j$}} ; 
        \node at (0 , 5.3) {\scalebox{0.5}{$j+1$}} ; 
        \node at (1 , 5.3) {\scalebox{0.5}{$j+2$}} ; 
        \node at (3 , 5.3) {\scalebox{0.5}{$j+3$}} ;
        \node at (4.5, 2.65)  {\scalebox{1}{$\Longrightarrow$}} ;
        \node at (4.5, 3.4)  {\scalebox{1}{$\zeta$}} ;
        Strand at (1, 4.8)
        \strand[thick] (2, 3.87) to (1, 4.53) to (1, 4.8) ;
        
        Strand at (3, 4.8)
        \strand[thick] (-1, 0.0) to (-1, 0.27) to (-1, 0.93) to (-1, 1.2) to (-1, 1.47) to (0, 2.13) to (0, 2.4) to (0, 2.67) to (2, 3.33) to (2, 3.6) to (2, 3.87) to (3, 4.53) to (3, 4.8) ;
        
        Strand at (0, 4.8)
        \strand[thick] (0, 0.0) to (0, 0.27) to (2, 0.93) to (2, 1.2) to (2, 1.47) to (2, 2.13) to (2, 2.4) to (2, 2.67) to (0, 3.33) to (0, 3.6) to (0, 3.87) to (0, 4.53) to (0, 4.8) ;
        
        Strand at (-1, 4.8)
        \strand[thick] (2, 0.0) to (2, 0.27) to (0, 0.93) to (0, 1.2) to (0, 1.47) to (-1, 2.13) to (-1, 2.4) to (-1, 2.67) to (-1, 3.33) to (-1, 3.6) to (-1, 3.87) to (-1, 4.53) to (-1, 4.8) ;
        
      \end{knot} 
      
    \end{tikzpicture} 
    & \begin{tikzpicture}[xscale=0.5, yscale=0.5]
      
      \begin{knot}[
        clip width=5,
        clip radius=8pt,
        ]
        
        \node at (0, -0.5)  {\scalebox{0.5}{$j$}} ; 
        \node at (2 , -0.5) {\scalebox{0.5}{$j+1$}} ; 
        \node at (3 , -0.5) {\scalebox{0.5}{$j+2$}} ; 
        \node at (-1, 7.7 )  {\scalebox{0.5}{$j$}} ; 
        \node at (1 , 7.7) {\scalebox{0.5}{$j+1$}} ; 
        \node at (2 , 7.7) {\scalebox{0.5}{$j+2$}} ; 
        \node at (3 , 7.7) {\scalebox{0.5}{$j+3$}} ;
        \node at (5 , 3.6) {\scalebox{1}{$=$}} ; 
        Strand at (2, 7.2)
        \strand[thick] (0, 0.27) to (-1, 0.93) to (-1, 1.2) to (-1, 1.47) to (-1, 2.13) to (-1, 2.4) to (-1, 2.67) to (-1, 3.33) to (-1, 3.6) to (-1, 3.87) to (1, 4.53) to (1, 4.8) to (1, 5.07) to (1, 5.73) to (1, 6.0) to (1, 6.27) to (2, 6.93) to (2, 7.2) ;
        
        Strand at (3, 7.2)
        \strand[thick] (0, 0.0) to (0, 0.27) to (1, 0.93) to (1, 1.2) to (1, 1.47) to (1, 2.13) to (1, 2.4) to (1, 2.67) to (2, 3.33) to (2, 3.6) to (2, 3.87) to (2, 4.53) to (2, 4.8) to (2, 5.07) to (3, 5.73) to (3, 6.0) to (3, 6.27) to (3, 6.93) to (3, 7.2) ;
        
        Strand at (1, 7.2)
        \strand[thick] (2, 0.0) to (2, 0.27) to (2, 0.93) to (2, 1.2) to (2, 1.47) to (3, 2.13) to (3, 2.4) to (3, 2.67) to (3, 3.33) to (3, 3.6) to (3, 3.87) to (3, 4.53) to (3, 4.8) to (3, 5.07) to (2, 5.73) to (2, 6.0) to (2, 6.27) to (1, 6.93) to (1, 7.2) ;
        
        Strand at (-1, 7.2)
        \strand[thick] (3, 0.0) to (3, 0.27) to (3, 0.93) to (3, 1.2) to (3, 1.47) to (2, 2.13) to (2, 2.4) to (2, 2.67) to (1, 3.33) to (1, 3.6) to (1, 3.87) to (-1, 4.53) to (-1, 4.8) to (-1, 5.07) to (-1, 5.73) to (-1, 6.0) to (-1, 6.27) to (-1, 6.93) to (-1, 7.2) ;
        
      \end{knot} 
      
    \end{tikzpicture}
    & \begin{tikzpicture}[xscale=0.5, yscale=0.5]
      
      \begin{knot}[
        clip width=5,
        clip radius=8pt,
        ]
        
        \node at (0, -0.5)  {\scalebox{0.5}{$j$}} ; 
        \node at (2 , -0.5) {\scalebox{0.5}{$j+1$}} ; 
        \node at (3 , -0.5) {\scalebox{0.5}{$j+2$}} ; 
        \node at (-1, 7.7 )  {\scalebox{0.5}{$j$}} ; 
        \node at (1 , 7.7) {\scalebox{0.5}{$j+1$}} ; 
        \node at (2 , 7.7) {\scalebox{0.5}{$j+2$}} ; 
        \node at (3 , 7.7) {\scalebox{0.5}{$j+3$}} ; 
        % Strand at (2, 7.2)
        \strand[thick] (0, 0.27) to (-1, 0.93) to (-1, 1.2) to (-1, 1.47) to (-1, 2.13) to (-1, 2.4) to (-1, 2.67) to (1, 3.33) to (1, 3.6) to (1, 3.87) to (1, 4.53) to (1, 4.8) to (1, 5.07) to (2, 5.73) to (2, 6.0) to (2, 6.27) to (2, 6.93) to (2, 7.2) ;
        
        % Strand at (3, 7.2)
        \strand[thick] (0, 0.0) to (0, 0.27) to (1, 0.93) to (1, 1.2) to (1, 1.47) to (2, 2.13) to (2, 2.4) to (2, 2.67) to (2, 3.33) to (2, 3.6) to (2, 3.87) to (3, 4.53) to (3, 4.8) to (3, 5.07) to (3, 5.73) to (3, 6.0) to (3, 6.27) to (3, 6.93) to (3, 7.2) ;
        
        % Strand at (1, 7.2)
        \strand[thick] (2, 0.0) to (2, 0.27) to (2, 0.93) to (2, 1.2) to (2, 1.47) to (1, 2.13) to (1, 2.4) to (1, 2.67) to (-1, 3.33) to (-1, 3.6) to (-1, 3.87) to (-1, 4.53) to (-1, 4.8) to (-1, 5.07) to (-1, 5.73) to (-1, 6.0) to (-1, 6.27) to (1, 6.93) to (1, 7.2) ;
        
        % Strand at (-1, 7.2)
        \strand[thick] (3, 0.0) to (3, 0.27) to (3, 0.93) to (3, 1.2) to (3, 1.47) to (3, 2.13) to (3, 2.4) to (3, 2.67) to (3, 3.33) to (3, 3.6) to (3, 3.87) to (2, 4.53) to (2, 4.8) to (2, 5.07) to (1, 5.73) to (1, 6.0) to (1, 6.27) to (-1, 6.93) to (-1, 7.2) ;
        
      \end{knot} 
      
    \end{tikzpicture}  &
                         
                         \begin{tikzpicture}[xscale=0.5, yscale=0.5]
                           
                           \begin{knot}[
                             clip width=5,
                             clip radius=8pt,
                             ]
                             
                             \node at (-1, -0.5)  {\scalebox{0.5}{$j$}} ; 
                             \node at (0 , -0.5) {\scalebox{0.5}{$j+1$}} ; 
                             \node at (2 , -0.5) {\scalebox{0.5}{$j+2$}} ; 
                             \node at (-1, 5.3 )  {\scalebox{0.5}{$j$}} ; 
                             \node at (0 , 5.3) {\scalebox{0.5}{$j+1$}} ; 
                             \node at (1 , 5.3) {\scalebox{0.5}{$j+2$}} ; 
                             \node at (3 , 5.3) {\scalebox{0.5}{$j+3$}} ; 
                             \node at (-2.5, 2.65)  {\scalebox{1}{$\Longleftarrow$}} ;
                             \node at (-2.5, 3.4)  {\scalebox{1}{$\zeta$}} ;
                             Strand at (1, 4.8)
                             \strand[thick] (2, 3.87) to (1, 4.53) to (1, 4.8) ;
                             
                             Strand at (3, 4.8)
                             \strand[thick] (-1, 0.0) to (-1, 0.27) to (0, 0.93) to (0, 1.2) to (0, 1.47) to (2, 2.13) to (2, 2.4) to (2, 2.67) to (2, 3.33) to (2, 3.6) to (2, 3.87) to (3, 4.53) to (3, 4.8) ;
                             
                             Strand at (0, 4.8)
                             \strand[thick] (0, 0.0) to (0, 0.27) to (-1, 0.93) to (-1, 1.2) to (-1, 1.47) to (-1, 2.13) to (-1, 2.4) to (-1, 2.67) to (0, 3.33) to (0, 3.6) to (0, 3.87) to (0, 4.53) to (0, 4.8) ;
                             
                             Strand at (-1, 4.8)
                             \strand[thick] (2, 0.0) to (2, 0.27) to (2, 0.93) to (2, 1.2) to (2, 1.47) to (0, 2.13) to (0, 2.4) to (0, 2.67) to (-1, 3.33) to (-1, 3.6) to (-1, 3.87) to (-1, 4.53) to (-1, 4.8) ;
                             
                           \end{knot} 
                           
                         \end{tikzpicture} 
  \end{tabular}
\]
\item $i=j+1$. We have 
  \begin{align*}
    \zeta (\partial_{j+1}^n \otimes \chi_{j+1}^n \chi_j^n \chi_{j+1}^n) 
    % = & \chi_{j+2}^{n+1} \chi_{j+1}^{n+1} \otimes \zeta (\partial_{j+2}^n \otimes \chi_{j}^n) \otimes \chi_{j+1}^n \\ 
    % & =\chi_{j+2}^{n+1} \chi_{j+1}^{n+1} \otimes \chi_j^{n+1}  \ \otimes \zeta (\partial_{j+2}^n \otimes \chi_{j+1}^n) \\ 
    = & \chi_{j+2}^{n+1} \chi_{j+1}^{n+1} \chi_j^{n+1}  \chi_{j+1}^{n+1}  \chi_{j+2}^{n+1} \otimes \partial_{j+1}^n
  \end{align*}
  On the other hand we have 
  \begin{align*}
    \chi_{j+2}^{n+1} \chi_{j+1}^{n+1} \chi_j^{n+1}  \chi_{j+1}^{n+1}  \chi_{j+2}^{n+1} 
    = &  \chi_{j+2}^{n+1} \chi_{j}^{n+1} \chi_{j+1}^{n+1}  \chi_{j}^{n+1}  \chi_{j+2}^{n+1} \\
    = &  \chi_{j}^{n+1} \chi_{j+2}^{n+1} \chi_{j+1}^{n+1}  \chi_{j+2}^{n+1}  \chi_{j}^{n+1} 
        = \chi_{j}^{n+1} \chi_{j+1}^{n+1} \chi_{j+2}^{n+1} \chi_{j+1}^{n+1} \chi_{j}^{n+1}
  \end{align*}
  Thus
  \begin{align*}
    \chi_{j+2}^{n+1} \chi_{j+1}^{n+1} \chi_j^{n+1}  \chi_{j+1}^{n+1}  \chi_{j+2}^{n+1} \otimes \partial_{j+1}^n = & \chi_{j}^{n+1} \chi_{j+1}^{n+1} \chi_{j+2}^{n+1} \chi_{j+1}^{n+1} \chi_{j}^{n+1} \otimes \partial_{j+1}^n \\ = & \zeta (\partial_{j+1}^n \otimes \chi_{j}^n \chi_{j+1}^n \chi_{j}^n) 
    % = & \chi_{j}^{n+1} \chi_{j+1}^{n+1} \otimes \zeta (\partial_{j}^n \otimes \chi_{j+1}^n) \otimes \chi_{j}^n \\ 
    % & =\chi_{j}^{n+1} \chi_{j+1}^{n+1} \chi_{j+2}^{n+1} \otimes \zeta (\partial_{j}^n \otimes \chi_{j}^n) \\ 
  \end{align*}
  Diagrammatic interpretation is the following for this case:
  
  \[ \begin{tabular}{cccc}
    
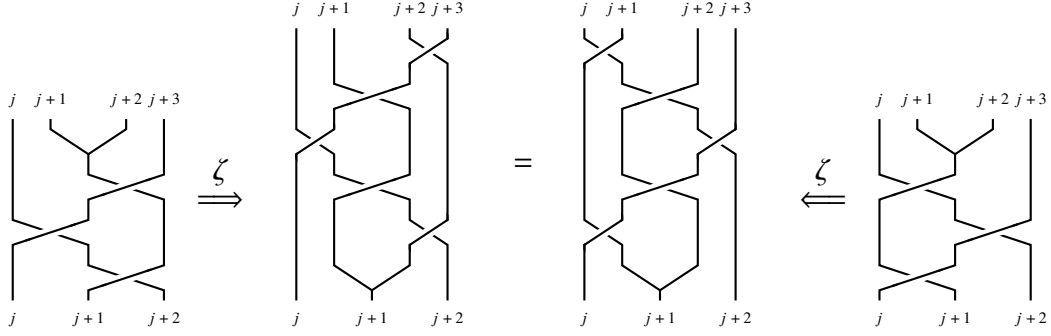
    \begin{tikzpicture}[xscale=0.5, yscale=0.5]
      
      \begin{knot}[
        clip width=5,
        clip radius=8pt,
        ]
        
        \node at (-1, -0.5)  {\scalebox{0.5}{$j$}} ; 
        \node at (1 , -0.5) {\scalebox{0.5}{$j+1$}} ; 
        \node at (3 , -0.5) {\scalebox{0.5}{$j+2$}} ; 
        \node at (-1, 5.3 )  {\scalebox{0.5}{$j$}} ; 
        \node at (0 , 5.3) {\scalebox{0.5}{$j+1$}} ; 
        \node at (2 , 5.3) {\scalebox{0.5}{$j+2$}} ; 
        \node at (3 , 5.3) {\scalebox{0.5}{$j+3$}} ; 
        \node at (4.5, 2.65)  {\scalebox{1}{$\Longrightarrow$}} ;
        \node at (4.5, 3.4)  {\scalebox{1}{$\zeta$}} ;
        Strand at (3, 4.8)
        \strand[thick] (-1, 0.0) to (-1, 0.27) to (-1, 0.93) to (-1, 1.2) to (-1, 1.47) to (1, 2.13) to (1, 2.4) to (1, 2.67) to (3, 3.33) to (3, 3.6) to (3, 3.87) to (3, 4.53) to (3, 4.8) ;
        
        Strand at (0, 4.8)
        \strand[thick] (1, 3.87) to (0, 4.53) to (0, 4.8) ;
        
        Strand at (2, 4.8)
        \strand[thick] (1, 0.0) to (1, 0.27) to (3, 0.93) to (3, 1.2) to (3, 1.47) to (3, 2.13) to (3, 2.4) to (3, 2.67) to (1, 3.33) to (1, 3.6) to (1, 3.87) to (2, 4.53) to (2, 4.8) ;
        
        Strand at (-1, 4.8)
        \strand[thick] (3, 0.0) to (3, 0.27) to (1, 0.93) to (1, 1.2) to (1, 1.47) to (-1, 2.13) to (-1, 2.4) to (-1, 2.67) to (-1, 3.33) to (-1, 3.6) to (-1, 3.87) to (-1, 4.53) to (-1, 4.8) ;
        
      \end{knot} 
    \end{tikzpicture}
    & \begin{tikzpicture}[xscale=0.5, yscale=0.5]
      
      \begin{knot}[
        clip width=5,
        clip radius=8pt,
        ]
        
        \node at (-1, -0.5)  {\scalebox{0.5}{$j$}} ; 
        \node at (1 , -0.5) {\scalebox{0.5}{$j+1$}} ; 
        \node at (3 , -0.5) {\scalebox{0.5}{$j+2$}} ; 
        \node at (-1, 7.7 )  {\scalebox{0.5}{$j$}} ; 
        \node at (0 , 7.7) {\scalebox{0.5}{$j+1$}} ; 
        \node at (2 , 7.7) {\scalebox{0.5}{$j+2$}} ; 
        \node at (3 , 7.7) {\scalebox{0.5}{$j+3$}} ; 
        \node at (5 , 3.6) {\scalebox{1}{$=$}} ;
        Strand at (3, 7.2)
        \strand[thick] (-1, 0.0) to (-1, 0.27) to (-1, 0.93) to (-1, 1.2) to (-1, 1.47) to (-1, 2.13) to (-1, 2.4) to (-1, 2.67) to (-1, 3.33) to (-1, 3.6) to (-1, 3.87) to (0, 4.53) to (0, 4.8) to (0, 5.07) to (2, 5.73) to (2, 6.0) to (2, 6.27) to (3, 6.93) to (3, 7.2) ;
        
        Strand at (0, 7.2)
        \strand[thick] (1, 0.27) to (0, 0.93) to (0, 1.2) to (0, 1.47) to (0, 2.13) to (0, 2.4) to (0, 2.67) to (2, 3.33) to (2, 3.6) to (2, 3.87) to (2, 4.53) to (2, 4.8) to (2, 5.07) to (0, 5.73) to (0, 6.0) to (0, 6.27) to (0, 6.93) to (0, 7.2) ;
        
        Strand at (2, 7.2)
        \strand[thick] (1, 0.0) to (1, 0.27) to (2, 0.93) to (2, 1.2) to (2, 1.47) to (3, 2.13) to (3, 2.4) to (3, 2.67) to (3, 3.33) to (3, 3.6) to (3, 3.87) to (3, 4.53) to (3, 4.8) to (3, 5.07) to (3, 5.73) to (3, 6.0) to (3, 6.27) to (2, 6.93) to (2, 7.2) ;
        
        Strand at (-1, 7.2)
        \strand[thick] (3, 0.0) to (3, 0.27) to (3, 0.93) to (3, 1.2) to (3, 1.47) to (2, 2.13) to (2, 2.4) to (2, 2.67) to (0, 3.33) to (0, 3.6) to (0, 3.87) to (-1, 4.53) to (-1, 4.8) to (-1, 5.07) to (-1, 5.73) to (-1, 6.0) to (-1, 6.27) to (-1, 6.93) to (-1, 7.2) ;
      \end{knot} 
    \end{tikzpicture}
    & \begin{tikzpicture}[xscale=0.5, yscale=0.5]
      
      \begin{knot}[
        clip width=5,
        clip radius=8pt,
        ]
        
        \node at (-1, -0.5)  {\scalebox{0.5}{$j$}} ; 
        \node at (1 , -0.5) {\scalebox{0.5}{$j+1$}} ; 
        \node at (3 , -0.5) {\scalebox{0.5}{$j+2$}} ; 
        \node at (-1, 7.7 )  {\scalebox{0.5}{$j$}} ; 
        \node at (0 , 7.7) {\scalebox{0.5}{$j+1$}} ; 
        \node at (2 , 7.7) {\scalebox{0.5}{$j+2$}} ; 
        \node at (3 , 7.7) {\scalebox{0.5}{$j+3$}} ; 
        Strand at (3, 7.2)
        \strand[thick] (-1, 0.0) to (-1, 0.27) to (-1, 0.93) to (-1, 1.2) to (-1, 1.47) to (0, 2.13) to (0, 2.4) to (0, 2.67) to (2, 3.33) to (2, 3.6) to (2, 3.87) to (3, 4.53) to (3, 4.8) to (3, 5.07) to (3, 5.73) to (3, 6.0) to (3, 6.27) to (3, 6.93) to (3, 7.2) ;
        
        Strand at (0, 7.2)
        \strand[thick] (1, 0.27) to (0, 0.93) to (0, 1.2) to (0, 1.47) to (-1, 2.13) to (-1, 2.4) to (-1, 2.67) to (-1, 3.33) to (-1, 3.6) to (-1, 3.87) to (-1, 4.53) to (-1, 4.8) to (-1, 5.07) to (-1, 5.73) to (-1, 6.0) to (-1, 6.27) to (0, 6.93) to (0, 7.2) ;
        
        Strand at (2, 7.2)
        \strand[thick] (1, 0.0) to (1, 0.27) to (2, 0.93) to (2, 1.2) to (2, 1.47) to (2, 2.13) to (2, 2.4) to (2, 2.67) to (0, 3.33) to (0, 3.6) to (0, 3.87) to (0, 4.53) to (0, 4.8) to (0, 5.07) to (2, 5.73) to (2, 6.0) to (2, 6.27) to (2, 6.93) to (2, 7.2) ;
        
        Strand at (-1, 7.2)
        \strand[thick] (3, 0.0) to (3, 0.27) to (3, 0.93) to (3, 1.2) to (3, 1.47) to (3, 2.13) to (3, 2.4) to (3, 2.67) to (3, 3.33) to (3, 3.6) to (3, 3.87) to (2, 4.53) to (2, 4.8) to (2, 5.07) to (0, 5.73) to (0, 6.0) to (0, 6.27) to (-1, 6.93) to (-1, 7.2) ;
        
      \end{knot} 
      
    \end{tikzpicture}
    & \begin{tikzpicture}[xscale=0.5, yscale=0.5]
      
      \begin{knot}[
        clip width=5,
        clip radius=8pt,
        ]
        
        \node at (-1, -0.5)  {\scalebox{0.5}{$j$}} ; 
        \node at (1 , -0.5) {\scalebox{0.5}{$j+1$}} ; 
        \node at (3 , -0.5) {\scalebox{0.5}{$j+2$}} ; 
        \node at (-1, 5.3 )  {\scalebox{0.5}{$j$}} ; 
        \node at (0 , 5.3) {\scalebox{0.5}{$j+1$}} ; 
        \node at (2 , 5.3) {\scalebox{0.5}{$j+2$}} ; 
        \node at (3 , 5.3) {\scalebox{0.5}{$j+3$}} ; 
        \node at (-2.5, 2.65)  {\scalebox{1}{$\Longleftarrow$}} ;
        \node at (-2.5, 3.4)  {\scalebox{1}{$\zeta$}} ;
        Strand at (3, 4.8)
        \strand[thick] (-1, 0.0) to (-1, 0.27) to (1, 0.93) to (1, 1.2) to (1, 1.47) to (3, 2.13) to (3, 2.4) to (3, 2.67) to (3, 3.33) to (3, 3.6) to (3, 3.87) to (3, 4.53) to (3, 4.8) ;
        
        Strand at (0, 4.8)
        \strand[thick] (1, 3.87) to (0, 4.53) to (0, 4.8) ;
        
        Strand at (2, 4.8)
        \strand[thick] (1, 0.0) to (1, 0.27) to (-1, 0.93) to (-1, 1.2) to (-1, 1.47) to (-1, 2.13) to (-1, 2.4) to (-1, 2.67) to (1, 3.33) to (1, 3.6) to (1, 3.87) to (2, 4.53) to (2, 4.8) ;
        
        Strand at (-1, 4.8)
        \strand[thick] (3, 0.0) to (3, 0.27) to (3, 0.93) to (3, 1.2) to (3, 1.47) to (1, 2.13) to (1, 2.4) to (1, 2.67) to (-1, 3.33) to (-1, 3.6) to (-1, 3.87) to (-1, 4.53) to (-1, 4.8) ;
        
      \end{knot} 
      
    \end{tikzpicture} 
  \end{tabular}
\]
		
\item $i=j$ 
  \begin{align*}
    \zeta (\partial_{j}^n \otimes \chi_{j+1}^n \chi_j^n \chi_{j+1}^n) 
    % = & \chi_{j+2}^{n+1} \otimes \zeta (\partial_{j}^n \otimes \chi_{j}^n) \otimes \chi_{j+1}^n \\ 
    % & =\chi_{j+2}^{n+1} \chi_{j+1}^{n+1} \otimes \chi_j^{n+1}  \ \otimes \zeta (\partial_{j+1}^n \otimes \chi_{j+1}^n) \\ 
    = & \chi_{j+2}^{n+1} \chi_{j+1}^{n+1} \chi_j^{n+1}  \chi_{j+2}^{n+1}  \chi_{j+1}^{n+1} \otimes \partial_{j+2}^n
  \end{align*}
  On the other hand we have
  \begin{align*}
    \chi_{j+2}^{n+1} \chi_{j+1}^{n+1} \chi_j^{n+1}  \chi_{j+2}^{n+1}  \chi_{j+1}^{n+1}
    = & \chi_{j+2}^{n+1} \chi_{j+1}^{n+1}  \chi_{j+2}^{n+1} \chi_j^{n+1} \chi_{j+1}^{n+1} 
        = \chi_{j+1}^{n+1} \chi_{j+2}^{n+1}  \chi_{j+1}^{n+1} \chi_j^{n+1} \chi_{j+1}^{n+1}   \\
    = & \chi_{j+1}^{n+1} \chi_{j+2}^{n+1}  \chi_{j}^{n+1} \chi_{j+1}^{n+1} \chi_{j}^{n+1} 
        = \chi_{j+1}^{n+1} \chi_{j}^n \chi_{j+2}^{n+1} \chi_{j+1}^{n+1} \chi_{j}^n
  \end{align*}
  Thus
  \begin{align*}
    \chi_{j+2}^{n+1} \chi_{j+1}^{n+1} \chi_j^{n+1}  \chi_{j+2}^{n+1}  \chi_{j+1}^{n+1} \otimes \partial_{j+2}^n
    = & \chi_{j+1}^{n+1} \chi_{j}^n \chi_{j+2}^{n+1} \chi_{j+1}^{n+1} \chi_{j}^n \otimes \partial_{j+2}^n \\
    = & \zeta (\partial_{j}^n \otimes \chi_{j}^n \chi_{j+1}^n \chi_{j}^n) 
  \end{align*}
  
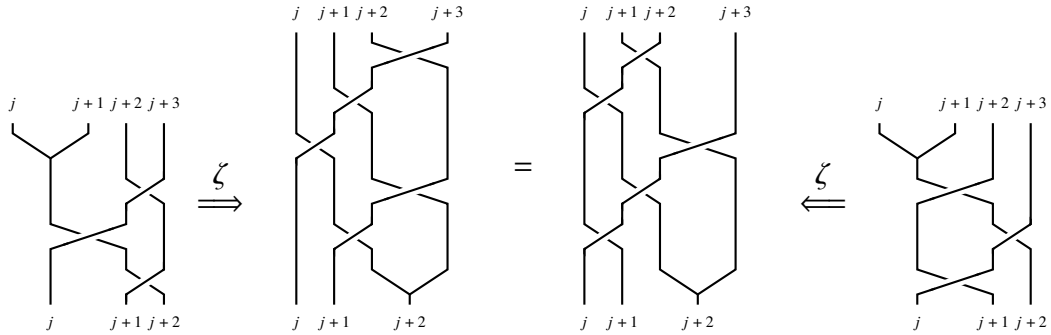
  \[ \begin{tabular}{c c c c}
    \begin{tikzpicture}[xscale=0.5, yscale=0.5]
      
      \begin{knot}[
        clip width=5,
        clip radius=8pt,
        ]
        
        \node at (0, -0.5)  {\scalebox{0.5}{$j$}} ; 
        \node at (2 , -0.5) {\scalebox{0.5}{$j+1$}} ; 
        \node at (3 , -0.5) {\scalebox{0.5}{$j+2$}} ; 
        \node at (-1, 5.3 )  {\scalebox{0.5}{$j$}} ; 
        \node at (1 , 5.3) {\scalebox{0.5}{$j+1$}} ; 
        \node at (2 , 5.3) {\scalebox{0.5}{$j+2$}} ; 
        \node at (3 , 5.3) {\scalebox{0.5}{$j+3$}} ; 
        \node at (4.5, 2.65)  {\scalebox{1}{$\Longrightarrow$}} ;
        \node at (4.5, 3.4)  {\scalebox{1}{$\zeta$}} ;
        Strand at (3, 4.8)
        \strand[thick] (0, 0.0) to (0, 0.27) to (0, 0.93) to (0, 1.2) to (0, 1.47) to (2, 2.13) to (2, 2.4) to (2, 2.67) to (3, 3.33) to (3, 3.6) to (3, 3.87) to (3, 4.53) to (3, 4.8) ;
        
        Strand at (2, 4.8)
        \strand[thick] (2, 0.0) to (2, 0.27) to (3, 0.93) to (3, 1.2) to (3, 1.47) to (3, 2.13) to (3, 2.4) to (3, 2.67) to (2, 3.33) to (2, 3.6) to (2, 3.87) to (2, 4.53) to (2, 4.8) ;
        
        Strand at (-1, 4.8)
        \strand[thick] (0, 3.87) to (-1, 4.53) to (-1, 4.8) ;
        
        Strand at (1, 4.8)
        \strand[thick] (3, 0.0) to (3, 0.27) to (2, 0.93) to (2, 1.2) to (2, 1.47) to (0, 2.13) to (0, 2.4) to (0, 2.67) to (0, 3.33) to (0, 3.6) to (0, 3.87) to (1, 4.53) to (1, 4.8) ;
        
      \end{knot} 					
    \end{tikzpicture}
    & \begin{tikzpicture}[xscale=0.5, yscale=0.5]
      
      \begin{knot}[
        clip width=5,
        clip radius=8pt,
        ]
        
        \node at (-1, -0.5)  {\scalebox{0.5}{$j$}} ; 
        \node at (0 , -0.5) {\scalebox{0.5}{$j+1$}} ; 
        \node at (2 , -0.5) {\scalebox{0.5}{$j+2$}} ; 
        \node at (-1, 7.7 )  {\scalebox{0.5}{$j$}} ; 
        \node at (0 , 7.7) {\scalebox{0.5}{$j+1$}} ; 
        \node at (1 , 7.7) {\scalebox{0.5}{$j+2$}} ; 
        \node at (3 , 7.7) {\scalebox{0.5}{$j+3$}} ; 
        \node at (5 , 3.6) {\scalebox{1}{$=$}} ;
        Strand at (3, 7.2)
        \strand[thick] (-1, 0.0) to (-1, 0.27) to (-1, 0.93) to (-1, 1.2) to (-1, 1.47) to (-1, 2.13) to (-1, 2.4) to (-1, 2.67) to (-1, 3.33) to (-1, 3.6) to (-1, 3.87) to (0, 4.53) to (0, 4.8) to (0, 5.07) to (1, 5.73) to (1, 6.0) to (1, 6.27) to (3, 6.93) to (3, 7.2) ;
        
        Strand at (1, 7.2)
        \strand[thick] (0, 0.0) to (0, 0.27) to (0, 0.93) to (0, 1.2) to (0, 1.47) to (1, 2.13) to (1, 2.4) to (1, 2.67) to (3, 3.33) to (3, 3.6) to (3, 3.87) to (3, 4.53) to (3, 4.8) to (3, 5.07) to (3, 5.73) to (3, 6.0) to (3, 6.27) to (1, 6.93) to (1, 7.2) ;
        
        Strand at (-1, 7.2)
        \strand[thick] (2, 0.27) to (1, 0.93) to (1, 1.2) to (1, 1.47) to (0, 2.13) to (0, 2.4) to (0, 2.67) to (0, 3.33) to (0, 3.6) to (0, 3.87) to (-1, 4.53) to (-1, 4.8) to (-1, 5.07) to (-1, 5.73) to (-1, 6.0) to (-1, 6.27) to (-1, 6.93) to (-1, 7.2) ;
        
        Strand at (0, 7.2)
        \strand[thick] (2, 0.0) to (2, 0.27) to (3, 0.93) to (3, 1.2) to (3, 1.47) to (3, 2.13) to (3, 2.4) to (3, 2.67) to (1, 3.33) to (1, 3.6) to (1, 3.87) to (1, 4.53) to (1, 4.8) to (1, 5.07) to (0, 5.73) to (0, 6.0) to (0, 6.27) to (0, 6.93) to (0, 7.2) ;
        
      \end{knot} 
    \end{tikzpicture}
    & \begin{tikzpicture}[xscale=0.5, yscale=0.5]
      
      \begin{knot}[
        clip width=5,
        clip radius=8pt,
        ]
        
        \node at (-1, -0.5)  {\scalebox{0.5}{$j$}} ; 
        \node at (0 , -0.5) {\scalebox{0.5}{$j+1$}} ; 
        \node at (2 , -0.5) {\scalebox{0.5}{$j+2$}} ; 
        \node at (-1, 7.7 )  {\scalebox{0.5}{$j$}} ; 
        \node at (0 , 7.7) {\scalebox{0.5}{$j+1$}} ; 
        \node at (1 , 7.7) {\scalebox{0.5}{$j+2$}} ; 
        \node at (3 , 7.7) {\scalebox{0.5}{$j+3$}} ; 
        Strand at (3, 7.2)
        \strand[thick] (-1, 0.0) to (-1, 0.27) to (-1, 0.93) to (-1, 1.2) to (-1, 1.47) to (0, 2.13) to (0, 2.4) to (0, 2.67) to (1, 3.33) to (1, 3.6) to (1, 3.87) to (3, 4.53) to (3, 4.8) to (3, 5.07) to (3, 5.73) to (3, 6.0) to (3, 6.27) to (3, 6.93) to (3, 7.2) ;
        
        Strand at (1, 7.2)
        \strand[thick] (0, 0.0) to (0, 0.27) to (0, 0.93) to (0, 1.2) to (0, 1.47) to (-1, 2.13) to (-1, 2.4) to (-1, 2.67) to (-1, 3.33) to (-1, 3.6) to (-1, 3.87) to (-1, 4.53) to (-1, 4.8) to (-1, 5.07) to (0, 5.73) to (0, 6.0) to (0, 6.27) to (1, 6.93) to (1, 7.2) ;
        
        Strand at (-1, 7.2)
        \strand[thick] (2, 0.27) to (1, 0.93) to (1, 1.2) to (1, 1.47) to (1, 2.13) to (1, 2.4) to (1, 2.67) to (0, 3.33) to (0, 3.6) to (0, 3.87) to (0, 4.53) to (0, 4.8) to (0, 5.07) to (-1, 5.73) to (-1, 6.0) to (-1, 6.27) to (-1, 6.93) to (-1, 7.2) ;
        
        Strand at (0, 7.2)
        \strand[thick] (2, 0.0) to (2, 0.27) to (3, 0.93) to (3, 1.2) to (3, 1.47) to (3, 2.13) to (3, 2.4) to (3, 2.67) to (3, 3.33) to (3, 3.6) to (3, 3.87) to (1, 4.53) to (1, 4.8) to (1, 5.07) to (1, 5.73) to (1, 6.0) to (1, 6.27) to (0, 6.93) to (0, 7.2) ;
        
      \end{knot} 
      
    \end{tikzpicture}  &
                         \begin{tikzpicture}[xscale=0.5, yscale=0.5]
                           
                           \begin{knot}[
                             clip width=5,
                             clip radius=8pt,
                             ]
                             
                             \node at (0, -0.5)  {\scalebox{0.5}{$j$}} ; 
                             \node at (2 , -0.5) {\scalebox{0.5}{$j+1$}} ; 
                             \node at (3 , -0.5) {\scalebox{0.5}{$j+2$}} ; 
                             \node at (-1, 5.3 )  {\scalebox{0.5}{$j$}} ; 
                             \node at (1 , 5.3) {\scalebox{0.5}{$j+1$}} ; 
                             \node at (2 , 5.3) {\scalebox{0.5}{$j+2$}} ; 
                             \node at (3 , 5.3) {\scalebox{0.5}{$j+3$}} ;
                             \node at (-2.5, 2.65)  {\scalebox{1}{$\Longleftarrow$}} ;
                             \node at (-2.5, 3.4)  {\scalebox{1}{$\zeta$}} ;
                             Strand at (3, 4.8)
                             \strand[thick] (0, 0.0) to (0, 0.27) to (2, 0.93) to (2, 1.2) to (2, 1.47) to (3, 2.13) to (3, 2.4) to (3, 2.67) to (3, 3.33) to (3, 3.6) to (3, 3.87) to (3, 4.53) to (3, 4.8) ;
                             
                             Strand at (2, 4.8)
                             \strand[thick] (2, 0.0) to (2, 0.27) to (0, 0.93) to (0, 1.2) to (0, 1.47) to (0, 2.13) to (0, 2.4) to (0, 2.67) to (2, 3.33) to (2, 3.6) to (2, 3.87) to (2, 4.53) to (2, 4.8) ;
                             
                             Strand at (-1, 4.8)
                             \strand[thick] (0, 3.87) to (-1, 4.53) to (-1, 4.8) ;
                             
                             Strand at (1, 4.8)
                             \strand[thick] (3, 0.0) to (3, 0.27) to (3, 0.93) to (3, 1.2) to (3, 1.47) to (2, 2.13) to (2, 2.4) to (2, 2.67) to (0, 3.33) to (0, 3.6) to (0, 3.87) to (1, 4.53) to (1, 4.8) ;
                             
                           \end{knot} 
                           
                         \end{tikzpicture} 
  \end{tabular}
\]
\end{enumerate}
	
We also need to check if $\zeta: T(\partial) \otimes_{\B{K}} T(\chi) \to T(\chi) \otimes_{\B{K}} T(\partial)$ extends to a right transposition for $\Mag$:
\[\begin{tikzcd}
  \Mag \otimes_{\B{K}} \Mag \otimes_{\B{K}} T(\chi) 
  \arrow[rr, "{(\Mag\otimes\zeta)}"]  
  \arrow[d, "{\mu_{\Mag}\otimes {T(\chi)}}"] 
  & & 
  \Mag \otimes_{\B{K}} T(\chi)\otimes_{\B{K}}\Mag
  \arrow[rr, "{\zeta\otimes {\Mag}}"]
  & &
  T(\chi)\otimes_{\B{K}}\Mag \otimes_{\B{K}}\Mag  
  \arrow[d, "{{T(\chi)}\otimes\mu_{\Mag}}"']
  \\
  \Mag \otimes_{\B{K}} T(\chi) 
  \arrow[rrrr, "\zeta"'] 
  & & & &
  T(\chi) \otimes_{\B{K}} \Mag 
\end{tikzcd}
\]
For this, in $T(\chi) \otimes_{\B{K}} \Mag$ we need to verify
\[ \zeta (\partial_i^{n+1} \partial_j^n\otimes \chi_k^n) = \zeta (\partial_{j+1}^{n+1}\partial_i^n \otimes \chi_{k}^n) \]
for all $0\leq i < j \leq n$ and $0\leq k \leq n-1$ where $n \geq 1$. We examine this in $7$ cases as follows:
\begin{enumerate}[{Case} (1):]  
\item $k<i-1<i<j$. Since  $j>k+1$, we also have $j+1>k+1$ which implies
  \begin{align*}
    \zeta (\partial_i^{n+1} \partial_j^n \otimes \chi_{k}^n)
    = & \chi_{k}^{n+2} \otimes \partial_i^{n+1} \partial_j^n \\
    = & \chi_{k}^{n+2} \otimes \partial_{j+1}^{n+1} \partial_i^n
        = \zeta (\partial_{j+1}^{n+1} \partial_i^n \otimes \chi_{k}^n)
  \end{align*}
  
\item $k=i-1<i<j$
  \begin{align*}
    \zeta\left(\partial_i^{n+1} \partial_j^n \otimes \chi_{i-1}^n\right)
    = & \chi_{i-1}^{n+2} \chi_i^{n+2} \otimes \partial_{i-1}^{n+1} \partial_j^n\\
    = & \chi_{i-1}^{n+2} \chi_i^{n+2} \otimes \partial_{j+1}^{n+1} \partial_{i-1}^n
        =\zeta\left(\partial_{j+1}^{n+1} \partial_i^n \otimes \chi_{i-1}^n\right)
  \end{align*}
  
\item $k=i< i+1 < j$
  \begin{align*}
    \zeta\left(\partial_i^{n+1} \partial_{j}^n \otimes \chi_i^n\right)
    = & \chi_{i+1}^{n+2} \chi_i^{n+2} \otimes \partial_{i+1}^{n+1} \partial_j^n  \\
    =  & \chi_{i+1}^{n+2} \chi_i^{n+2} \otimes \partial_{j+1}^{n+1} \partial_{i+1}^n
         =  \zeta\left(\partial_{j+1}^{n+1} \partial_{i}^n \otimes \chi_i^n\right)& 
  \end{align*}
  
\item $k=i<i+1=j$
  \begin{align*}
    \zeta\left(\partial_i^{n+1} \partial_{i+1}^n \otimes \chi_i^n\right)
    = & \chi_{i+1}^{n+2} \chi_i^{n+2} \chi_{i+2}^{n+2} \chi_{i+1}^{n+2} \otimes \partial_{i+2}^{n+1} \partial_i^n \\
    = & \chi_{i+1}^{n+2} \chi_{i+2}^{n+2} \chi_{i}^{n+2} \chi_{i+1}^{n+2} \otimes \partial_{i}^{n+1} \partial_{i+1}^n
        =  \zeta\left(\partial_{i+2}^{n+1} \partial_{i}^n \otimes \chi_i^n\right)& 
  \end{align*}
  The string diagrams for the equality above are as follows:
  \[ \begin{tabular}{c c c c}  
    \begin{tikzpicture}[xscale=0.5, yscale=0.5]
      
      \begin{knot}[
        clip width=5,
        clip radius=8pt,
        ]
        
        \node at (-1, -0.5)  {\scalebox{0.5}{$i$}} ; 
        \node at (2 , -0.5) {\scalebox{0.5}{$i+1$}} ; 
        \node at (-2, 4.1 )  {\scalebox{0.5}{$i$}} ; 
        \node at (0 , 4.1) {\scalebox{0.5}{$i+1$}} ; 
        \node at (1 , 4.1) {\scalebox{0.5}{$i+2$}} ; 
        \node at (3 , 4.1) {\scalebox{0.5}{$i+3$}} ; 
        \node at (4.5, 1.8)  {\scalebox{1}{$\Longrightarrow$}} ;
        \node at (4.5, 2.4)  {\scalebox{1}{$\zeta$}} ;
        
        Strand at (1, 3.6)
        \strand[thick] (2, 1.47) to (1, 2.13) to (1, 2.4) to (1, 2.67) to (1, 3.33) to (1, 3.6) ;
        
        Strand at (3, 3.6)
        \strand[thick] (-1, 0.0) to (-1, 0.27) to (2, 0.93) to (2, 1.2) to (2, 1.47) to (3, 2.13) to (3, 2.4) to (3, 2.67) to (3, 3.33) to (3, 3.6) ;
        
        Strand at (-2, 3.6)
        \strand[thick] (-1, 2.67) to (-2, 3.33) to (-2, 3.6) ;
        
        Strand at (0, 3.6)
        \strand[thick] (2, 0.0) to (2, 0.27) to (-1, 0.93) to (-1, 1.2) to (-1, 1.47) to (-1, 2.13) to (-1, 2.4) to (-1, 2.67) to (0, 3.33) to (0, 3.6) ;
      \end{knot} 
    \end{tikzpicture}
    &\begin{tikzpicture}[xscale=0.5, yscale=0.5]
      
      \begin{knot}[
        clip width=5,
        clip radius=8pt,
        ]
        
        \node at (-1, -0.5)  {\scalebox{0.5}{$i$}} ; 
        \node at (2 , -0.5) {\scalebox{0.5}{$i+1$}} ; 
        
        \node at (-2, 7.7 )  {\scalebox{0.5}{$i$}} ; 
        \node at (0 , 7.7) {\scalebox{0.5}{$i+1$}} ; 
        \node at (1 , 7.7) {\scalebox{0.5}{$i+2$}} ; 
        \node at (3 , 7.7) {\scalebox{0.5}{$i+3$}} ;
        \node at (4.5 , 3.6) {\scalebox{1}{$=$}} ; 
        
        Strand at (1, 7.2)
        \strand[thick] (-1, 0.27) to (-2, 0.93) to (-2, 1.2) to (-2, 1.47) to (-2, 2.13) to (-2, 2.4) to (-2, 2.67) to (-2, 3.33) to (-2, 3.6) to (-2, 3.87) to (-2, 4.53) to (-2, 4.8) to (-2, 5.07) to (0, 5.73) to (0, 6.0) to (0, 6.27) to (1, 6.93) to (1, 7.2) ;
        
        Strand at (3, 7.2)
        \strand[thick] (-1, 0.0) to (-1, 0.27) to (0, 0.93) to (0, 1.2) to (0, 1.47) to (0, 2.13) to (0, 2.4) to (0, 2.67) to (1, 3.33) to (1, 3.6) to (1, 3.87) to (3, 4.53) to (3, 4.8) to (3, 5.07) to (3, 5.73) to (3, 6.0) to (3, 6.27) to (3, 6.93) to (3, 7.2) ;
        
        Strand at (-2, 7.2)
        \strand[thick] (2, 1.47) to (1, 2.13) to (1, 2.4) to (1, 2.67) to (0, 3.33) to (0, 3.6) to (0, 3.87) to (0, 4.53) to (0, 4.8) to (0, 5.07) to (-2, 5.73) to (-2, 6.0) to (-2, 6.27) to (-2, 6.93) to (-2, 7.2) ;
        
        Strand at (0, 7.2)
        \strand[thick] (2, 0.0) to (2, 0.27) to (2, 0.93) to (2, 1.2) to (2, 1.47) to (3, 2.13) to (3, 2.4) to (3, 2.67) to (3, 3.33) to (3, 3.6) to (3, 3.87) to (1, 4.53) to (1, 4.8) to (1, 5.07) to (1, 5.73) to (1, 6.0) to (1, 6.27) to (0, 6.93) to (0, 7.2) ;
      \end{knot}  
    \end{tikzpicture}
    & \begin{tikzpicture}[xscale=0.5, yscale=0.5]
      \begin{knot}[
        clip width=5,
        clip radius=8pt,
        ]
        
        \node at (-1, -0.5)  {\scalebox{0.5}{$i$}} ; 
        \node at (2 , -0.5) {\scalebox{0.5}{$i+1$}} ;  
        \node at (-2, 7.7 )  {\scalebox{0.5}{$i$}} ; 
        \node at (0 , 7.7) {\scalebox{0.5}{$i+1$}} ; 
        \node at (1 , 7.7) {\scalebox{0.5}{$i+2$}} ; 
        \node at (3 , 7.7) {\scalebox{0.5}{$i+3$}} ;

        Strand at (1, 7.2)
        \strand[thick] (-1, 1.47) to (-2, 2.13) to (-2, 2.4) to (-2, 2.67) to (-2, 3.33) to (-2, 3.6) to (-2, 3.87) to (0, 4.53) to (0, 4.8) to (0, 5.07) to (0, 5.73) to (0, 6.0) to (0, 6.27) to (1, 6.93) to (1, 7.2) ;
        
        Strand at (3, 7.2)
        \strand[thick] (-1, 0.0) to (-1, 0.27) to (-1, 0.93) to (-1, 1.2) to (-1, 1.47) to (0, 2.13) to (0, 2.4) to (0, 2.67) to (1, 3.33) to (1, 3.6) to (1, 3.87) to (1, 4.53) to (1, 4.8) to (1, 5.07) to (3, 5.73) to (3, 6.0) to (3, 6.27) to (3, 6.93) to (3, 7.2) ;
        
        Strand at (-2, 7.2)
        \strand[thick] (2, 0.27) to (1, 0.93) to (1, 1.2) to (1, 1.47) to (1, 2.13) to (1, 2.4) to (1, 2.67) to (0, 3.33) to (0, 3.6) to (0, 3.87) to (-2, 4.53) to (-2, 4.8) to (-2, 5.07) to (-2, 5.73) to (-2, 6.0) to (-2, 6.27) to (-2, 6.93) to (-2, 7.2) ;
        
        Strand at (0, 7.2)
        \strand[thick] (2, 0.0) to (2, 0.27) to (3, 0.93) to (3, 1.2) to (3, 1.47) to (3, 2.13) to (3, 2.4) to (3, 2.67) to (3, 3.33) to (3, 3.6) to (3, 3.87) to (3, 4.53) to (3, 4.8) to (3, 5.07) to (1, 5.73) to (1, 6.0) to (1, 6.27) to (0, 6.93) to (0, 7.2) ;
      \end{knot} 
    \end{tikzpicture}
    &\begin{tikzpicture}[xscale=0.5, yscale=0.5]
      \begin{knot}[
        clip width=5,
        clip radius=8pt,
        ]
        
        \node at (-1, -0.5)  {\scalebox{0.5}{$i$}} ; 
        \node at (2 , -0.5) {\scalebox{0.5}{$i+1$}} ; 
        \node at (-2, 4.1 )  {\scalebox{0.5}{$i$}} ; 
        \node at (0 , 4.1) {\scalebox{0.5}{$i+1$}} ; 
        \node at (1 , 4.1) {\scalebox{0.5}{$i+2$}} ; 
        \node at (3 , 4.1) {\scalebox{0.5}{$i+3$}} ; 
        \node at (-3.5, 1.8)  {\scalebox{1}{$\Longleftarrow$}} ;
        \node at (-3.5, 2.4)  {\scalebox{1}{$\zeta$}} ;
        
        Strand at (1, 3.6)
        \strand[thick] (2, 2.67) to (1, 3.33) to (1, 3.6) ;
        
        Strand at (3, 3.6)
        \strand[thick] (-1, 0.0) to (-1, 0.27) to (2, 0.93) to (2, 1.2) to (2, 1.47) to (2, 2.13) to (2, 2.4) to (2, 2.67) to (3, 3.33) to (3, 3.6) ;
        
        Strand at (-2, 3.6)
        \strand[thick] (-1, 1.47) to (-2, 2.13) to (-2, 2.4) to (-2, 2.67) to (-2, 3.33) to (-2, 3.6) ;
        
        Strand at (0, 3.6)
        \strand[thick] (2, 0.0) to (2, 0.27) to (-1, 0.93) to (-1, 1.2) to (-1, 1.47) to (0, 2.13) to (0, 2.4) to (0, 2.67) to (0, 3.33) to (0, 3.6) ;
      \end{knot} 
    \end{tikzpicture} 
  \end{tabular}
\]
		
\item $i<k<j-1<j$
  \begin{align*}
    \zeta\left(\partial_i^{n+1} \partial_{j}^n \otimes \chi_{k}^n\right)
    = & \chi_{k+1}^{n+2} \otimes \partial_{i}^{n+1} \partial_{j}^n  \\
    = &  \chi_{k+1}^{n+2} \otimes \partial_{j+1}^{n+1} \partial_{i}^n
        =  \zeta\left(\partial_{j+1}^{n+1} \partial_{i}^n \otimes \chi_k^{n}\right)& 
  \end{align*}
  
\item $i<k=j-1<j$
  \begin{align*}
    \zeta\left(\partial_i^{n+1} \partial_{j}^n \otimes \chi_{j-1}^n\right)
    = & \chi_{j}^{n+2} \chi_{j+1}^{n+2} \otimes \partial_{i}^{n+1} \partial_{j-1}^n  \\
    = & \chi_{j}^{n+2} \chi_{j+1}^{n+2} \otimes \partial_{j}^{n+1} \partial_{i}^n
        =  \zeta\left(\partial_{j+1}^{n+1} \partial_{i}^n \otimes \chi_{j-1}^{n}\right)& 
  \end{align*}
  
\item $i<k=j$
  \begin{align*}
    \zeta\left(\partial_i^{n+1} \partial_{j}^n \otimes \chi_{j}^n\right)
    = & \chi_{j+2}^{n+2} \chi_{j+1}^{n+2} \otimes \partial_{i}^{n+1} \partial_{j+1}^n  \\
    = & \chi_{j+2}^{n+2} \chi_{j+1}^{n+2} \otimes \partial_{j+2}^{n+1} \partial_{i}^n
        =  \zeta\left(\partial_{j+1}^{n+1} \partial_{i}^n \otimes \chi_j^{n}\right)& 
  \end{align*}

\item $ i<j<k$.
  \begin{align*}
    \zeta (\partial_i^{n+1} \partial_j^n \otimes \chi_{k}^n)
    = & \chi_{k+2}^{n+2} \otimes \partial_i^{n+1} \partial_j^n\\
    = & \chi_{k+2}^{n+2} \otimes \partial_{j+1}^{n+1} \partial_i^n
        = \zeta (\partial_{j+1}^{n+1} \partial_i^n \otimes \chi_{k}^n)
  \end{align*}
\end{enumerate}
	
\subsection{$\Simp  \otimes_{\B{K}} \Braid \to \Braid  \otimes_{\B{K}} \Simp$}

We have already established that $\zeta: T(\partial) \otimes_{\B{K}} T(\chi) \to T(\chi) \otimes_{\B{K}} T(\partial)$ extends to a left transposition of the form $\zeta: T(\partial) \otimes_{\B{K}} \Braid \to \Braid \otimes_{\B{K}} T(\partial)$. We only need to show that $\zeta: T(\partial) \otimes_{\B{K}} T(\chi) \to T(\chi) \otimes_{\B{K}} T(\partial)$ also extends to a right transposition for $\Simp$,  $\zeta\colon \Simp \otimes_{\B{K}} T(\chi) \to T(\chi) \otimes_{\B{K}} \Simp$. Observe that most of the work is also done here because we showed that the distributive law $\zeta$ extends from ${\Mag}$ to $\Simp$. Also, note that an equivalence of elements in $\Mag$ implies their equivalence in $\Simp$.  The only remaining relation we check in $\mathcal{I}_{\Simp}$ is this:
\[ \zeta (\partial_i^{n+1} \partial_i^n\otimes \chi_k^n) = \zeta (\partial_{i+1}^{n+1}\partial_i^n \otimes \chi_{k}^n) \]
for all $0\leq i  \leq n$ and $0\leq k \leq n-1$ where $n \geq 1$. We examine this in $4$ cases as follows:

\begin{enumerate}[{Case}(1): ]
\item $i<k$. Since  $i<k$, we also have $i+1<k+1$ which implies
  \begin{align*}
    \zeta\left(\partial_i^{n+1} \partial_{i}^n \otimes \chi_{k}^n\right)
    = & \chi_{k+2}^{n+2}  \otimes \partial_{i}^{n+1} \partial_{i}^n  \\
    = & \chi_{k+2}^{n+2} \otimes \partial_{i+1}^{n+1} \partial_{i}^n 
        =  \zeta\left(\partial_{i+1}^{n+1} \partial_{i}^n \otimes \chi_k^{n}\right)
  \end{align*}
  
\item $i>k+1$ 
  \begin{align*}
    \zeta\left(\partial_i^{n+1} \partial_{i}^n \otimes \chi_{k}^n\right)
    = & \chi_{k}^{n+2}  \otimes \partial_{i}^{n+1} \partial_{i}^n \\
    = & \chi_{k}^{n+2} \otimes \partial_{i+1}^{n+1} \partial_{i}^n 
        = \zeta\left(\partial_{i+1}^{n+1} \partial_{i}^n \otimes \chi_k^{n}\right)& 
  \end{align*}
  
\item $i=k$
  \begin{align*}
    \zeta\left(\partial_i^{n+1} \partial_{i}^n \otimes \chi_{i}^n\right)
    = & \chi_{i+2}^{n+2}\chi_{i+1}^{n+2} \chi_{i}^{n+2}    \otimes \partial_{i+1}^{n+1} \partial_{i+1}^n  \\
    = & \chi_{i+2}^{n+2}\chi_{i+1}^{n+2} \chi_{i}^{n+2}    \otimes \partial_{i+2}^{n+1} \partial_{i+1}^n 
        = \zeta\left(\partial_{i+1}^{n+1} \partial_{i}^n \otimes \chi_i^{n}\right)
  \end{align*}
  
\[ \begin{tabular}{c c c c}
  \begin{tikzpicture}[xscale=0.5, yscale=0.5]
    \begin{knot}[
      clip width=5,
      clip radius=8pt,
      ]
      
      \node at (0, -0.5)  {\scalebox{0.5}{$i$}} ; 
      \node at (2 , -0.5) {\scalebox{0.5}{$i+1$}} ; 
      \node at (-2, 4.1 )  {\scalebox{0.5}{$i$}} ; 
      \node at (0 , 4.1) {\scalebox{0.5}{$i+1$}} ; 
      \node at (1 , 4.1) {\scalebox{0.5}{$i+2$}} ; 
      \node at (2 , 4.1) {\scalebox{0.5}{$i+3$}} ; 
      \node at (3.5, 1.8)  {\scalebox{1}{$\Longrightarrow$}} ;
      \node at (3.5, 2.4)  {\scalebox{1}{$\zeta$}} ;
      Strand at (2, 3.6)
      \strand[thick] (0, 0.0) to (0, 0.27) to (2, 0.93) to (2, 1.2) to (2, 1.47) to (2, 2.13) to (2, 2.4) to (2, 2.67) to (2, 3.33) to (2, 3.6) ;
      
      Strand at (-2, 3.6)
      \strand[thick] (-1, 2.67) to (-2, 3.33) to (-2, 3.6) ;
      
      Strand at (0, 3.6)
      \strand[thick] (0, 1.47) to (-1, 2.13) to (-1, 2.4) to (-1, 2.67) to (0, 3.33) to (0, 3.6) ;
      
      Strand at (1, 3.6)
      \strand[thick] (2, 0.0) to (2, 0.27) to (0, 0.93) to (0, 1.2) to (0, 1.47) to (1, 2.13) to (1, 2.4) to (1, 2.67) to (1, 3.33) to (1, 3.6) ;
    \end{knot} 
  \end{tikzpicture}
  & \begin{tikzpicture}[xscale=0.5, yscale=0.5]
    \begin{knot}[
      clip width=5,
      clip radius=8pt,
      ]
      
      \node at (-2, -0.5)  {\scalebox{0.5}{$i$}} ; 
      \node at (1 , -0.5) {\scalebox{0.5}{$i+1$}} ; 
      \node at (-2, 6.5 )  {\scalebox{0.5}{$i$}} ; 
      \node at (-1 , 6.5) {\scalebox{0.5}{$i+1$}} ; 
      \node at (1 , 6.5) {\scalebox{0.5}{$i+2$}} ; 
      \node at (2 , 6.5) {\scalebox{0.5}{$i+3$}} ;
      \node at (3.5 , 3) {\scalebox{1}{$=$}} ;
      Strand at (2, 6.0)
      \strand[thick] (-2, 0.0) to (-2, 0.27) to (-2, 0.93) to (-2, 1.2) to (-2, 1.47) to (-2, 2.13) to (-2, 2.4) to (-2, 2.67) to (-1, 3.33) to (-1, 3.6) to (-1, 3.87) to (1, 4.53) to (1, 4.8) to (1, 5.07) to (2, 5.73) to (2, 6.0) ;
      
      Strand at (-2, 6.0)
      \strand[thick] (0, 1.47) to (-1, 2.13) to (-1, 2.4) to (-1, 2.67) to (-2, 3.33) to (-2, 3.6) to (-2, 3.87) to (-2, 4.53) to (-2, 4.8) to (-2, 5.07) to (-2, 5.73) to (-2, 6.0) ;
      
      Strand at (-1, 6.0)
      \strand[thick] (1, 0.27) to (0, 0.93) to (0, 1.2) to (0, 1.47) to (1, 2.13) to (1, 2.4) to (1, 2.67) to (1, 3.33) to (1, 3.6) to (1, 3.87) to (-1, 4.53) to (-1, 4.8) to (-1, 5.07) to (-1, 5.73) to (-1, 6.0) ;
      
      Strand at (1, 6.0)
      \strand[thick] (1, 0.0) to (1, 0.27) to (2, 0.93) to (2, 1.2) to (2, 1.47) to (2, 2.13) to (2, 2.4) to (2, 2.67) to (2, 3.33) to (2, 3.6) to (2, 3.87) to (2, 4.53) to (2, 4.8) to (2, 5.07) to (1, 5.73) to (1, 6.0) ;
    \end{knot} 
  \end{tikzpicture}
  & \begin{tikzpicture}[xscale=0.5, yscale=0.5]
    \begin{knot}[
      clip width=5,
      clip radius=8pt,
      ]
      
      \node at (-1, -0.5)  {\scalebox{0.5}{$i$}} ; 
      \node at (1 , -0.5) {\scalebox{0.5}{$i+1$}} ; 
      \node at (-1, 6.5 )  {\scalebox{0.5}{$i$}} ; 
      \node at (0 , 6.5) {\scalebox{0.5}{$i+1$}} ; 
      \node at (1 , 6.5) {\scalebox{0.5}{$i+2$}} ; 
      \node at (3 , 6.5) {\scalebox{0.5}{$i+3$}} ; 
      Strand at (3, 6.0)
      \strand[thick] (-1, 0.0) to (-1, 0.27) to (-1, 0.93) to (-1, 1.2) to (-1, 1.47) to (-1, 2.13) to (-1, 2.4) to (-1, 2.67) to (0, 3.33) to (0, 3.6) to (0, 3.87) to (1, 4.53) to (1, 4.8) to (1, 5.07) to (3, 5.73) to (3, 6.0) ;
      
      Strand at (-1, 6.0)
      \strand[thick] (1, 0.27) to (0, 0.93) to (0, 1.2) to (0, 1.47) to (0, 2.13) to (0, 2.4) to (0, 2.67) to (-1, 3.33) to (-1, 3.6) to (-1, 3.87) to (-1, 4.53) to (-1, 4.8) to (-1, 5.07) to (-1, 5.73) to (-1, 6.0) ;
      
      Strand at (0, 6.0)
      \strand[thick] (2, 1.47) to (1, 2.13) to (1, 2.4) to (1, 2.67) to (1, 3.33) to (1, 3.6) to (1, 3.87) to (0, 4.53) to (0, 4.8) to (0, 5.07) to (0, 5.73) to (0, 6.0) ;
      
      Strand at (1, 6.0)
      \strand[thick] (1, 0.0) to (1, 0.27) to (2, 0.93) to (2, 1.2) to (2, 1.47) to (3, 2.13) to (3, 2.4) to (3, 2.67) to (3, 3.33) to (3, 3.6) to (3, 3.87) to (3, 4.53) to (3, 4.8) to (3, 5.07) to (1, 5.73) to (1, 6.0) ;
    \end{knot} 
  \end{tikzpicture} &
                      \begin{tikzpicture}[xscale=0.5, yscale=0.5]
                        \begin{knot}[
                          clip width=5,
                          clip radius=8pt,
                          ]
                          \node at (0, -0.5)  {\scalebox{0.5}{$i$}} ; 
                          \node at (3 , -0.5) {\scalebox{0.5}{$i+1$}} ; 
                          \node at (-1, 4.1 )  {\scalebox{0.5}{$i$}} ; 
                          \node at (0 , 4.1) {\scalebox{0.5}{$i+1$}} ; 
                          \node at (2 , 4.1) {\scalebox{0.5}{$i+2$}} ; 
                          \node at (3 , 4.1) {\scalebox{0.5}{$i+3$}} ; 
                          \node at (-2.5, 1.8)  {\scalebox{1}{$\Longleftarrow$}} ;
                          \node at (-2.5, 2.4)  {\scalebox{1}{$\zeta$}} ;
                          Strand at (3, 3.6)
                          \strand[thick] (0, 0.0) to (0, 0.27) to (3, 0.93) to (3, 1.2) to (3, 1.47) to (3, 2.13) to (3, 2.4) to (3, 2.67) to (3, 3.33) to (3, 3.6) ;
                          
                          Strand at (-1, 3.6)
                          \strand[thick] (0, 1.47) to (-1, 2.13) to (-1, 2.4) to (-1, 2.67) to (-1, 3.33) to (-1, 3.6) ;
                          
                          Strand at (0, 3.6)
                          \strand[thick] (1, 2.67) to (0, 3.33) to (0, 3.6) ;
                          
                          Strand at (2, 3.6)
                          \strand[thick] (3, 0.0) to (3, 0.27) to (0, 0.93) to (0, 1.2) to (0, 1.47) to (1, 2.13) to (1, 2.4) to (1, 2.67) to (2, 3.33) to (2, 3.6) ;
                        \end{knot} 
                      \end{tikzpicture} 
\end{tabular}
\]
		
\item $i=k+1$
  \begin{align*}
    \zeta\left(\partial_i^{n+1} \partial_{i}^n \otimes \chi_{i-1}^n\right)
    = & \chi_{i-1}^{n+2}\chi_{i}^{n+2} \chi_{i+1}^{n+2}    \otimes \partial_{i-1}^{n+1} \partial_{i-1}^n  \\
    = & \chi_{i-1}^{n+2}\chi_{i}^{n+2} \chi_{i+1}^{n+2}    \otimes \partial_{i}^{n+1} \partial_{i-1}^n 
        = \zeta\left(\partial_{i+1}^{n+1} \partial_{i}^n \otimes \chi_{i-1}^{n}\right)
  \end{align*}
  
  \[\begin{tabular}{c c c c}
      \begin{tikzpicture}[xscale=0.5, yscale=0.5]
        \begin{knot}[
          clip width=5,
          clip radius=8pt,
          ]
          \node at (-2, -0.5)  {\scalebox{0.5}{$i-1$}} ; 
          \node at (1 , -0.5) {\scalebox{0.5}{$i$}} ; 
          \node at (-2, 4.1 )  {\scalebox{0.5}{$i-1$}} ; 
          \node at (-1 , 4.1) {\scalebox{0.5}{$i$}} ; 
          \node at (1 , 4.1) {\scalebox{0.5}{$i+1$}} ; 
          \node at (2 , 4.1) {\scalebox{0.5}{$i+2$}} ;
          \node at (3.5, 1.8)  {\scalebox{1}{$\Longrightarrow$}} ;
          \node at (3.5, 2.4)  {\scalebox{1}{$\zeta$}} ;
          
          Strand at (-1, 3.6)
          \strand[thick] (0, 2.67) to (-1, 3.33) to (-1, 3.6) ;
          
          Strand at (1, 3.6)
          \strand[thick] (1, 1.47) to (0, 2.13) to (0, 2.4) to (0, 2.67) to (1, 3.33) to (1, 3.6) ;
          
          Strand at (2, 3.6)
          \strand[thick] (-2, 0.0) to (-2, 0.27) to (1, 0.93) to (1, 1.2) to (1, 1.47) to (2, 2.13) to (2, 2.4) to (2, 2.67) to (2, 3.33) to (2, 3.6) ;
          
          Strand at (-2, 3.6)
          \strand[thick] (1, 0.0) to (1, 0.27) to (-2, 0.93) to (-2, 1.2) to (-2, 1.47) to (-2, 2.13) to (-2, 2.4) to (-2, 2.67) to (-2, 3.33) to (-2, 3.6) ;
          
        \end{knot} 
      \end{tikzpicture}
      & \begin{tikzpicture}[xscale=0.5, yscale=0.5]
        \begin{knot}[
          clip width=5,
          clip radius=8pt,
          ]
          \node at (0, -0.5)  {\scalebox{0.5}{$i-1$}} ; 
          \node at (2 , -0.5) {\scalebox{0.5}{$i$}} ; 
          \node at (-2, 6.5 )  {\scalebox{0.5}{$i-1$}} ; 
          \node at (0 , 6.5) {\scalebox{0.5}{$i$}} ; 
          \node at (1 , 6.5) {\scalebox{0.5}{$i+1$}} ; 
          \node at (2 , 6.5) {\scalebox{0.5}{$i+2$}} ; 
          \node at (3.5 , 3) {\scalebox{1}{$=$}} ;
          Strand at (0, 6.0)
          \strand[thick] (-1, 1.47) to (-2, 2.13) to (-2, 2.4) to (-2, 2.67) to (-2, 3.33) to (-2, 3.6) to (-2, 3.87) to (-2, 4.53) to (-2, 4.8) to (-2, 5.07) to (0, 5.73) to (0, 6.0) ;
          
          Strand at (1, 6.0)
          \strand[thick] (0, 0.27) to (-1, 0.93) to (-1, 1.2) to (-1, 1.47) to (0, 2.13) to (0, 2.4) to (0, 2.67) to (0, 3.33) to (0, 3.6) to (0, 3.87) to (1, 4.53) to (1, 4.8) to (1, 5.07) to (1, 5.73) to (1, 6.0) ;
          
          Strand at (2, 6.0)
          \strand[thick] (0, 0.0) to (0, 0.27) to (1, 0.93) to (1, 1.2) to (1, 1.47) to (1, 2.13) to (1, 2.4) to (1, 2.67) to (2, 3.33) to (2, 3.6) to (2, 3.87) to (2, 4.53) to (2, 4.8) to (2, 5.07) to (2, 5.73) to (2, 6.0) ;
          
          Strand at (-2, 6.0)
          \strand[thick] (2, 0.0) to (2, 0.27) to (2, 0.93) to (2, 1.2) to (2, 1.47) to (2, 2.13) to (2, 2.4) to (2, 2.67) to (1, 3.33) to (1, 3.6) to (1, 3.87) to (0, 4.53) to (0, 4.8) to (0, 5.07) to (-2, 5.73) to (-2, 6.0) ;
        \end{knot} 
      \end{tikzpicture}
      & \begin{tikzpicture}[xscale=0.5, yscale=0.5]
        \begin{knot}[
          clip width=5,
          clip radius=8pt,
          ]
          
          \node at (0, -0.5)  {\scalebox{0.5}{$i-1$}} ; 
          \node at (3 , -0.5) {\scalebox{0.5}{$i$}} ; 
          \node at (-1, 6.5 )  {\scalebox{0.5}{$i-1$}} ; 
          \node at (0 , 6.5) {\scalebox{0.5}{$i$}} ; 
          \node at (2 , 6.5) {\scalebox{0.5}{$i+1$}} ; 
          \node at (3 , 6.5) {\scalebox{0.5}{$i+2$}} ; 
          Strand at (0, 6.0)
          \strand[thick] (0, 0.27) to (-1, 0.93) to (-1, 1.2) to (-1, 1.47) to (-1, 2.13) to (-1, 2.4) to (-1, 2.67) to (-1, 3.33) to (-1, 3.6) to (-1, 3.87) to (-1, 4.53) to (-1, 4.8) to (-1, 5.07) to (0, 5.73) to (0, 6.0) ;
          
          Strand at (2, 6.0)
          \strand[thick] (1, 1.47) to (0, 2.13) to (0, 2.4) to (0, 2.67) to (0, 3.33) to (0, 3.6) to (0, 3.87) to (2, 4.53) to (2, 4.8) to (2, 5.07) to (2, 5.73) to (2, 6.0) ;
          
          Strand at (3, 6.0)
          \strand[thick] (0, 0.0) to (0, 0.27) to (1, 0.93) to (1, 1.2) to (1, 1.47) to (2, 2.13) to (2, 2.4) to (2, 2.67) to (3, 3.33) to (3, 3.6) to (3, 3.87) to (3, 4.53) to (3, 4.8) to (3, 5.07) to (3, 5.73) to (3, 6.0) ;
          
          Strand at (-1, 6.0)
          \strand[thick] (3, 0.0) to (3, 0.27) to (3, 0.93) to (3, 1.2) to (3, 1.47) to (3, 2.13) to (3, 2.4) to (3, 2.67) to (2, 3.33) to (2, 3.6) to (2, 3.87) to (0, 4.53) to (0, 4.8) to (0, 5.07) to (-1, 5.73) to (-1, 6.0) ;
        \end{knot} 
      \end{tikzpicture}
    & \begin{tikzpicture}[xscale=0.5, yscale=0.5]
      \begin{knot}[
        clip width=5,
        clip radius=8pt,
        ]
        
        \node at (-1, -0.5)  {\scalebox{0.5}{$i-1$}} ; 
        \node at (1 , -0.5) {\scalebox{0.5}{$i$}} ; 
        \node at (-1, 4.1 )  {\scalebox{0.5}{$i-1$}} ; 
        \node at (0 , 4.1) {\scalebox{0.5}{$i$}} ; 
        \node at (1 , 4.1) {\scalebox{0.5}{$i+1$}} ; 
        \node at (3 , 4.1) {\scalebox{0.5}{$i+2$}} ;
        \node at (-2.5, 1.8)  {\scalebox{1}{$\Longleftarrow$}} ;
        \node at (-2.5, 2.4)  {\scalebox{1}{$\zeta$}} ;
        Strand at (0, 3.6)
        \strand[thick] (1, 1.47) to (0, 2.13) to (0, 2.4) to (0, 2.67) to (0, 3.33) to (0, 3.6) ;
        
        Strand at (1, 3.6)
        \strand[thick] (2, 2.67) to (1, 3.33) to (1, 3.6) ;
        
        Strand at (3, 3.6)
        \strand[thick] (-1, 0.0) to (-1, 0.27) to (1, 0.93) to (1, 1.2) to (1, 1.47) to (2, 2.13) to (2, 2.4) to (2, 2.67) to (3, 3.33) to (3, 3.6) ;
        
        Strand at (-1, 3.6)
        \strand[thick] (1, 0.0) to (1, 0.27) to (-1, 0.93) to (-1, 1.2) to (-1, 1.47) to (-1, 2.13) to (-1, 2.4) to (-1, 2.67) to (-1, 3.33) to (-1, 3.6) ;
      \end{knot} 
    \end{tikzpicture} 
  \end{tabular}
\]
\end{enumerate}
	
\subsection{$\Mag  \otimes_{\B{K}} \Sym \to \Sym  \otimes_{\B{K}} \Mag$} Similar to the previous case, most of the things we need to show are done in Section~\ref{subsect:MagToBraid}.  The only remaining relation we check in $\mathcal{I}_{\Sym}$ is 
\[ \zeta ( \partial_i^n\otimes \chi_j^n\chi_j^n) = \zeta (\partial_i^n\otimes 1_{n}) \]
for all $0\leq i \leq n$ and $0\leq j \leq n-1$ where $n \geq 1$. We examine this in $4$ cases as follows:

\begin{enumerate}[{Case}(1): ]
\item $i<j$
  \begin{align*}
    \zeta\left( \partial_i^n\otimes \chi_j^n\chi_j^n\right)
    = & \chi_{j+1}^{n+1} \chi_{j+1}^{n+1}  \otimes \partial_{i}^{n} \\
    = & 1_{n+1} \otimes \partial_{i}^{n} 
        =  \zeta\left(\partial_{i}^{n} \otimes 1_{n}\right)
  \end{align*}
  
\item $i>j+1$
  \begin{align*}
    \zeta\left( \partial_i^n\otimes \chi_j^n\chi_j^n\right)
    = & \chi_{j}^{n+1} \chi_{j}^{n+1}  \otimes \partial_{i}^{n}\\ 
    = & 1_{n+1} \otimes \partial_{i}^{n} 
        =  \zeta\left(\partial_{i}^{n} \otimes 1_{n}\right)
  \end{align*}
  
\item $i=j$
  \begin{align*}
    \zeta\left( \partial_i^n\otimes \chi_i^n\chi_i^n\right)
    = & \chi_{i+1}^{n+1} \chi_{i}^{n+1} \chi_{i}^{n+1} \chi_{i+1}^{n+1}  \otimes \partial_{i}^{n} 
        = \chi_{i+1}^{n+1}\chi_{i+1}^{n+1} \otimes \partial_{i}^{n}\\
    = & 1_{n+1} \otimes \partial_{i}^{n} 
        =  \zeta\left(\partial_{i}^{n} \otimes 1_{n}\right)
  \end{align*}
  
\item $i=j+1$
  \begin{align*}
    \zeta\left( \partial_i^n\otimes \chi_{i-1}^n\chi_{i-1}^n\right)
    = & \chi_{i-1}^{n+1} \chi_{i}^{n+1} \chi_{i}^{n+1} \chi_{i-1}^{n+1}  \otimes \partial_{i}^{n} 
        = \chi_{i-1}^{n+1}\chi_{i-1}^{n+1} \otimes \partial_{i}^{n}\\
    = & 1_{n+1} \otimes \partial_{i}^{n} 
        =  \zeta\left(\partial_{i}^{n} \otimes 1_{n}\right)
  \end{align*}
\end{enumerate}
	
\subsection{$\Simp  \otimes_{\B{K}} \Sym \to \Sym  \otimes_{\B{K}} \Simp$} 

We have already established that $\zeta: T(\partial) \otimes_{\B{K}} T(\chi) \to T(\chi) \otimes_{\B{K}} T(\partial)$ extends to both left transposition of the form $\zeta: T(\partial) \otimes_{\B{K}} \Sym \to \Sym \otimes_{\B{K}} T(\partial)$ and right transposition $\zeta: \Simp \otimes_{\B{K}} T(\chi) \to T(\chi) \otimes_{\B{K}} \Simp$. By Lemma~\ref{lem:distributive}, $\zeta$ induces a distributive law $\Simp  \otimes_{\B{K}} \Sym \to \Sym  \otimes_{\B{K}} \Simp$ as well.

\section{The Leibniz $\B{K}$-algebra  }\label{sect:LeibnizAlgebras}

From here on, by abuse of notation, we will write  $\partial_j^n$ for $1_{n+1}\otimes \partial_j^n$ and similarly $\chi_{j}^{n}$ for $\chi_{j}^{n} \otimes 1_{n}$ in $\Sym\otimes_{\zeta}\Mag$. Also, we are going to explicitly use the distributive law $\zeta$ between $\Sym$ and $\Mag$ for the sake of simplicity. However, the results of this section are valid for the braided case at no extra cost.

\subsection{$\Leib$ and $\Leib^{op}$}

Recall that $\Sym \otimes_{\zeta} \Mag $ is a $\B{K}$-algebra where the multiplication is determined by the distributive law $\zeta$ as described in Lemma~\ref{lem:CanonicalDistributiveLaw}. We define two ideals of
$\Sym \otimes_{\zeta} \Mag$ as follows:
\begin{equation}\label{eq:LeibRelations}
  I_{\Leib} := \left\langle \partial_{j+1}^{n+1} \partial_j^n-\left(1_{n+2} - \chi_{j+1}^{n+2}\right) \partial_j^{n+1} \partial_j^n\mid 0 \leq j \leq n,\ n\geq 0 \right\rangle
\end{equation}
and 
\begin{equation}\label{eq:LeibOpRelations}
  I_{\Leib^{op}} := \left\langle \partial_{j}^{n+1} \partial_j^n-\left(1_{n+2} - \chi_{j+1}^{n+2}\right) \partial_{j+1}^{n+1} \partial_j^n\mid 0 \leq j \leq n,\ n\geq 0 \right\rangle
\end{equation}
Then we define $\B{K}$-algebras $\Leib$ and $\Leib^{op}$ as the quotients $(\Sym \otimes_{\zeta} \Mag )/I_{\Leib }$ and $(\Sym \otimes_{\zeta} \Mag )/I_{\Leib^{op}}$, respectively. These $\B{K}$-algebras are called \emph{left Leibniz} and \emph{right Leibniz} $\B{K}$-algebras, respectively.

\begin{remark}
The element of the $I_{\Leib}$ with the smallest superscript $n=0$ is represented by the following string diagram:
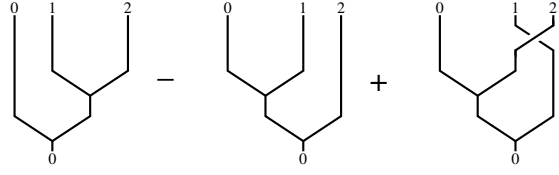
\[\begin{tabular}{c c c}
  \begin{tikzpicture}[xscale=0.5, yscale=0.5]
    
    \begin{knot}[
      clip width=5,
      clip radius=8pt,
      ]
      
      \node at (0, -0.2)  {\scalebox{0.5}{$0$}} ; 
      \node at (-1, 3.8)  {\scalebox{0.5}{$0$}} ; 
      \node at (0 , 3.8) {\scalebox{0.5}{$1$}} ; 
      \node at (2 , 3.8) {\scalebox{0.5}{$2$}} ; 
      \node at (3 , 1.8) {\scalebox{1}{$-$}} ;
      % Strand at (-1, 2.4)
      \strand[thick] (0, 0.27) to (-1, 0.93) to (-1, 1.2) to (-1, 1.47) to (-1, 2.13) to (-1, 3.6) ;
      
      % Strand at (0, 2.4)
      \strand[thick] (1, 1.47) to (0, 2.13) to (0, 3.6) ;
      
      % Strand at (2, 2.4)
      \strand[thick] (0, 0.0) to (0, 0.27) to (1, 0.93) to (1, 1.2) to (1, 1.47) to (2, 2.13) to (2, 3.6) ;
      
    \end{knot} 
    
  \end{tikzpicture} 
  &\begin{tikzpicture}[xscale=0.5, yscale=0.5]
    
    \begin{knot}[
      clip width=5,
      clip radius=8pt,
      ]
      
      \node at (0, -0.2)  {\scalebox{0.5}{$0$}} ; 
      \node at (-2, 3.8)  {\scalebox{0.5}{$0$}} ; 
      \node at (0 , 3.8) {\scalebox{0.5}{$1$}} ; 
      \node at (1 , 3.8) {\scalebox{0.5}{$2$}} ;
      \node at (2 , 1.8) {\scalebox{1}{$+$}} ; 
      % Strand at (-2, 2.4)
      \strand[thick] (-1, 1.47) to (-2, 2.13) to (-2, 3.6) ;
      
      % Strand at (0, 2.4)
      \strand[thick] (0, 0.27) to (-1, 0.93) to (-1, 1.2) to (-1, 1.47) to (0, 2.13) to (0, 3.6) ;
      
      % Strand at (1, 2.4)
      \strand[thick] (0, 0.0) to (0, 0.27) to (1, 0.93) to (1, 1.2) to (1, 1.47) to (1, 2.13) to (1, 3.6) ;
      
    \end{knot} 
    
  \end{tikzpicture} 
  &\begin{tikzpicture}[xscale=0.5, yscale=0.5]
    
    \begin{knot}[
      clip width=5,
      clip radius=8pt,
      ]
      
      \node at (0, -0.2)  {\scalebox{0.5}{$0$}} ; 
      \node at (-2, 3.8)  {\scalebox{0.5}{$0$}} ; 
      \node at (0 , 3.8) {\scalebox{0.5}{$1$}} ; 
      \node at (1 , 3.8) {\scalebox{0.5}{$2$}} ; 
      % Strand at (-2, 3.6)
      \strand[thick] (-1, 1.47) to (-2, 2.13) to (-2, 2.4) to (-2, 2.67) to (-2, 3.33) to (-2, 3.6) ;
      
      % Strand at (1, 3.6)
      \strand[thick] (0, 0.27) to (-1, 0.93) to (-1, 1.2) to (-1, 1.47) to (0, 2.13) to (0, 2.4) to (0, 2.67) to (1, 3.33) to (1, 3.6) ;
      
      % Strand at (0, 3.6)
      \strand[thick] (0, 0.0) to (0, 0.27) to (1, 0.93) to (1, 1.2) to (1, 1.47) to (1, 2.13) to (1, 2.4) to (1, 2.67) to (0, 3.33) to (0, 3.6) ;
      
    \end{knot} 
    
  \end{tikzpicture} 
\end{tabular}
\]
Elementwise this string diagram corresponds to the (left)Leibniz identity $[x,[y,z]]-[[x,y],z]+[[x,z],y]=0$~\cite{Loday_1993}.
\end{remark}

\begin{proposition}\label{prop:automorphism}
  Consider the $\B{K}$-bimodule automorphism
  $\alpha\colon \Sym\otimes_\zeta\Mag\to\Sym\otimes_\zeta\Mag$ of
 defined on the generators by
  \[ \alpha(1_n) = 1_n, \quad \alpha(\brd{n}{j}) = \brd{n}{j}
    \quad\text{ and }\quad
    \alpha(\partial^n_j) = \brd{n+1}{j}\partial^n_j
  \]
  for any $n\geq 0$ and $0\leq j\leq n$.  Then $\alpha$ extends to an
  automorphism of $\B{K}$-algebras.
\end{proposition}

\begin{proof}
   If you would like to follow the equations below by drawing corresponding string diagrams for $\alpha$, the nontrivial parts depicted after relevant calculations. The diagram for the nonidentity part of $\alpha$ on the generators can be seen as follows:
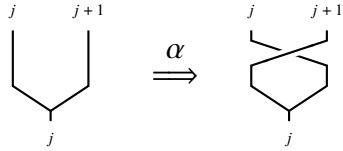
\[\begin{tabular}{c c}
  	
    \begin{tikzpicture}[xscale=0.5, yscale=0.5]
      \begin{knot}[
        clip width=5,
        clip radius=8pt,
        ]
        \node at (0, -0.5)  {\scalebox{0.5}{$j$}} ; 
        \node at (-1, 2.9 )  {\scalebox{0.5}{$j$}} ; 
        \node at (1 , 2.9) {\scalebox{0.5}{$j+1$}} ; 
        % Strand at (1, 2.4)
        \strand[thick] (0, 0.27) to (-1, 0.93) to (-1, 1.2) to (-1, 1.47) to (-1, 2.13) to (-1, 2.4) ;
        
        % Strand at (-1, 2.4)
        \strand[thick] (0, 0.0) to (0, 0.27) to (1, 0.93) to (1, 1.2) to (1, 1.47) to (1, 2.13) to (1, 2.4) ;
        
      \end{knot}  
    \end{tikzpicture}
    &\begin{tikzpicture}[xscale=0.5, yscale=0.5]
      \begin{knot}[
        clip width=5,
        clip radius=8pt,
        ]
        \node at (-3, 1.85 )  {\scalebox{1}{$\alpha$}} ;
        \node at (-3, 1.2 )  {\scalebox{1}{$\Longrightarrow$}} ;
        \node at (0, -0.5)  {\scalebox{0.5}{$j$}} ; 
        \node at (-1, 2.9 )  {\scalebox{0.5}{$j$}} ; 
        \node at (1 , 2.9) {\scalebox{0.5}{$j+1$}} ; 
        % Strand at (1, 2.4)
        \strand[thick] (0, 0.27) to (-1, 0.93) to (-1, 1.2) to (-1, 1.47) to (1, 2.13) to (1, 2.4) ;
        
        % Strand at (-1, 2.4)
        \strand[thick] (0, 0.0) to (0, 0.27) to (1, 0.93) to (1, 1.2) to (1, 1.47) to (-1, 2.13) to (-1, 2.4) ;
        
      \end{knot} 
      
    \end{tikzpicture}
  \end{tabular} 
\]
We must prove that $\alpha$ preserves the relations in
$\Sym\otimes_\zeta\Mag$. We start with the relations in $\Mag$:
for $i<j$ we obtain
\begin{align*}
  \alpha(\partial^{n+1}_i\partial^n_j)
  % = & \alpha(\partial^{n+1}_i)\alpha(\partial^n_j)
    = & \brd{n+2}{i}\partial^{n+1}_i\brd{n+1}{j}\partial^n_j
    = \brd{n+2}{i}\brd{n+2}{j+1}\partial^{n+1}_i\partial^n_j\\
    = & \brd{n+2}{j+1}\brd{n+2}{i}\partial^{n+1}_{j+1}\partial^n_i
    = \brd{n+2}{j+1}\partial^{n+1}_{j+1}\brd{n+1}{i}\partial^n_i\\
    = & \alpha(\partial^{n+1}_{j+1}\partial^n_i)
\end{align*}
by using the distributive law $\zeta$.  As for the interaction
between $\brd{n}{i}$ and $\partial^n_j$, we consider 4 different
cases: 
\begin{enumerate}[{Case} (1)]
\item $i<j$.
  \begin{align*}
    \alpha(\partial^n_i\brd{n}{j})
    = & \brd{n+1}{i}\partial^n_i\brd{n}{j}
        = \brd{n+1}{i}\brd{n}{j+1}\partial^n_i\\
    = & \brd{n+1}{j+1}\brd{n}{i}\partial^n_i
        = \alpha(\brd{n+1}{j+1}\partial^n_i)
  \end{align*}
\item $i=j$.
  \begin{align*}
    \alpha(\partial^n_i\brd{n}{i})
    = & \brd{n+1}{i}\partial^n_i\brd{n}{i}
        = \brd{n+1}{i}\brd{n+1}{i+1}\brd{n+1}{i}\partial^n_{i+1}\\
    = & \brd{n+1}{i+1}\brd{n+1}{i}\brd{n+1}{i+1}\partial^n_{i+1}
        = \alpha(\brd{n+1}{i+1}\brd{n+1}{i}\partial^n_{i+1})
  \end{align*}
  \[\begin{tabular}{c c c c}
      \begin{tikzpicture}[xscale=0.5, yscale=0.5]
        
        \begin{knot}[
          clip width=5,
          clip radius=8pt,
          ]
          
          \node at (0, -0.5)  {\scalebox{0.5}{$i$}} ; 
          \node at (2 , -0.5) {\scalebox{0.5}{$i+1$}} ; 
          \node at (-1, 2.9 )  {\scalebox{0.5}{$i$}} ; 
          \node at (1 , 2.9) {\scalebox{0.5}{$i+1$}} ; 
          \node at (2 , 2.9) {\scalebox{0.5}{$i+2$}} ; 
          % Strand at (2, 2.4)
          \strand[thick] (0, 0.0) to (0, 0.27) to (2, 0.93) to (2, 1.2) to (2, 1.47) to (2, 2.13) to (2, 2.4) ;
          
          % Strand at (-1, 2.4)
          \strand[thick] (0, 1.47) to (-1, 2.13) to (-1, 2.4) ;
          
          % Strand at (1, 2.4)
          \strand[thick] (2, 0.0) to (2, 0.27) to (0, 0.93) to (0, 1.2) to (0, 1.47) to (1, 2.13) to (1, 2.4) ;
          
        \end{knot} 
        
      \end{tikzpicture} 
      & \begin{tikzpicture}[xscale=0.5, yscale=0.5]
    	
    	\begin{knot}[
          clip width=5,
          clip radius=8pt,
          ]
          \node at (4, 1.2 )  {\scalebox{1}{$=$}} ;
          \node at (-3, 1.85 )  {\scalebox{1}{$\alpha$}} ;
          \node at (-3, 1.2 )  {\scalebox{1}{$\Longrightarrow$}} ;
          \node at (0, -0.5)  {\scalebox{0.5}{$i$}} ; 
          \node at (2 , -0.5) {\scalebox{0.5}{$i+1$}} ; 
          \node at (-1, 4.1 )  {\scalebox{0.5}{$i$}} ; 
          \node at (1 , 4.1) {\scalebox{0.5}{$i+1$}} ; 
          \node at (2 , 4.1) {\scalebox{0.5}{$i+2$}} ; 
          
          % Strand at (2, 3.6)
          \strand[thick] (0, 0.0) to (0, 0.27) to (2, 0.93) to (2, 1.2) to (2, 1.47) to (2, 2.13) to (2, 2.4) to (2, 2.67) to (2, 3.33) to (2, 3.6) ;
          
          % Strand at (1, 3.6)
          \strand[thick] (0, 1.47) to (-1, 2.13) to (-1, 2.4) to (-1, 2.67) to (1, 3.33) to (1, 3.6) ;
          
          % Strand at (-1, 3.6)
          \strand[thick] (2, 0.0) to (2, 0.27) to (0, 0.93) to (0, 1.2) to (0, 1.47) to (1, 2.13) to (1, 2.4) to (1, 2.67) to (-1, 3.33) to (-1, 3.6) ;
          
    	\end{knot} 
    	
      \end{tikzpicture} 
      & \begin{tikzpicture}[xscale=0.5, yscale=0.5]
    	
    	\begin{knot}[
          clip width=5,
          clip radius=8pt,
          ]
          
          \node at (-1, -0.5)  {\scalebox{0.5}{$i$}} ; 
          \node at (1 , -0.5) {\scalebox{0.5}{$i+1$}} ; 
          \node at (-1, 5.3 )  {\scalebox{0.5}{$i$}} ; 
          \node at (0 , 5.3) {\scalebox{0.5}{$i+1$}} ; 
          \node at (2 , 5.3) {\scalebox{0.5}{$i+2$}} ; 
          % Strand at (2, 4.8)
          \strand[thick] (-1, 0.0) to (-1, 0.27) to (-1, 0.93) to (-1, 1.2) to (-1, 1.47) to (-1, 2.13) to (-1, 2.4) to (-1, 2.67) to (0, 3.33) to (0, 3.6) to (0, 3.87) to (2, 4.53) to (2, 4.8) ;
          
          % Strand at (0, 4.8)
          \strand[thick] (1, 0.27) to (0, 0.93) to (0, 1.2) to (0, 1.47) to (2, 2.13) to (2, 2.4) to (2, 2.67) to (2, 3.33) to (2, 3.6) to (2, 3.87) to (0, 4.53) to (0, 4.8) ;
          
          % Strand at (-1, 4.8)
          \strand[thick] (1, 0.0) to (1, 0.27) to (2, 0.93) to (2, 1.2) to (2, 1.47) to (0, 2.13) to (0, 2.4) to (0, 2.67) to (-1, 3.33) to (-1, 3.6) to (-1, 3.87) to (-1, 4.53) to (-1, 4.8) ;
          
    	\end{knot} 
    	
      \end{tikzpicture}
      & \begin{tikzpicture}[xscale=0.5, yscale=0.5]
	
	\begin{knot}[
          clip width=5,
          clip radius=8pt,
          ]
          \node at (-3, 1.85 )  {\scalebox{1}{$\alpha$}} ;
          \node at (-3, 1.2 )  {\scalebox{1}{$\Longleftarrow$}} ;
          \node at (-1, -0.5)  {\scalebox{0.5}{$i$}} ; 
          \node at (1 , -0.5) {\scalebox{0.5}{$i+1$}} ; 
          \node at (-1, 4.1 )  {\scalebox{0.5}{$i$}} ; 
          \node at (0 , 4.1) {\scalebox{0.5}{$i+1$}} ; 
          \node at (2 , 4.1) {\scalebox{0.5}{$i+2$}} ; 
          % Strand at (2, 3.6)
          \strand[thick] (-1, 0.0) to (-1, 0.27) to (-1, 0.93) to (-1, 1.2) to (-1, 1.47) to (0, 2.13) to (0, 2.4) to (0, 2.67) to (2, 3.33) to (2, 3.6) ;
          
          % Strand at (-1, 3.6)
          \strand[thick] (1, 0.27) to (0, 0.93) to (0, 1.2) to (0, 1.47) to (-1, 2.13) to (-1, 2.4) to (-1, 2.67) to (-1, 3.33) to (-1, 3.6) ;
          
          % Strand at (0, 3.6)
          \strand[thick] (1, 0.0) to (1, 0.27) to (2, 0.93) to (2, 1.2) to (2, 1.47) to (2, 2.13) to (2, 2.4) to (2, 2.67) to (0, 3.33) to (0, 3.6) ;
          
	\end{knot} 	
      \end{tikzpicture}         
    \end{tabular}
  \]
    
  \item $i=j+1$.
    \begin{align*}
      \alpha(\partial^n_{j+1}\brd{n}{j})
      = & \brd{n+1}{j+1}\partial^n_{j+1}\brd{n}{j}
          = \brd{n+1}{j+1}\brd{n+1}{j}\brd{n+1}{j+1}\partial^n_j\\
      = & \brd{n+1}{j}\brd{n+1}{j+1}\brd{n+1}{j}\partial^n_j
          = \alpha(\brd{n+1}{j}\brd{n+1}{j+1}\partial^n_j)
    \end{align*}
    
    \[ \begin{tabular}{c c c c}
         \begin{tikzpicture}[xscale=0.5, yscale=0.5]
           
           \begin{knot}[
             clip width=5,
             clip radius=8pt,
             ]
             
             \node at (-1, -0.5)  {\scalebox{0.5}{$j$}} ; 
             \node at (1 , -0.5) {\scalebox{0.5}{$j+1$}} ; 
             \node at (-1, 2.9 )  {\scalebox{0.5}{$j$}} ; 
             \node at (0 , 2.9) {\scalebox{0.5}{$j+1$}} ; 
             \node at (2 , 2.9) {\scalebox{0.5}{$j+2$}} ; 
             % Strand at (0, 2.4)
             \strand[thick] (1, 1.47) to (0, 2.13) to (0, 2.4) ;
             
             % Strand at (2, 2.4)
             \strand[thick] (-1, 0.0) to (-1, 0.27) to (1, 0.93) to (1, 1.2) to (1, 1.47) to (2, 2.13) to (2, 2.4) ;
             
             % Strand at (-1, 2.4)
             \strand[thick] (1, 0.0) to (1, 0.27) to (-1, 0.93) to (-1, 1.2) to (-1, 1.47) to (-1, 2.13) to (-1, 2.4) ;
             
           \end{knot} 
           
         \end{tikzpicture}
         & \begin{tikzpicture}[xscale=0.5, yscale=0.5]
           
           \begin{knot}[
             clip width=5,
             clip radius=8pt,
             ]
             \node at (4, 1.2 )  {\scalebox{1}{$=$}} ;
             \node at (-3, 1.85 )  {\scalebox{1}{$\alpha$}} ;
             \node at (-3, 1.2 )  {\scalebox{1}{$\Longrightarrow$}} ;
             \node at (-1, -0.5)  {\scalebox{0.5}{$j$}} ; 
             \node at (1 , -0.5) {\scalebox{0.5}{$j+1$}} ; 
             \node at (-1, 4.1 )  {\scalebox{0.5}{$j$}} ; 
             \node at (0 , 4.1) {\scalebox{0.5}{$j+1$}} ; 
             \node at (2 , 4.1) {\scalebox{0.5}{$j+2$}} ; 
             % Strand at (2, 3.6)
             \strand[thick] (1, 1.47) to (0, 2.13) to (0, 2.4) to (0, 2.67) to (2, 3.33) to (2, 3.6) ;
             
             % Strand at (0, 3.6)
             \strand[thick] (-1, 0.0) to (-1, 0.27) to (1, 0.93) to (1, 1.2) to (1, 1.47) to (2, 2.13) to (2, 2.4) to (2, 2.67) to (0, 3.33) to (0, 3.6) ;
             
             % Strand at (-1, 3.6)
             \strand[thick] (1, 0.0) to (1, 0.27) to (-1, 0.93) to (-1, 1.2) to (-1, 1.47) to (-1, 2.13) to (-1, 2.4) to (-1, 2.67) to (-1, 3.33) to (-1, 3.6) ;
             
           \end{knot} 
           
         \end{tikzpicture}
         & \begin{tikzpicture}[xscale=0.5, yscale=0.5]
           
           \begin{knot}[
             clip width=5,
             clip radius=8pt,
             ]
             
             \node at (0, -0.5)  {\scalebox{0.5}{$j$}} ; 
             \node at (2 , -0.5) {\scalebox{0.5}{$j+1$}} ; 
             \node at (-1, 5.3 )  {\scalebox{0.5}{$j$}} ; 
             \node at (1 , 5.3) {\scalebox{0.5}{$j+1$}} ; 
             \node at (2 , 5.3) {\scalebox{0.5}{$j+2$}} ; 
             % Strand at (2, 4.8)
             \strand[thick] (0, 0.27) to (-1, 0.93) to (-1, 1.2) to (-1, 1.47) to (1, 2.13) to (1, 2.4) to (1, 2.67) to (2, 3.33) to (2, 3.6) to (2, 3.87) to (2, 4.53) to (2, 4.8) ;
             
             % Strand at (1, 4.8)
             \strand[thick] (0, 0.0) to (0, 0.27) to (1, 0.93) to (1, 1.2) to (1, 1.47) to (-1, 2.13) to (-1, 2.4) to (-1, 2.67) to (-1, 3.33) to (-1, 3.6) to (-1, 3.87) to (1, 4.53) to (1, 4.8) ;
             
             % Strand at (-1, 4.8)
             \strand[thick] (2, 0.0) to (2, 0.27) to (2, 0.93) to (2, 1.2) to (2, 1.47) to (2, 2.13) to (2, 2.4) to (2, 2.67) to (1, 3.33) to (1, 3.6) to (1, 3.87) to (-1, 4.53) to (-1, 4.8) ;
             
           \end{knot}            
         \end{tikzpicture}
         & \begin{tikzpicture}[xscale=0.5, yscale=0.5]
           
           \begin{knot}[
             clip width=5,
             clip radius=8pt,
             ]
             \node at (-3, 1.85 )  {\scalebox{1}{$\alpha$}} ;
             \node at (-3, 1.2 )  {\scalebox{1}{$\Longleftarrow$}} ;
             \node at (0, -0.5)  {\scalebox{0.5}{$j$}} ; 
             \node at (2 , -0.5) {\scalebox{0.5}{$j+1$}} ; 
             \node at (-1, 4.1 )  {\scalebox{0.5}{$j$}} ; 
             \node at (1 , 4.1) {\scalebox{0.5}{$j+1$}} ; 
             \node at (2 , 4.1) {\scalebox{0.5}{$j+2$}} ; 
             % Strand at (1, 3.6)
             \strand[thick] (0, 0.27) to (-1, 0.93) to (-1, 1.2) to (-1, 1.47) to (-1, 2.13) to (-1, 2.4) to (-1, 2.67) to (1, 3.33) to (1, 3.6) ;
             
             % Strand at (2, 3.6)
             \strand[thick] (0, 0.0) to (0, 0.27) to (1, 0.93) to (1, 1.2) to (1, 1.47) to (2, 2.13) to (2, 2.4) to (2, 2.67) to (2, 3.33) to (2, 3.6) ;
             
             % Strand at (-1, 3.6)
             \strand[thick] (2, 0.0) to (2, 0.27) to (2, 0.93) to (2, 1.2) to (2, 1.47) to (1, 2.13) to (1, 2.4) to (1, 2.67) to (-1, 3.33) to (-1, 3.6) ;
             
           \end{knot} 
           
         \end{tikzpicture} 
       \end{tabular}
     \]
    
  \item $i>j+1$.
    \begin{align*}
      \alpha(\partial^n_i\brd{n}{j})
      = & \brd{n+1}{i}\partial^n_i\brd{n}{j}
          = \brd{n+1}{i}\brd{n+1}{j}\partial^n_i \\
      = & \brd{n+1}{j}\brd{n+1}{i}\partial^n_i
          = \alpha(\brd{n+1}{j}\partial^n_i)
    \end{align*}
  \end{enumerate}
  The result follows.
\end{proof}
	
\begin{proposition}
  The left and right Leibniz $\B{K}$-algebras $\Leib$ and $\Leib^{op}$ are isomorphic via $\alpha$ as we defined in Proposition~\ref{prop:automorphism}.    
\end{proposition}

\begin{proof}
  We know that $\alpha$ is an automorphism of the $\B{K}$-algebra $\Sym\otimes_\zeta\Mag$.  We will show that $\alpha$ maps the ideals $I_{\Leib}$ and $I_{\Leib^{op}}$ to each other, thus proving our statement.  So we apply $\alpha$ to each generator of $I_{\Leib}$
  \begin{align*}
    \alpha(\partial^{n+1}_{j+1}\partial^n_j - & (1_{n+2}-\brd{n+2}{j+1})\partial^{n+1}_j\partial^n_j)\\
    % =  \brd{n+2}{j+1}\partial^{n+1}_{j+1}\brd{n+1}{j}\partial^n_j
    % - (1_{n+2}-\brd{n+2}{j+1})\brd{n+2}{j}\partial^{n+1}_j\brd{n+1}{j}\partial^n_j\\
    = & \brd{n+2}{j+1}\brd{n+1}{j}\brd{n+1}{j+1}\partial^{n+1}_j\partial^n_j 
	- (1_{n+2}-\brd{n+2}{j+1})\brd{n+2}{j}\brd{n+2}{j+1}\brd{n+2}{j}\partial^{n+1}_{j+1}\partial^n_j\\
    = & \brd{n+2}{j}\brd{n+1}{j+1}\brd{n+1}{j}
 	\left(\partial^{n+1}_j\partial^n_j - (1_{n+2}-\brd{n+2}{j})\partial^{n+1}_{j+1}\partial^n_j\right)
  \end{align*}
  for any $0\leq j\leq n$.
\end{proof}
	
\subsection{A different presentation for $\Leib$}
	
\begin{proposition}\label{prop:simpbasis}
 The product $\Sym\otimes_\zeta \Mag$ has a basis that consists of monomials of the form 
  \begin{equation}\label{eq:leibbase}
  	\tau^{m+1} \partial^m_{j_m}\cdots \partial^n_{j_n}
  \end{equation} 
  where $m\geq n-1$, $\tau^{m+1} \in S_{m+2}$ and $j_m\geq \cdots \geq j_n$. Based on this fact, $\Sym\otimes_\zeta\Simp$ has a basis that consists of monomials of the form 
\[ \tau^{m+1} \partial^m_{j_m}\cdots \partial^n_{j_n} \]
  where $m\geq n-1$, $\tau^{m+2} \in S_{m+1}$ and this time we have $j_m> \cdots > j_n$.     
\end{proposition}

\begin{proof}
 Note that the relations of the distributive law given in Equation~\eqref{eq:CanonicalDistributiveLaw} indicate that we can straighten the arbitrary mixed monomials of $\brd{n}{i}$ and $\partial^n_j$ where $\brd{n}{i}$'s move to the left and $\partial^n_j$'s move to the right. Once this is done, we can straighten $\partial^n_j$'s using Propositions~\ref{prop:MagStraighten} and~\ref{prop:SimpStraighten}. We regroup the monomials of  $\brd{n}{i}$ on the left and call it $\tau^{m+1}$. Hence we can always obtain monomials of the form $\tau^{m+1} \partial^m_{j_m}\cdots \partial^n_{j_n}$ with particular subindices described above.
\end{proof}

\begin{proposition}\label{prop:leibbasis}
  $\Leib^{op}$ has a basis that consists of monomials given in Equation~\eqref{eq:leibbase} where $m\geq n-1$, $\tau^{m+1} \in S_{m+1}$ and this time $j_m>\cdots>j_n$. Thus the $\B{K}$-algebras $\Leib^{op}$ and $\Sym\otimes_\zeta \Simp$ have  the same $\Bbbk$-basis.
\end{proposition}

\begin{proof}
  By applying the new relations we get from $I_{\Leib^{op}}$, we need to straighten the basis monomials with monomials conforming to the condition stated above. Notice that we do not need to straighten the $\Sym$ part, but the $\Mag$ part.   We will write the proof by induction on the length of the monomials coming from $\Mag$. For $m=n-1$ (the trivial monomial $1_n$) and $m=n$ the statement is trivial. So, the base case is when $m=n+1$. If we have a monomial of the form $\partial^{n+1}_{j_{n+1}}\partial^n_{j_n}$ with $j_{n+1}\geq j_n$. The only case where this monomial has to be replaced is when $j_{n+1}=j_n$. In that case we replace $\partial^{n+1}_{j_n}\partial^n_{j_n}$ with $(1-\brd{n+2}{j_n+1})\partial^{n+1}_{j_n+1}\partial^n_{j_n}$ and since $j_n+1>j_n$, the new monomial conforms to the statement. Assume any monomial of length $\ell$ can be straightened to conform to the statement. Assume $\partial^m_{j_m}\cdots \partial^n_{j_n}$ with $m-n=\ell$ is a monomial in $\Leib$ of length $\ell+1$. Notice that the relations in $\Leib^{op}$ indicate that if a length 2 part $\partial^{u+1}_{j_{u+1}}\partial^u_{j_u}$ of a monomial $\partial^m_{j_m}\cdots \partial^n_{j_n}$ is replaced
  \[ \underbrace{\partial^m_{j_m}\cdots \partial^{u+1}_{j_{u+1}}}_\text{affected region}\underbrace{\partial^u_{j_u}\cdots \partial^n_{j_n}}_\text{unaffected region} \]
  the part of the monomial to the right of $\partial^u_{j_u}$ stays unaffected. Thus if the monomial already satisfies $j_{n+1}>j_n$, we can straighten $\partial^m_{j_m}\cdots\partial^{n+1}_{j_{n+1}}$ and then attach $\partial^n_{j_n}$ after the fact. If, on the other hand, $j_{n+1}=j_n$ then $\partial^m_{j_m}\cdots \partial^n_{j_n}$ is replaced with
  \[ \partial^m_{j_m}\cdots \partial^{n+2}_{j_{n+2}}(1-\brd{n+2}{j_n+1})\partial^{n+1}_{j_n+1}\partial^n_{j_n} \]
  and we can move the elements of $\Sym$ all the way to the left, straighten the part to the left of $\partial^n_{j_n}$ in $\Mag$ and then apply the induction hypothesis. The result follows.
\end{proof}
	
\begin{defn}
  Let  $\rho_{-1}^n:=0$ in $\Leib $ and then recursively define
  $$\rho_{j+1}^n:=\partial_{j+1}^n+\chi_{j+1}^{n+1} \rho_j^n$$
  for any $0 \leq j \leq n-1$. One can also define $\rho_j^n$ non-recursively as
  $$\rho_j^n=\partial_j^n+\sum_{a=1}^j \chi_j^{n+1} \cdots \chi_a^{n+1} \partial_{a-1}^n$$
  for any $0 \leq j \leq n$.
\end{defn}
	
\begin{lemma}\label{lem:LeibDistributive}
  We have the following relationship between $\chi_j^n$ and $\rho_i^n$ in $\Leib $:
  \begin{equation}\label{eq:NewDistributive}
    \rho_i^n \chi_j^n
    = \begin{cases}
      \chi_{j+1}^{n+1} \rho_i^n
      & \text { if } j>i \\
      \chi_{i+1}^{n+1} \chi_i^{n+1} \rho_{i+1}^n-\chi_{i+1}^{n+1} \chi_i^{n+1} \chi_{i+1}^{n+1} \rho_i^n+\chi_i^{n+1} \chi_{i+1}^{n+1} \rho_{i-1}^n
      & \text { if } j=i \\
      \chi_j^{n+1} \rho_i^n
      & \text { if } j<i
    \end{cases}
  \end{equation}
  for any $n \geq 1$ and for any $0 \leq i\leq n$ and $0\leq j \leq n-1$.
\end{lemma}

\begin{proof}
  To prove this lemma, we will use a $\B{K}$-algebra endomorphism $(\cdot)[+1]: \Leib \to \Leib $ and its variants. This endomorphism shifts up the right and left degrees on generators by $1$:
  \begin{equation}
    1_n[+1]=1_{n+1} \qquad \chi_i^n[+1]=\chi_i^{n+1} \qquad \partial_j^n[+1]=\partial_j^{n+1}
  \end{equation}
  for $0 \leq i \leq n-1$ and $0 \leq j \leq n$.  Observe that this endomorphism can be generalized to greater shifts in degree. In particular, for all $0\leq j \leq n$ we have 
  $\rho_j^n=\rho_j^j[n-j]$.  Note also that this is a $\B{K}$-algebra morphism because the multiplication structure is determined solely by the lower indices. 
  
  When we visualize each elementary tensor in terms of its diagrammatic representation, we see that the endomorphism simply adds strands to the right of the diagram which does not interfere with the multiplication. On the opposite side, we can remove idle strands in the diagram when we multiply elements and then add them later. So, it is a naturally arising endomorphism. 
  
  We prove the lemma by using case-by-case analysis:
  \begin{enumerate}[{Case } (1)]
  \item $j>i$. 
    Under this assumption, we can reduce the case of calculation of $\rho_i^{n}\chi_j^n$ to $\rho_i^{j+1}\chi_j^{j+1}$ because we can adjust the upper indices as follows:
    \begin{equation*}
      \rho_i^{n}\chi_j^n
      % =\left(\rho_i^{j+1}[n-j-1] \right)\left(\chi_j^{j+1}[n-j-1]\right)
      =\left(\rho_i^{j+1}\chi_j^{j+1}\right)[n-j-1]
    \end{equation*}
    So, it is enough to consider products of the form $\rho_i^{n}\chi_{n-1}^{n}$ for $0 \leq i < n$.
    \begin{align*}
      \rho_i^n\chi_{n-1}^{n}
      = & \partial_i^n\chi_{n-1}^{n} + \sum_{a=1}^i \chi_i^{n+1} \cdots \chi_a^{n+1} \partial_{a-1}^n\chi_{n-1}^{n} \\
      = & \chi^{n+1}_{n}\partial_{i}^n + \sum_{a=1}^i \chi_i^{n+1} \cdots \chi_a^{n+1} \chi^{n+1}_{n} \partial_{a-1}^n \\
      % = &  \chi^{n+1}_{n}\partial_{i}^n + \sum_{a=1}^i \chi^{n+1}_{n+1} \chi_i^{n+1} \cdots \chi_a^{n+1}  \partial_{a-1}^n  \\
      = & \chi^{n+1}_{n} \rho_i^n
    \end{align*}
    Having this equality, we can go back and complete the remaining cases as follows:
    $$\rho_i^{n}\chi_j^n=(\rho_i^{j+1}\chi_j^{j+1})[n-j-1]=(\chi_{j+1}^{j+2}\rho_i^{j+1})[n-j-1]=\chi^{n+1}_{j+1} \rho_i^n$$
    
  \item $i=j$.
    As it is done above, we can reduce the case to $\rho_{n-1}^{n}\chi_{n-1}^n$. 
    \begin{align*}
      \rho_{n-1}^{n}\chi_{n-1}^n
      = & \partial_{n-1}^{n}\chi_{n-1}^n + \chi_{n-1}^{n+1}\partial_{n-2}^{n}\chi_{n-1}^n
          + \sum_{a=1}^{n-2}\chi_{n-1}^{n+1}\chi_{n-2}^{n+1} \cdots \chi_a^{n+1} \partial_{a-1}^n\chi_{n-1}^n \\
      = & \chi^{n+1}_{n}\chi^{n+1}_{n-1}\partial_n^n + \chi_{n-1}^{n+1}\chi^{n+1}_{n}\partial_{n-2}^n
          + \sum_{a=1}^{n-2} \chi_{n-1}^{n+1}\chi_{n-2}^{n+1} \cdots \chi_a^{n+1} \chi^{n+1}_{n} \partial_{a-1}^n \\
          % = & \chi^{n+1}_{n}\chi^{n+1}_{n-1}\partial_n^n + \chi_{n-1}^{n+1}\chi^{n+1}_{n}\partial_{n-2}^n+ \sum_{a=1}^{n-2} \chi_{n-1}^{n+1}\chi^{n+1}_{n} \chi_{n-2}^{n+1} \cdots \chi_a^{n+1} \partial_{a-1}^n \\
          % = & \chi^{n+1}_{n}\chi^{n+1}_{n-1} \left( \rho_n^n - \chi_n^{n+1}\rho_{n-1}^n\right) + \chi_{n-1}^{n+1}\chi^{n+1}_{n} \left( \partial_{n-2}^n+ \sum_{a=1}^{n-2}  \chi_{n-2}^{n+1} \cdots \chi_a^{n+1} \partial_{a-1}^n \right) \\
      = & \chi^{n+1}_{n}\chi^{n+1}_{n-1} \rho_n^n - \chi^{n+1}_{n}\chi^{n+1}_{n-1}\chi_n^{n+1}\rho_{n-1}^n + \chi_{n-1}^{n+1}\chi^{n+1}_{n}\rho_{n-2}^n
    \end{align*}
    
   % \item $i=j=0$
   % \begin{align*}
   % 	\rho_{0}^{n}\chi_{0}^n
   % 	= & \partial_{0}^{n}\chi_{0}^n = \chi^{n+1}_{1}\chi^{n+1}_{0}\partial_1^n  \\
   % 	= & \chi^{n+1}_{1}\chi^{n+1}_{0}\partial_1^n  + \chi^{n+1}_{1}\chi^{n+1}_{0}\chi^{n+1}_{1}\partial_0^n-\chi^{n+1}_{1}\chi^{n+1}_{0}\chi^{n+1}_{1}\partial_0^n  \\
   	% = & \chi^{n+1}_{n}\chi^{n+1}_{n-1}\partial_n^n + \chi_{n-1}^{n+1}\chi^{n+1}_{n}\partial_{n-2}^n+ \sum_{a=1}^{n-2} \chi_{n-1}^{n+1}\chi^{n+1}_{n} \chi_{n-2}^{n+1} \cdots \chi_a^{n+1} \partial_{a-1}^n \\
   	% = & \chi^{n+1}_{n}\chi^{n+1}_{n-1} \left( \rho_n^n - \chi_n^{n+1}\rho_{n-1}^n\right) + \chi_{n-1}^{n+1}\chi^{n+1}_{n} \left( \partial_{n-2}^n+ \sum_{a=1}^{n-2}  \chi_{n-2}^{n+1} \cdots \chi_a^{n+1} \partial_{a-1}^n \right) \\
%   	= & \chi^{n+1}_{1}\chi^{n+1}_{0}\rho_1^n - \chi^{n+1}_{1}\chi^{n+1}_{0}\chi^{n+1}_{1}\rho_0^n
%   \end{align*}
  \item $i=j+1$.
    We can reduce this case to $\rho_{n}^{n}\chi_{n-1}^n$.
    \begin{align*}
      \rho_{n}^{n}\chi_{n-1}^n
      = & \partial_{n}^{n}\chi_{n-1}^n + \chi_{n}^{n+1}\partial_{n-1}^{n}\chi_{n-1}^n
          + \sum_{a=1}^{n-1}\chi_{n}^{n+1}\chi_{n-1}^{n+1}\chi_{n-2}^{n+1} \cdots \chi_a^{n+1} \partial_{a-1}^n\chi_{n-1}^n\\
      = & \chi^{n+1}_{n-1}\chi^{n+1}_{n}\partial_{n-1}^n + \chi_{n}^{n+1}\chi_{n}^{n+1}\chi_{n-1}^{n+1}\partial_{n}^n
          + \sum_{a=1}^{n-1} \chi_{n}^{n+1}\chi_{n-1}^{n+1}\chi_{n-2}^{n+1} \cdots \chi_a^{n+1} \chi^{n+1}_{n} \partial_{a-1}^n \\
      % = & \chi^{n+1}_{n-1}\chi^{n+1}_{n}\partial_{n-1}^n + \chi_{n-1}^{n+1}\partial_{n}^n+ \sum_{a=1}^{n-1} \chi_{n}^{n+1}\chi_{n-1}^{n+1} \chi^{n+1}_{n}\chi_{n-2}^{n+1} \cdots \chi_a^{n+1} \partial_{a-1}^n \\
      = & \chi^{n+1}_{n-1}\chi^{n+1}_{n}\partial_{n-1}^n + \chi_{n-1}^{n+1}\partial_{n}^n+ \sum_{a=1}^{n-1} \chi_{n-1}^{n+1}\chi_{n}^{n+1} \chi^{n+2}_{n-1}\chi_{n-2}^{n+1} \cdots \chi_a^{n+1} \partial_{a-1}^n \\
      % = & \chi^{n+1}_{n-1} \left(\partial_{n}^n+\chi^{n+1}_{n}\partial_{n-1}^n + \sum_{a=1}^{n-1} \chi_{n}^{n+1} \chi^{n+1}_{n-1} \cdots \chi_a^{n+1} \partial_{a-1}^n \right) \\
      = & \chi_{n-1}^{n+1}\rho_{n}^{n}
    \end{align*}
    
  \item $i>j+1$.
    This time we can reduce the case to $\rho_n^n\chi_j^n$ for $0\leq j < n$.
    \begin{align*}
      \rho_{n}^{n}\chi_{j}^n 
      % = & \partial_{n}^{n}\chi_{j}^n + \sum_{a=1}^{n}\chi_{n}^{n+1} \cdots \chi_a^{n+1} \partial_{a-1}^n\chi_{j}^n\\
	= & \chi^{n+1}_{j}\partial_{n}^n + \sum_{a=1}^{j-1} \chi_{n}^{n+1} \cdots \chi_{j+1}^{n+1}\chi_{j}^{n+1}\chi_{j+1}^{n+1} \chi_{j-1}^{n+1}\cdots \chi_a^{n+1} \partial_{a-1}^n \\
          & + \chi_{n}^{n+1} \cdots \chi_{j+1}^{n+1}\chi_{j}^{n+1}\chi_{j+1}^{n+1} \partial_{j-1}^n 
	    +  \chi_{n}^{n+1} \cdots \chi_{j+2}^{n+1} \chi_{j+1}^{n+1}\chi_{j+1}^{n+1} \chi_{j}^{n+1}\partial_{j+1}^n \\
	  & + \chi_{n}^{n+1} \cdots \chi_{j+2}^{n+1}\chi_{j}^{n+1}\chi_{j+1}^{n+1} \partial_{j}^n 
	    + \sum_{a=j+3}^{n}\chi_{n}^{n+1} \cdots \chi_a^{n+1} \chi_{j}^{n+1} \partial_{a-1}^n\\
      % = & \chi^{n+1}_{j}\partial_{n}^n + \sum_{a=1}^{j-1}\chi^{n+1}_{j} \chi_{n}^{n+1} \cdots \chi_a^{n+1} \partial_{a-1}^n \\
         %& + \chi^{n+1}_{j}\chi_{n}^{n+1} \cdots \chi_{j+1}^{n+1}\chi_{j}^{n+1} \partial_{j-1}^n 
	 %  + \chi^{n+1}_{j} \chi_{n}^{n+1} \cdots \chi_{j+2}^{n+1} \partial_{j+1}^n \\
	 %& + \chi^{n+1}_{j}\chi_{n}^{n+1} \cdots \chi_{j+2}^{n+1}\chi_{j+1}^{n+1} \partial_{j}^n 
	 %  + \sum_{a=j+3}^{n}\chi_{j}^{n+1}\chi_{n}^{n+1} \cdots \chi_a^{n+1} \partial_{a-1}^n\\
       = &  \chi^{n+1}_{j} \left( \partial_n^n 
            + \sum_{a=1}^{n}\chi_{n}^{n+1} \cdots \chi_a^{n+1} \partial_{a-1}^n \right) \\
       = & \chi^{n+1}_{j} \rho_n^n
    \end{align*} 
  \end{enumerate}
  Hence we are done.
\end{proof}

\begin{proposition}\label{prop:rhosimp}
  We have the following relations satisfied in $\Leib $:
  $$\rho_i^{n+1} \rho_j^n=\rho_{j+1}^{n+1} \rho_i^n$$
  for any $n \geq 0$ and $0 \leq i \leq j \leq n$.
\end{proposition}

\begin{proof}
  Since we have
  \begin{equation}\label{eq:reduce}
    \rho_i^{n+1}\rho_j^n
    = (\rho_i^{j+1}\rho_j^j)[n-j]
    = \rho_i^{j+1}[n-j]\cdot \rho_j^j[n-j]
  \end{equation}
  the statement reduces to proving $\rho^{n+1}_i\rho^n_n = \rho^{n+1}_{n+1}\rho^n_i$ for every $i\leq n$.  We will use induction on $n$ to prove the statement. For the base case $n=0$ we have
  % $\rho_0^0=\partial_0^0$	and for $n=1$ we have $\rho_0^1=\partial_0^1$ and $\rho_1^1=\partial_1^1+\chi_1^2\partial_0^1 $. So we have $\rho_0^1\rho_0^0=\partial_0^1\partial_0^0 $ and $\rho_1^1\rho_0^0=  (\partial_1^1+\chi_1^2\partial_0^1)\;\partial_0^0=\partial_1^1\partial_0^0+\chi_1^2\partial_0^1\partial_0^0$, which means their difference is given by  
  \begin{align*}
    \rho_1^1\rho_0^0-\rho_0^1\rho_0^0= & \partial_1^1\partial_0^0+\chi_1^2\partial_0^1\partial_0^0-\partial_0^1\partial_0^0= \partial_{1}^{1} \partial_0^0-\left(1_{2} - \chi_{1}^{2}\right) \partial_0^{1} \partial_0^0
  \end{align*}
  Right-hand side of the equation lies in $I_{\Leib} $, meaning that  $\rho_0^1\rho_0^0=\rho_1^1\rho_0^0$ in $\Leib$.
  Now, assume as our induction hypothesis that we have $\rho_i^{n} \rho_{n-1}^{n-1}=\rho_{n}^{n} \rho_i^{n-1}$ for any $0 \leq i \leq n-1$. We need to show that $\rho_i^{n+1} \rho_{n}^{n}=\rho_{n+1}^{n+1} \rho_i^{n}$ holds for any $0 \leq i\leq n$. We will divide this into $3$ cases.
  \begin{enumerate}[{Case} (1)]
  \item $0 \leq i < n-1$
    \begin{align*}
      \rho_i^{n+1} \rho_n^n  
      % & =  \rho_i^{n+1} \left( \partial_n^n+\chi_n^{n+1} \rho_{n-1}^n \right)\\
      % & = \rho_i^{n+1} \left(\partial_n^n+\chi_n^{n+1}\left(\rho_{n-1}^{n-1}\right)[+1] \right)\\ 
      & = \rho_i^{n+1} \partial_n^n+\rho_i^{n+1} \chi_n^{n+1}\left(\rho_{n-1}^{n-1}\right)[+1]\\
      % & = \partial_{n+1}^{n+1} \rho_i^n+\chi_{n+1}^{n+2} \rho_i^{n+1}\left(\rho_{n-1}^{n-1}\right)[+1] \\
      % & = \partial_{n+1}^{n+1} \rho_i^n+\chi_{n+1}^{n+2}\left(\rho_i^n \rho_{n-1}^{n-1}\right)[+1]\\
      & =\partial_{n+1}^{n+1} \rho_i^n+\chi_{n+1}^{n+2}\left(\rho_n^n \rho_i^{n-1}\right)[+1] \\
      & = \partial_{n+1}^{n+1} \rho_i^n+\chi_{n+1}^{n+2} \rho_n^{n+1} \rho_i^n\\
      % & = \left(\partial_{n+1}^{n+1}+\chi_{n+1}^{n+2} \rho_n^{n+1}\right)\rho_i^n\\ 
      & = \rho_{n+1}^{n+1} \rho_i^n
    \end{align*}
    where we use two main identities $\rho_i^{n+1} \chi_n^{n+1}=\chi_{n+1}^{n+2} \rho_i^{n+1}$ from Lemma~\ref{lem:LeibDistributive} and $\rho_i^n \rho_{n-1}^{n-1}=\rho_n^n \rho_i^{n-1}$ from the induction hypothesis. We also use $\rho_i^{n+1} \partial_n^n=\partial_{n+1}^{n+1} \rho_i^n$ which is straightforward to show, using the non-recursive definition of $\rho_i^{n+1}$ and relations in $\Mag$. 
			
  \item $i=n-1$
    \begin{align*}
      \rho_{n-1}^{n+1} \rho_n^n 
      % & =\rho_{n-1}^{n+1}\left(\partial_n^n+\chi_n^{n+1} \rho_{n-1}^n\right) \\
	& =\rho_{n-1}^{n+1}\partial_n^n+\rho_{n-1}^{n+1} \chi_n^{n+1} \rho_{n-1}^n\\
        & =\partial_{n+1}^{n+1} \rho_{n-1}^n+\chi_{n+1}^{n+1} \rho_{n-1}^{n+1} \rho_{n-1}^n \\
        & =\partial_{n+1}^{n+1} \rho_{n-1}^n+\chi_{n+1}^{n+1} \rho_n^{n+1} \rho_{n-1}^n\\
        % & = \left(\partial_{n+1}^{n+1} +\chi_{n+1}^{n+1} \rho_n^{n+1}  \right)\rho_{n-1}^n \\
	& =\rho_{n+1}^{n+1} \rho_{n-1}^n
    \end{align*}
    We use similar identities as in the previous case, namely $\rho_{n-1}^{n+1}\partial_n^n=\partial_{n+1}^{n+1} \rho_{n-1}^n$ and $\rho_{n-1}^{n+1} \chi_n^{n+1} = \chi_{n+1}^{n+1} \rho_n^{n+1}$.
  \item $i=n$ \\
    First, we note that $\partial_n^{n+1}\partial_n^n=\partial_{n+1}^{n+1}\partial_n^n+\chi_{n+1}^{n+2}\partial_{n}^{n+1}\partial_n^n$ in $\Leib$ by the relations in the ideal $I_{\Leib}$. Moreover, we use the identities $\partial_n^{n+1} \chi_n^{n+1}=\chi_{n+1}^{n+2} \chi_n^{n+2} \partial_{n+1}^{n+1}$ and $\partial_{n+1}^{n+1} \rho_{n-1}^n=\rho_{n-1}^{n+1}\partial_{n}^{n}$ and $\chi_n^{n+2} \rho_{n+1}^{n+1}=\rho_{n+1}^{n+1} \chi_n^{n+1}$. So we have:
    
    \begin{align*}
      \rho_n^{n+1} \rho_n^n
      % & =\left(\partial_n^{n+1}+\chi_n^{n+2} \rho_{n-1}^{n+1}\right)\left(\partial_n^n+\chi_n^{n+1} \rho_{n-1}^n\right) \\
      % & =\partial_n^{n+1} \partial_n^n+\partial_n^{n+1} \chi_n^{n+1} \rho_{n-1}^n+\chi_n^{n+2} \rho_{n-1}^{n+1} \left(\partial_n^n+\chi_n^{n+1} \rho_{n-1}^n\right)\\
	& = \partial_n^{n+1} \partial_n^n+\partial_n^{n+1} \chi_n^{n+1} \rho_{n-1}^n+\chi_n^{n+2} \rho_{n-1}^{n+1} \rho_n^n  \\
        & = \partial_{n+1}^{n+1} \partial_n^n+\chi_{n+1}^{n+2} \partial_n^{n+1} \partial_n^n + \chi_{n+1}^{n+2} \chi_n^{n+2} \partial_{n+1}^{n+1} \rho_{n-1}^n+\chi_n^{n+2} \rho_{n+1}^{n+1} \rho_{n-1}^n \\
        & =\partial_{n+1}^{n+1} \partial_n^n+\chi_{n+1}^{n+2} \partial_n^{n+1} \partial_n^n  +\chi_{n+1}^{n+2} \chi_n^{n+2} \rho_{n-1}^{n+1}\partial_{n}^{n}+\chi_n^{n+2} \rho_{n+1}^{n+1} \rho_{n-1}^n \\
        % & =\partial_{n+1}^{n+1} \partial_n^n+\chi_{n+1}^{n+2}\left(\partial_n^{n+1}+\chi_n^{n+2} \rho_{n-1}^{n+1}\right) \partial_n^n+\chi_n^{n+2} \rho_{n+1}^{n+1} \rho_{n-1}^n \\
	& =\rho_{n+1}^{n+1} \partial_n^n+\chi_n^{n+2} \rho_{n+1}^{n+1} \rho_{n-1}^n\\
        & =\rho_{n+1}^{n+1} \partial_n^n+\rho_{n+1}^{n+1} \chi_n^{n+1} \rho_{n-1}^n\\
        & =\rho_{n+1}^{n+1} \rho_n^n
    \end{align*}
  \end{enumerate}
  Hence we are done.
\end{proof}

\subsection{$\Leib$ is a crossed presimplicial algebra}
	
\begin{theorem}\label{thm:MainResult}
  Equation~\eqref{eq:NewDistributive} describes a distributive law of the form
  \[ \omega\colon \Simp\otimes_{\B{K}}\Sym\to {\Sym\otimes_{\B{K}}\Simp} \]
  Consequently, $\Leib$ is isomorphic to $\Sym \otimes_{\omega}\Simp$ as $\B{K}$-algebras.
\end{theorem}

\begin{proof}
  Lemma~\ref{lem:LeibDistributive} together with Proposition~\ref{prop:rhosimp} ensures that any element of  $\Leib$ can be written (not necessarily uniquely) as a $\Bbbk$-linear combination of monomials of the following form:
  \[ \tau^{m+1} \rho^m_{j_m}\cdots \rho^n_{j_n} \]
  where $\tau^{m+1} \in S_{m+2}$ and $j_m>\cdots>j_n$. We also know that $\Leib$ is isomorphic to $\Leib^{op}$, which in turn is isomorphic to $\Sym\otimes_{\B{K}}\Simp$ as $\Bbbk$-vector spaces. Therefore the monomials above must form a $\Bbbk$-vector space basis for $\Leib$. Note that this bijection of the bases works at a bigraded level, so these bijections are in fact between bigraded finite sets. Moreover, this implies that $\Leib$ is also isomorphic to $\Sym\otimes_{\B{K}}\Simp$ as a $\B{K}$-bimodule. Now, by Proposition ~\ref{prop:unique} we have a unique distributive law  $\omega\colon  \Simp  \otimes_{\B{K}} \Sym \to \Sym \otimes_{\B{K}} \Simp$ such that $\Leib$ is isomorphic to $ \Sym \otimes_{\omega} \Simp$ as a $\B{K}$-algebra. The inherent algebra structure in $\Leib$ forces $\omega$ to be the one explicitly described distributive law in Equation~\eqref{eq:NewDistributive}. 
\end{proof}

\begin{remark}
  We again stress the fact that the isomorphism of $\B{K}$-algebras $\Leib$ and $\Sym \otimes_{\omega}\Simp$ we obtained in Theorem~\ref{thm:MainResult} does not lift to the PROP or operad level over the maps mentioned in Section~\ref{sec:Final}. This is because the monoidality of the morphisms is not preserved as can be seen from the diagrams of the elements $\rho_i^n$.
\end{remark}

\subsection{Connection Between the $\B{K}$-algebras and $\Bbbk$-linear PRO(P)s Associated to Operads}\label{sec:Final}

In this Section, we bridge our formalism and the literature on operads.

Recall that the combinatorial description of morphisms in PROPs can be given by stacking rooted trees that represent the elements of the operad side by side, followed by a permutation of the input leaves~\cite[Section 5.4.1]{loday2012algebraic}. In a PRO, there is no permutation of the input leaves.

Let $\C{C}$ be a PRO.  We can forget the monoidal structure on $\C{C}$, and the categorical algebra of the underlying $\Bbbk$-linear category is a $\B{K}$-algebra, and thus we obtain a forgetful functor of the form $\PRO_{\Bbbk}\to \Alg(\B{K})$. 
For example, consider the $\Bbbk$-linear operad $\C{P}$ where 
\begin{equation}\label{identityoperad}
\C{P}(n) = \begin{cases}
	\Bbbk=\langle id \rangle & \text{ if } n=1\\
	0 & \text{ if } n\neq 1
\end{cases}
\end{equation}
If we consider the PRO $\text{cat}\C{P}$ associated to the nonsymmetric operad $\C{P}$, then the categorical algebra of  $\text{cat}\C{P}$ is the $\Bbbk$-algebra $\B{K}$. We only need a degree shift by one to make them fit with the conventions of the operads.

The operad $\C{P}$ we defined in~\eqref{identityoperad} can also be viewed as a symmetric operad with the trivial $S_1$ action. If we consider the PROP  $\text{cat}\C{P}$ associated to the symmetric operad  $\C{P}$, then this is equal to the category $\Bbbk[\B{S}]$ whose categorical algebra is the $\B{K}$-algebra $\Sym$. % In particular, in the combinatorial description of the PROP we only get permutation of the input leaves in $\text{cat}\C{P}$. 

On the other hand, the free nonsymmetric operad generated by a single binary operation is the nonunital magmatic operad $\textbf{Mag}$  and the categorical algebra of the PRO $\text{cat}\textbf{Mag}$ associated to $\textbf{Mag}$ is the $\B{K}$-algebra $\Mag$.  In our formalism, the binary operations i.e. $\partial_j^n$'s are applied sequentially to the input leaves. This is analogous to operadic partial compositions. All of the relations between partial composition operators in our $\Mag$ come from the parallel composition laws of monoidal categories written in the nonunital nonsymmetric magmatic operad $\textbf{Mag}$. However, since we do not have the luxury of working within a monoidal category, we need to write all of these relations explicitly in the definition of $I_{\Mag}$ in our formalism.

 Permutations of the input leaves and their interactions with compositions in a PROP correspond to a specific distributive law between symmetric groups $\Sym$ and the $\B{K}$-algebra corresponding to that PROP in our formalism. In particular, $\B{K}$-algebra $\Sym \otimes_{\zeta} \Mag$ is the categorical algebra of the PROP $\text{cat}\textbf{Mag}$ associated to the nonunital symmetric magmatic operad $\textbf{Mag}$. Similarly, the $\B{K}$-algebra $\Sym \otimes_{\zeta} \Simp$ is the categorical algebra of the PROP $\text{cat}\textbf{As}$ associated to the nonunital symmetric associative operad $\textbf{As}$.

Thus, the $\B{K}$-algebras $\Mag$ and $\Simp$ fit this picture. The categorical algebras of the PROs associated to nonunital nonsymmetric operads $\textbf{Mag}$ and $\textbf{As}$ are exactly the $\B{K}$-algebras $\Mag$ and $\Simp$ that we work with in this paper. Similarly, the categorical algebras of the PROPs associated to nonunital symmetric operads $\textbf{Mag}$ and $\textbf{As}$ are exactly the $\B{K}$-algebras $\Sym \otimes_{\zeta} \Mag$ and $\Sym \otimes_{\zeta} \Simp$. We would like to reiterate emphatically that even though $\Leib$ by itself also conforms to this picture, the isomorphism between $\Leib$ and  $\Sym \otimes_{\omega} \Mag$ we furnish in Theorem~\ref{thm:MainResult} cannot.  The isomorphism cannot be lifted to an isomorphism of \emph{monoidal} categories, and therefore, cannot be operadic.

\subsection{Homological ramifications}\label{subsect:homologicalramifications}

Using the underlying $\Bbbk$-linear categories of the PROPs related to operads as we did in this paper allows one to write the (co)homology of any type of algebraic structure as an ordinary derived functor since the categorical algebra of any such $\Bbbk$-linear category is \emph{locally unital} and \emph{associative}.

Let $\C{C}$ be a $\Bbbk$-linear category parametrizing all formal $n$-to-$m$ operations of an algebraic structure, and let $A$ be an algebra over $\C{C}$.  Let $T(A)$ be the graded vector space $\bigoplus_n A^{\otimes n+1}$ considered as a right $\C{C}$-module. Then the (co)homology of $A$ simply should be
\[ H_*(A) = \Tor^{\C{C}}_*(T(A),\B{K})\qquad H^*(A) = \Ext_{\C{C}}^*(T(A),\B{K}) \]
However, this shifts the difficulty to the combinatorial side: we must now write a resolution for the point object $\B{K}$ as a $\C{C}$-module as Connes did for $\Delta C$ in~\cite{Connes:ExtFunctors} for cyclic (co)homology.

Now, we can use Theorem~\ref{thm:MainResult} together with the machinery of $\Bbbk$-linear categories to show that the mechanics of the homotopy theory of Leibniz algebras are slightly different than the mechanics of the corresponding homotopy theory of Lie algebras.

Let $CE^{\Leib}_*(L)$ be the Chevalley-Eilenberg-Leibniz complex of a Leibniz algebra $L$, and let $CE^{\rm Lie}_*(L)$ be the Chevalley-Eilenberg-Lie complex as defined in~\cite{Loday_1993}.  Let $T(L)$ be the graded vector space $\bigoplus_n L^{\otimes n+1}$ considered as a right $\Leib$-module. A corollary of our result Theorem~\ref{thm:MainResult} would be that the face maps in the expansion of Chevalley-Eilenberg-Leibniz differentials defining Leibniz homology implement the embedding of $\Simp$ in $\Leib$.

Since $\Simp\to \Leib$ is a subalgebra, we have two functors: (i) an induction functor ${\rm Ind}_{\Simp\to \Leib}$, and (ii) a restriction functor ${\rm Res}_{\Simp\to\Leib}$. Therefore
\begin{align*}
	H^\Leib_*(L) := & H_*(CE^{\Leib}_*(L)) = \Tor^{\Simp}_*({\rm Res}_{\Simp\to\Leib} T(L),\B{K})\\
	\cong & \Tor^{\Leib}_*(T(L),{\rm Ind}_{\Simp\to\Leib} \B{K})  \cong  \Tor^{\Leib}_*(T(L), \B{S})
\end{align*}
where the last isomorphism is a direct corollary of our Theorem~\ref{thm:MainResult}.

Note that there is also an epimorphism of $\Bbbk$ -linear categories of the form $\Leib\to\Lie$ which yields (i) an induction functor ${\rm Ind}_{\Leib\to\Lie}$ which sends a Leibniz algebra $L$ to its universal Lie quotient $L_\Lie$, and (ii) restriction functor ${\rm Res}_{\Leib\to \Lie}$ where we consider a Lie algebra $L$ as a Leibniz algebra. These functors implement an isomorphism in (co)homology
\begin{align*}
	H^\Lie_*(L)
	:= & \Tor^{\Lie}_*({\rm Ind}_{\Leib\to\Lie}T(L),\B{K})
	= H_*(CE^\Lie_*(L_\Lie))\\
	\cong & \Tor^{\Leib}_*(T(L),{\rm Res}_{\Leib\to \Lie}\B{K}) = \Tor^{\Leib}_*(T(L),\B{K}) 
\end{align*}
In other words, the Lie (co)homology of a Leibniz algebra $L$ comes from Leibniz homology with coefficients in the point object $\mathbb{K}$, while the Leibniz (co)homology of $L$ comes from the collection of symmetric groups $\B{S}$ as coefficients.  To formalize this, we need to investigate the relationship between the homotopy categories of Lie and Leibniz algebras similar to~\cite{KaygunKaya2024}. We plan to do this in a future paper.

\appendix \section{Operads and PRO(P)s}\label{sect:Operads}

In order to put our definitions of $\B{K}$-algebras into their proper context, we need to recall some basic facts about operads, PROs, and PROPs.  Our main references are \cite{may1997definitions, leinster2000homotopy, markl_operads_2002, Hinich_2002, loday2012algebraic}.  We assume $\C{C}$ is a (symmetric) monoidal category where monoidal product is denoted by $\odot$ and the identity object is denoted by $\B{I}$ throughout the appendix.

\subsection{Nonsymmetric operads}{~\cite[Section 5.9.3]{loday2012algebraic}}

A nonsymmetric operad $\C{P}$ in $\C{C}$ consists of collection of objects $\C{P}_n$ in $\C{C}$ for $n \geq 0$, a unit morphism $ \iota:\B{I} \to \C{P}_1$, and a collection of composition morphisms 
\[\gamma(m;n_1,\ldots , n_m)\colon \C{P}_m \odot \C{P}_{n_1} \odot \cdots \odot \C{P}_{n_m} \longrightarrow \C{P}_n
\]
for every $n\geq 1$ and $m$-composition $n_1+\cdots+n_m=n$ of $n$. These morphisms, together with the unit morphism $\iota:\B{I}\to \C{P}_1$ must satisfy unitality 
\[\begin{tikzcd}
	\C{P}_m \odot \B{I}^{\odot m} 
	\arrow[rr,"\cong"]
	\arrow[dd,"{{\C{P}}_m \odot \iota^{\odot m}}"']
	&& {\C{P}_m} 
	&& {\B{I}\odot\C{P}_m} 
	\arrow[ll,"\cong"']
	\arrow[dd,"{\iota\odot {\C{P}_m}}"]
	\\\\
	\C{P}_m\odot \C{P}_1^{\odot m} 
	\arrow[uurr,"{\gamma(m;1,\ldots,1)}"']
	&&&& 
	{\C{P}_1 \odot \C{P}_m}
	\arrow[uull,"{\gamma(1;m)}"]
\end{tikzcd}
\]
and associativity axioms. 
\[\xymatrix{
	\C{P}_m \odot \big(\bigodot_{s=1}^m \big(\C{P}_{n_s}\odot\big(\bigodot_{r=1}^{n_s} \C{P}_{\ell_{s,r}}\big)\big)\big)
	\ar[rrrr]^{\hspace{10mm}{\C{P}_m}\odot \bigodot_{s=1}^m\gamma(n_s;\ell_{s,1},\ldots,\ell_{s,n_s})}
	\ar[d]_{\text{shuffle}}
	& & & & 
	\C{P}_m \odot \C{P}_{\ell_1}\odot\cdots\odot \C{P}_{\ell_m}
	\ar[ddd]^{\gamma(m;\ell_1,\ldots,\ell_m)}
	\\
	\C{P}_m\odot \big(\bigodot_{s=1}^m \C{P}_{n_s}\big)\odot \big(\bigodot_{s=1}^m\bigodot_{r=1}^{n_s}  \C{P}_{\ell_{s,r}}\big)
	\ar[dd]_{\gamma(m;n_1,\ldots,n_m)\odot \bigodot_{s=1}^m\bigodot_{r=1}^{n_s} {\C{P}_{\ell_{s,r}}}}
	\\\\
	\C{P}_n\odot \big(\bigodot_{s=1}^m\bigodot_{r=1}^{n_s}  \C{P}_{\ell_{s,r}}\big)
	\ar[rrrr]_{\gamma(n;\ell_{1,1},\ldots,\ell_{1,n_1},\ldots,\ell_{m,1},\ldots,\ell_{m,n_m})}
	& & & &  
	\C{P}_{\ell}
}\]
where $1\leq k \leq m$, $\ell=\ell_1+\cdots+\ell_m$, $\ell_k=\ell_{k, 1}+\cdots+\ell_{k, n_k}$, and $n=n_1+\cdots+n_m$.  Note that
$\ell = \sum_{i=1}^k\sum_{j=1}^{n_i} \ell_{i,j}$ is an $n$-composition of $\ell$ because $n=n_1+\cdots+n_m$.

As for the morphisms, we will say $f\colon \C{P}\to \C{Q}$ is a morphism of operads if $f$ is a collection of morphisms $f_n\colon \C{P}_n\to \C{Q}_n$ in $\C{C}$ such that 
\[\xymatrix{
	\C{P}_m\odot \C{P}_{n_1}\odot\cdots\odot\C{P}_{n_m}
	\ar[rrr]^{\hspace{1cm}\gamma(m;n_1,\ldots,n_m)} 
	\ar[d]_{f_m\odot f_{n_1}\odot\cdots\odot f_{n_m}}
	& & &\C{P}_n 
	\ar[d]^{f_{n}}\\
	\C{Q}_m\odot \C{Q}_{n_1}\odot\cdots\odot\C{Q}_{n_m}
	\ar[rrr]_{\hspace{1cm}\gamma(m;n_1,\ldots,n_m)} 
	& & & \C{Q}_n 
}\]
commutes for every $n$ and for every $m$-composition $n_1+\cdots+n_m=n$.

We use $\OP(\C{C})$ to denote the category of operads in $\C{C}$.

\subsection{Symmetric groups as an operad}\label{subsect:Assoc}

Consider the collection of symmetric groups $S_n$ for $n\geq 1$ which we denote by $\C{S}$.  We define an operad using $\C{S}$ in the category of finite sets as follows: let $n_1+\cdots+n_m=n$ be a $m$-composition of $n$. Define the composition maps 
\[ \gamma(m;n_1,\ldots,n_m)\colon S_m\times S_{n_1}\times\cdots\times S_{n_m} \to S_n \]
by
\[ \gamma(m;n_1,\ldots,n_m)(\mu,\sigma_1,\ldots,\sigma_m)
= \sigma_{\mu^{-1}(1)}\oplus\cdots\oplus\sigma_{\mu^{-1}(m)}
\]
for every $\mu\in S_m$ and $\sigma_i\in S_{n_i}$ for $i=1,\ldots,m$ where the sum $\sigma\oplus \sigma'$ of two permutation $\sigma\in S_a$ and $\sigma'\in S_b$ is defined as 
\begin{equation}\label{eq:sumPermutations}
	(\sigma\oplus \sigma')(\ell)
	= \begin{cases} 
		\sigma(\ell) & \text{ if } \ell\leq a \\ 
		a + \sigma'(\ell-n) & \text{ if } a<\ell\leq a+b
	\end{cases}
\end{equation}
for every $a,b\in \B{N}$ and $1\leq \ell\leq a+b$. See {\cite[Example 1.5]{heuts_simplicial_2022}}.

\subsection{Symmetric operads}\label{subsect:symmetricoperads}

An operad $\C{P}$ is called \emph{symmetric} if $S_n$ has a right action on $\C{P}_n$ for all $n \in \mathbb{N}$. The action comes with two equivariance conditions ~\cite{may1997definitions}.  The first equivariance condition is given by the following diagram:
\[\begin{tikzcd}
	{\C{P}_m \odot \C{P}_{n_1} \odot \cdots \odot \C{P}_{n_m}} 
	\arrow[rrrr,"{\sigma\odot {\C{P}_{n_1}}\odot\cdots\odot {\C{P}_{n_m}}}"]
	\arrow[dd,"{\gamma(m;n_1,\ldots,n_m)}"']
	&&&&
	{\C{P}_m \odot \C{P}_{n_1} \odot \cdots \odot \C{P}_{n_m}} 
	\arrow[d,"{\C{P}_m}\odot\sigma"]\\
	&&&&
	{\C{P}_m \odot \C{P}_{n_{\sigma(1)}} \odot \cdots \odot \C{P}_{n_{\sigma(m)}}} 
	\arrow[d,"{\gamma(m;n_{\sigma(1)},\ldots,n_{\sigma(m)})}"]\\
	{\C{P}_n} 
	\arrow[rrrr,"{\sigma(n_1,\ldots,n_m)}"']
	&&&&
	{\C{P}_n}
\end{tikzcd}
\]
for $\sigma \in S_m$, and for any composition $n_1+\cdots+n_m=n$.  The permutation $\sigma(n_1,\ldots,n_m) \in S_{n_1+\cdots+n_m}$ permutes the blocks $(1,\ldots, n_1),\ldots,(n_{m-1}+1,\ldots,n_m)$ as $\sigma$ permutes $\{1,\ldots,m\}$ with a left action.  The second equivariance condition is as follows: 
\[
\begin{tikzcd}
	{\C{P}_m \odot \C{P}_{n_1} \odot \cdots \odot \C{P}_{n_m}} 
	\arrow[rrr,"{{\C{P}_m}\odot\sigma_1 \odot\cdots \odot \sigma_m}"]
	\arrow[dd,"{\gamma(m;n_1,\ldots,n_m)}"']
	&&& 
	{\C{P}_m \odot \C{P}_{n_1} \odot \cdots \odot \C{P}_{n_m}} 
	\arrow[dd,"{\gamma(m;n_1,\ldots,n_m)}"]
	\\\\
	{\C{P}_n} 
	\arrow[rrr,"{\sigma_1 \oplus \cdots \oplus \sigma_m}"']
	&&& 
	{\C{P}_n}
\end{tikzcd}
\]
for given $\sigma_i \in S_{n_i}$ for $1\leq i \leq m$ and their block sum $\sigma_1 \oplus \cdots \oplus \sigma_m \in S_{n_1+\cdots+n_m}$.  Because of these two equivariance conditions, the $\gamma$ maps are defined as the sum of the following morphisms
\[\gamma(m;n_1,\ldots,n_m)\colon
\C{P}_m \odot_{S_m}
\operatorname{Ind}_{S_{n_1} \times \cdots \times S_{n_m}}^{S_n}
\left(\C{P}_{n_1} \odot\cdots\odot \C{P}_{n_m}\right) \longrightarrow \C{P}_n
\]
over all compositions of the form $n_1+\cdots+n_m=n$.  Tensor product over $S_n$ comes from the first equivariance condition, while the induced module comes from the second.

\subsection{The endomorphism operad}

Let $A$ be an object in $\C{C}$. There is a canonical operad associated with $A$ which we denote by $\C{O}(A)$ where
\[ \C{O}_n(A) = {\C{C}}(A^{\odot n},A) \]
The operadic composition law is defined as follows: let $n_1+\cdots+n_m=n$ be an $m$-composition of $n$ and define
\[ \C{O}_{m}(A)\odot\C{O}_{n_1}(A)\odot\cdots\odot\C{O}_{n_m}(A) \to \C{O}_n(A) \]
via
\[ \xymatrix{
	{\C{C}}(A^{\odot m},A)\otimes {\C{C}}(A^{\odot n_1},A)\otimes\cdots\otimes{\C{C}}(A^{\odot n_m},A)
	\ar[d]^{\C{C}(A^{\odot m},A)\otimes\odot}\\
	{\C{C}}(A^{\odot m},A)\otimes {\C{C}}(A^{\odot n}, A^{\odot m})
	\ar[d]^{\circ}\\
	{\C{C}}(A^{\odot n},A)
}\]
The operad $\C{O}(A)$ is called the \emph{endomorphism operad} associated with $A$. Note that if we assume $\C{C}$ is symmetric (resp. braided) monoidal, then the endomorphism operad is naturally symmetric (resp. braided).

\subsection{Algebras over operads}

Let $\C{P}$ be a operad in $\C{C}$. We call an object $A\in\C{C}$ as a $\C{P}$-algebra, if there is a morphism of operads of the form $\lambda_A\colon \C{P}\to \C{O}(A)$.  Given two $\C{P}$-algebras $A$ and $B$, a morphism $f\colon A\to B$ is called a morphism of $\C{P}$-algebras if the following diagram
\begin{equation}\label{eq:OperadAlgebra}
	\xymatrix{
		A^{\odot n} 
		\ar[d]_{f^{\odot n}}
		\ar[rr]^{\lambda_A(\alpha)} 
		& &  \ar[d]^{f} A\\
		B^{\odot n} 
		\ar[rr]_{\lambda_B(\alpha)} 
		& &  B
	}
\end{equation}
commutes for all $n\geq 1$ and $\alpha\in \C{P}_n$.  Algebras over $\C{P}$ together with $\C{P}$-algebra morphisms makes a category denoted by $\Alg_{\C{C}}(\C{P})$.

\subsection{PROs}

Our main references for this Section are \cite[Definition 2.2.2]{Hinich_2002} and ~\cite[Chapter 5]{MR0171826}.  

A PRO (\emph{PRoduct Operations}) in $\C{C}$ is a strict monoidal category $\C{P}$ enriched in $\C{C}$ with the objects $[n]$ for $n \in \B{N}$ where morphisms between objects are objects in $\C{C}$. $\C{P}$ is equipped with a monoidal product $\oplus$ given on the objects by the sum of natural numbers.

PROs can be thought of as categories that model algebraic operations with multiple inputs and multiple outputs, whereas operads model multiple inputs but one output.  Thus if $\C{P}$ is a PRO(P) in $\C{C}$, then $res\C{P}_n :=\C{P}(n,1)$ is an operad in $\C{C}$.

\subsection{Symmetric groups as a PRO}

Collection of symmetric groups $\B{S}=\bigsqcup_{n\geq 1} S_n$ forms a monoidal category as follows: $Ob(\B{S})=\B{N}\setminus\{0\}$ and 
\[ \B{S}(n,m) = 
\begin{cases}
	S_n & \text{ if } n=m\\
	\emptyset & \text{ otherwise}
\end{cases}
\]
The monoidal product $\oplus$ on objects is just addition of natural numbers. The monoidal product on permutations is defined in Equation~\eqref{eq:sumPermutations}.

Note that $\B{S}$ is a strict symmetric monoidal category where the switch is defined as
\[ \xymatrix{
	n \oplus m 
	\ar[d]_{\tau_{n,m}}
	\ar[rr]^{\sigma\oplus\sigma'} 
	& & n \oplus m 
	\ar[d]^{\tau_{n,m}}\\
	m\oplus n 
	\ar[rr]_{\sigma'\oplus\sigma} 
	& & m\oplus n
}\]
using the permutation $\tau_{n,m}$
\begin{equation*}\label{swap}
	\tau_{n,m}(i)
	= \begin{cases} 
		\ell+m & \text{ if } \ell\leq n \\ 
		\ell-n & \text{ if } n<\ell\leq n+m 
	\end{cases}
\end{equation*}
for every $n,m\in \B{N}$ and $1\leq \ell\leq n+m$.

\subsection{The PRO associated to a nonsymmetric operad}

Our main references for this Section are \cite[Section 4.1]{may1978uniqueness}, \cite{leinster2000homotopy}, \cite[Section 2.2.6]{Hinich_2002}, and \cite[Section 3]{BASTERRA2018130}.

Assume $\C{O}$ is an operad over $\C{C}$, and let us define a PRO $\text{cat}{\C{O}}$.  Since we are defining a PRO, the objects of the category $cat{\C{O}}$ are $[n]$ for any $n\in\B{N}$, and the monoidal product on the objects is given by addition of natural numbers.  The morphisms in $\text{cat}\C{O}$ are given by
\begin{equation}\label{eq:PROofOp}
	\text{cat}\C{O}([n],[m]) 
	= \bigoplus_{n_1+\cdots+n_m=n} \C{O}_{n_1}\odot\cdots\odot \C{O}_{n_m} 
\end{equation}
Here the sum in Equation~\eqref{eq:PROofOp} is taken over all $m$-compositions of $n$. Since we implicitly assume $\C{O}_0=0$, one can take the sum over all order preserving surjections of the form $f\colon [n]\to [m]$ where each $n_i = |f^{-1}(i)|$.  Thus a (homogeneous) morphism $f\in cat\C{O}([n],[m])$ is of the form $(f,b)$ where $f\colon [n]\to [m]$ is an order preserving surjection, and $b =(b_1\odot\cdots\odot b_m)\in\C{O}_{n_1}\odot\cdots\odot\C{O}_{n_m}$ where $n_i = |f^{-1}(i)|$ for $i=1,\ldots,m$.  For two homogeneous morphisms $(g,b)\colon [m]\to[\ell]$ and $(f,a)\colon [n]\to[m]$ in $cat\C{O}$, their composition $g\circ f\colon [n]\to [\ell]$ is defined via
\begin{equation} \label{eq:PropComposition}
	(g,b)\circ(f,a)
	= (g\circ f, 
	\gamma(b_1,\odot_{g(j_1)=1} a_{j_1})
	\odot\cdots\odot
	\gamma(b_\ell,\odot_{g(j_\ell)=\ell} a_{j_\ell}))
\end{equation} 
using the composition in the operad $\C{O}$. The monoidal product of morphisms is taken in the monoidal category $(\C{C},\odot)$ since morphisms are defined in this category.

\subsection{The PROP associated to a symmetric operad}
In order to extend the PRO $cat\C{O}$ to a PROP, we need to add a set of symmetric group actions. Thus we write

\begin{equation}\label{eq:PROPofAnOperad}
	\text{cat}\C{O}([n],[m]) 
	= \bigoplus_{n_1+\cdots+n_m=n} \left( \C{O}_{n_1}\odot\cdots\odot \C{O}_{n_m} \right)\otimes_{S_{n_1} \times \cdots \times S_{n_m}} S_n
\end{equation}

We denote a morphism $(f,\textbf{a},\sigma) \in 	\text{cat}\C{O}([n],[m]) $ where $\sigma \in S_n$ and $f\colon [n]\to [m]$ is an order preserving surjection and $\textbf{a} =(a_1\odot\cdots\odot a_m)\in\C{O}_{n_1}\odot \cdots\odot \C{O}_{n_m}$ where $n_i = |f^{-1}(i)|$ for $i=1,\ldots,m$.

For two morphisms $(f,\textbf{a},\sigma) \colon [n]\to[m]$ and $(g,\textbf{b},\tau) \colon [m]\to[\ell]$ in $cat\C{O}$, their composition $(g,\textbf{b},\tau)\circ(f,\textbf{a},\sigma) \colon [n]\to [\ell]$ is explicitly defined as follows:

$$
\begin{tikzpicture}[baseline= (a).base]
	\node[scale=.95] (a) at (0,0){
		\begin{tikzcd}[column sep=tiny]
			{
				(g\circ f, 
				\gamma(b_1,a_{\tau^{-1}(1)},\cdots, a_{\tau^{-1}(m_1)})
				\odot\cdots\odot
				\gamma(b_\ell,a_{\tau^{-1}(m_1+\cdots+m_{\ell -1}+1)},\cdots, a_{\tau^{-1}(m_1+\cdots+m_{\ell})}),\tau(n_1,\ldots,n_m)\circ \sigma)
			}
		\end{tikzcd} 
	};
\end{tikzpicture}$$

where $m_i=|g^{-1}(i)|$ and $m_1+\cdots+m_{\ell}=m$ and $\tau(n_1,\ldots,n_m) \in S_{n_1+\cdots+n_m}=S_n$ is the block permutation as it is defined on section~\ref{subsect:symmetricoperads}.

The right action of $S_m$ comes from the symmetric monoidal structure on $\C{C}$ permuting the terms $\C{O}_{n_1}\odot\cdots\odot \C{O}_{n_m}$, but also accordingly  changes the symmetric group element in $S_n$ that was in the label of the morphism. The left $S_n$ action now comes from the induced action from $S_{n_1}\times\cdots\times S_{n_m}$ to $S_n$. A similar account on this symmetric structure is given in ~\cite[Section 3]{BASTERRA2018130}. Note that the category we define is the opposite of the category in that reference.

\subsection{Variations in the literature}

The prototypical example of the PROP associated with an operad is Segal's category $\Gamma$~\cite{Segal:CategoriesCohomologyTheories} where $\Gamma$ appears as the PROP for the commutative operad in the category of based pointed sets with smash product as the monoidal product. Our definition of the PROP associated with an operad $\C{O}$ comes from~\cite[Section 3]{BASTERRA2018130}.  In~\cite[Section 4.1]{may1978uniqueness} May defines the PROP associated with an operad with different Hom objects, and in ~\cite[Section 10.2]{may2020operads} it is shown that these Hom objects are isomorphic to the Hom objects
\begin{equation}
	\text{cat}\C{O}([n],[m]) 
	= \bigoplus_{f:[n]\to [m]} \C{O}_{f^{-1}(1)}\odot\cdots\odot \C{O}_{f^{-1}(m)} 
\end{equation}
where the index runs through all surjective set functions.  Leinster~\cite[pp.21]{leinster2000homotopy} uses the same Hom objects, while Adams~\cite[pp.42]{adams1978infinite} and Markl~\cite[Example 60]{MARKL200887} use the definition of the Hom objects given in Equation~\eqref{eq:PROPofAnOperad}.

Note that May's description of the Hom objects is equivalent to writing the sum over compositions as in Equation~\eqref{eq:PROPofAnOperad}.  This is because for every surjective function $f\colon [n]\to [m]$, there is a unique order preserving map $\tilde{f}\colon [n]\to [m]$ and a (not necessarily unique) permutation $\sigma\in S_n$ such that $\tilde{f} = f\circ \sigma$.  The permutation determines a unique right coset $(S_{n_1}\times\cdots\times S_{n_m})\sigma$ of the subgroup $S_{n_1}\times\cdots\times S_{n_m}$ of $S_n$ coming from block sum of permutations. Thus May's PROP comes with a canonical symmetric structure which is the same as ours.

For two homogeneous morphisms $(g,b)\colon [m]\to[\ell]$ and $(f,a)\colon [n]\to[m]$ in $cat\C{O}$ May defines the composition as  
\begin{equation}
	(g \circ f, \gamma(b_1; \odot_{g(k)=1} a_k)\cdot \sigma_1, \cdots , \gamma(b_i; \odot_{g(k)=i} a_k)\cdot \sigma_i, \cdots , \gamma(b_m; \odot_{g(k)=m}a_k)\cdot \sigma_m)
\end{equation} 
where $\sigma_i \in S_{(g\circ f)^{-1}(i)}$ are the permutations that makes $(g \circ f)(\sigma_1\oplus\cdots\oplus\sigma_m)$ order preserving. Again, this is equivalent to our composition law in the symmetric case.  Adams~\cite[pp.42]{adams1978infinite} and Markl~\cite[Example 60]{MARKL200887} do not explicitly define the compositions for their PROPs. Leinster~\cite[pp.22]{leinster2000homotopy} claims he uses ``closely related but slightly different'' composition compared to May~\cite{may1978uniqueness}, but without an explicit description.

\subsection{The PRO(P) associated to an operad is as good as the operad itself}

We will call a PRO $\C{P}$ as \emph{reducible} if $\C{P} = cat (res\C{P})$, i.e.
$$ \C{P}([n],[m]) =  \bigoplus_{n_1+\cdots+n_m=m} \C{P}([n_1],[1])\odot\cdots\odot\C{P}([n_m],[1]) $$
for every $n$ and $m$.

\begin{proposition}[{\cite[Section 2.2.6]{Hinich_2002} and~\cite[Proposition 3.1]{BASTERRA2018130}}]
	The functors
	\[ \xymatrix{
		\text{cat}\colon\OP(\C{C})\ar@/^2ex/[rr] & & \ar@/^2ex/[ll] \PRO(\C{C})\colon\text{res}
	}\]
	are an adjoint pair. Moreover, they induce an equivalence between $\OP(\C{C})$ and $red\PRO(\C{C})$.
\end{proposition}

\begin{proof}
	We must show $\PRO(\text{cat}\C{O},\C{P}) \cong \OP(\C{O}, res\C{P})$ for every operad $\C{O}$ in $\C{C}$.  We first note that morphisms of PROs are functors which are identity on the set of objects. Moreover, since functors in $\PRO(cat\C{O},\C{P})$ are monoidal and morphisms in $cat\C{O}$ are obtained from $\C{O}$, any functor of the form $F\colon cat\C{O}\to \C{P}$ on morphisms is determined by their image $\C{O}_n\to \C{P}([n],[1])$. Thus we have shown that the monoidal category associated with an operad is a (reducible) PRO, and $\PRO(\text{cat}\C{O},\C{P}) \cong \OP(\C{O}, res\C{P})$.  On the opposite side, if we only consider the morphisms from $[n]$ to $[1] $ in a PRO, $res\C{P}_n :=\C{P}([n],[1])$ is an operad. When $\C{P}$ is reducible, every morphism $[m]\to [n]$ in $\C{P}$ can be written as a sum of monoidal products of morphisms of the form $[m_1]\to [1],\ldots,[m_n]\to [1]$ where $m=m_1 + \cdots + m_n$. In such cases, the (monoidal) category associated to $res\C{P}$ is $\C{P}$ itself since $\C{P}$ is assumed to be reducible.
\end{proof}

\begin{proposition}[{\cite[Theorem 1.6.1]{leinster2000homotopy}}] \label{prop:correspondence} 
	Let $\C{O}$ be an operad in $\C{C}$, and let $\Alg_{\C{C}}(\C{O})$ be the category of $\C{O}$-algebras in $\C{C}$.  Then there is an equivalence of categories between $\Alg_{\C{C}}(\C{O})$ and the category of monoidal functors $\textbf{Mon}(cat\C{O},\C{C})$ from $cat\C{O}$ to $\C{C}$. The same equivalence can be written for a symmetric (resp. braided) operads and a symmetric (resp. braided) monoidal functors.
\end{proposition}

\begin{proof}
	Let $A \in \C{C}$ be a $\C{O}$-algebra with the structure morphisms $\lambda^n_A\colon \C{O}_n\to \C{C}(A^{\odot n},A)$.  We define (symmetric, resp. braided) monoidal functor $\phi_A\colon cat{\C{O}} \to \C{C}$ by letting  $\phi_A([n])=A^{\odot n}$ on the set of objects.  Note that since $\C{C}$ is a strict (symmetric, resp. braided) monoidal category, we get $A^{\odot (n+m)} = A^{\odot n} \odot A^{\odot m}$ and we get a strict (symmetric, resp. braided) monoidal functor.  Since we have $cat{\C{O}}([n],[1])= \C{O}_n$, a  morphism  $f\colon [n]\to[1]$ in $cat{\C{O}} $ corresponds to an element $f \in \C{O}_n$ and  we have $\phi_A(f)=\lambda_A^n(f)\colon A^{\odot n}\to A$. This  extends to all morphisms in $cat{\C{O}}$ because  every morphism $[n]\to [m]$ in $cat{\C{O}}$ can be written as a sum of finite monoidal products of morphisms of the form $[n_i]\to [1]$ determined by $m$-compositions $n_1+\cdots+n_m$ of $n$. 
	
	Now, if $f\colon A\to B$ is a morphism of $\C{O}$-algebras, we must show that there is a natural transformation of the form $\phi_{f}\colon \phi_A\to \phi_B$ making diagrams of the form
	\[\xymatrix{
		A^{\odot n} 
		\ar[rr]^{\phi_A(\beta)} 
		\ar[d]_{f^{\odot n}}
		& & A^{\odot m}
		\ar[d]^{f^{\odot m}}\\
		B^{\odot n}
		\ar[rr]_{\phi_B(\beta)}
		& & B^{\odot m}
	}\]
	commutative for all $\beta\in cat\C{O}$. However, since the morphisms in $cat\C{O}$ is generated by morphisms of the form $[\ell]\to [1]$ via the monoidal product, in order to verify the natural transformation conditions, it is enough to consider diagrams of the form~\eqref{eq:OperadAlgebra} which all commute since we consider $\C{O}$-algebras. If we assume $\C{O}$ is symmetric (resp. braided) the category $cat\C{O}$ is symmetric (resp. braided).  This finishes the one side of the correspondence.
	
	On the other hand, assume we have a strict (symmetric, resp. braided) monoidal functor of the form $\phi\colon cat{\C{O}} \to \C{C}$. Then we have $\phi([n]) = A^{\odot n}$ where $A=\phi([1])$ for each $n\in \mathbb{N}$. Also note that since $\phi$ is a strict (symmetric, resp. braided) functor, we have ($S_n$-equivariant, resp. $B_n$-equivariant) maps $\phi_{n,1}\colon \C{O}_n \to \C{C}(A^{\odot n},A)$ since $cat{\C{O}}([n],[1]):=\C{O}_n$. The definition of a monoidal functor, together with the definition of composition of morphisms in $ cat{\C{O}}$ ensures that the collection $(\phi_{\cdot,[1]})_{n\in\B{N}}$ indeed defines a morphism of operads.  
	
	It is easy to see that these two constructions are mutual inverses on the set of objects, and morphisms.
\end{proof}

\subsection{A non-linear example}

Consider the skeletal category of finite sets with all set maps and with the Cartesian product as a strict symmetric monoidal category. The unit object is the set $[1]$. There is a unique operad $\Comm $ in this category which is defined as $\Comm _n = [1]$ for every $n\geq 1$. 

Let us first consider the \emph{non-unital} version where $\Comm _0 = \emptyset$.  The PROP associated with $\Comm $ then is the skeletal category of finite sets with surjections since
\begin{align*}
	\bigsqcup_{n_1+\cdots+n_m=n} & \left(\Comm _{n_1}\times\cdots\times\Comm _{n_m}\right)\times_{S_{n_1}\times\cdots\times S_{n_m}}\times S_n\\
	= & \bigsqcup_{n_1+\cdots+n_m=n}\underbrace{\left([1]\times\cdots\times [1]\right)}_{\text{$m$-times}}\times_{S_{n_1}\times\cdots\times S_{n_m}}\times S_n\\
	= & \bigsqcup_{n_1+\cdots+n_m=n} \left(S_{n_1}\times\cdots\times S_{n_m}\right)\backslash S_n
\end{align*}
and because surjective set maps of the form $f\colon [n]\to [m]$ with $n_i = f^{-1}(i)$ are in bijective correspondence with the set of right cosets of $S_{n_1}\times\cdots\times S_{n_m}$ in $S_n$ since both sets have the same size $\frac{n!}{n_1!n_2!\cdots n_m!}$ where $n=n_1+\cdots+n_m$ for each $n_i>0$.

On the other hand, if we forgo the symmetric structure, we get the opposite category of the skeletal category of well-ordered finite sets with order preserving surjections since
\begin{align*}
	\bigsqcup_{n_1+\cdots+n_m=n}\Comm _{n_1}\times\cdots\times\Comm _{n_m}
	%  = & \bigsqcup_{n_1+\cdots+n_m=n}\underbrace{[1]\times\cdots\times [1]}_{\text{$m$-times}}\\
	= & \bigsqcup_{n_1+\cdots+n_m=n} [1]
\end{align*}
and because the set of $m$-compositions of $n$ are in bijective correspondence with order preserving surjective maps of the form $f\colon [n]\to [m]$ where the bijection from functions to compositions is given by $f \mapsto (f^{-1}(1),\ldots,f^{-1}(m))$.  Note that since we assume each $|f^{-1}(i)| = n_i>0$, we have surjections.

Now, let us consider the \emph{unital} version where $\Comm _0 = [1]$, and where we allow $n_i=0$ in a composition of an integer in the indices we use in the unions. This means we now consider all maps $f\colon [n]\to [m]$ since we now allow $f^{-1}(i)$ to be empty for some $i\in [m]$. Then for the symmetric case we get the skeletal category of finite sets with all maps, and for the non-symmetric case we get the skeletal category of well-ordered finite sets with order preserving maps. The former is an unbased analogue of the opposite of Segal's category $\Gamma$ while the latter is the simplex category $\Delta$. 

In terms of parametrizing categories, the symmetric version of the non-unital $\Comm$ parametrizes commutative semi-groups while the non-symmetric version parametrizes all semi-groups. The unital version of the symmetric $\Comm$ then parametrizes commutative monoids while non-symmetric unital $\Comm$ parametrizes all monoids.

On the other hand, if we use the operad we defined Section~\ref{subsect:Assoc} and let $\Assoc_n = S_n$ for $n\geq 1$ the symmetric PROP associated with this operad is
\[ \bigsqcup_{n_1+\cdots+n_m=n} (S_{n_1}\times\cdots\times S_{n_m})\times_{S_{n_1}\times\cdots\times S_{n_m}} S_n\]
In the unital case we get the crossed simplicial group $\Delta \B{S}$~\cite{Fiedorowicz_1991}, while in the non-unital case we get $\Delta^+ \B{S}$ the subcategory of epimorphisms of $\Delta \B{S}$.  These categories also model semigroups and monoids, respectively.

As another variation, instead of the symmetric groups, we could have easily used the braid groups $\B{B} = \bigsqcup_{n\geq 1} B_n$, and we would have obtained crossed simplicial groups $\Delta\B{B}$ and $\Delta^+\B{B}$ for braided monoids and braided semigroups, respectively.

\bibliographystyle{siam}
\bibliography{phd_proposal}
	
\end{document}